\documentclass[british,smallextended]{svjour3}
\usepackage[T1]{fontenc}
\usepackage[latin9]{inputenc}
\usepackage[pdftex]{color}
\usepackage{babel}
\usepackage{mathrsfs}
\usepackage{url}
\usepackage{amsmath}
\usepackage{amssymb}
\usepackage[pdftex]{graphicx}
\usepackage[unicode=true,pdfusetitle,
 bookmarks=true,bookmarksnumbered=false,bookmarksopen=false,
 breaklinks=false,pdfborder={0 0 0},pdfborderstyle={},backref=false,colorlinks=true]
 {hyperref}
\hypersetup{
 linkcolor = blue, citecolor = blue, urlcolor  = blue}

\makeatletter

\newcommand{\noun}[1]{\textsc{#1}}
\providecommand{\tabularnewline}{\\}

\makeatother

\begin{document}
\title{Invariant spectral foliations with applications to model order reduction
and synthesis}
\titlerunning{Invariant Spectral Foliations}
\author{Robert Szalai}
\institute{University of Bristol, Department of Engineering Mathematics\\\email{r.szalai@bristol.ac.uk}}
\date{Last revision: 31 July 2020}
\maketitle
\begin{abstract}
The paper introduces a technique that decomposes the dynamics of a
nonlinear system about an equilibrium into low order components, which
then can be used to reconstruct the full dynamics. This is a nonlinear
analogue of linear modal analysis. The dynamics is decomposed using
Invariant Spectral Foliations (ISF), which is defined as the smoothest
invariant foliation about an equilibrium and hence unique under general
conditions. The conjugate dynamics of an ISF can be used as a reduced
order model. An ISF can be fitted to vibration data without carrying
out a model identification first. The theory is illustrated on a analytic
example and on free-vibration data of a clamped-clamped beam. 
\end{abstract}

\section{Introduction}

In this paper we highlight how invariant foliations \cite{hirsch1970,wiggins2013normally}
of dynamical systems can be used to derive reduced order models (ROM)
either from data or physical models. We consider dynamics about equilibria
only. We assume a deterministic process, such that future states of
the system are fully determined by initial conditions. An invariant
foliation is a decomposition of the state space into a family of manifolds,
called leaves, such that the dynamics brings each leaf into another
(see figure \ref{fig:IntroFoliation}). If a leaf is brought into
itself, then it is also an invariant manifold. A foliation is generally
characterised by its co-dimension, which equals the number of parameters
needed to describe the family of leaves so that it covers the state
space. The dynamics that maps one leaf of an invariant foliation into
another leaf has the same dimensionality as the co-dimension of the
foliation. We call this mapping the conjugate dynamics, which is lower
dimensional than the dynamics of the underlying system and therefore
suitable to be used as an ROM. Such ROM treats all initial conditions
within one leaf equivalent to each other and characterises the dynamics
of the whole system. In contrast, the conjugate dynamics (ROM) on
an invariant manifold captures the dynamics only on a low-dimensional
subset of the state space. The conjugate dynamics on an invariant
manifold, however describes the exact evolution of initial conditions
taken from the invariant manifold, while the conjugate dynamics on
an invariant foliation is imprecise about the evolution, it can only
tell which leaves a trajectory goes through. This ambiguity about
the state has some advantages: for all initial conditions there is
a leaf and a valid reduced dynamics. In contrast, when using invariant
manifolds, the initial condition must come from the invariant manifold
in order to have a valid prediction.

Multiple foliations can act as a coordinate system about the equilibrium.
When individual leaves from different foliations intersect in one
point, the dynamics can be fully reconstructed from the foliations.
Therefore invariant foliations are fully paralleled with linear modal
analysis of mechanical systems \cite{EwinsBook}: it allows both the
decomposition of the system and the reconstruction of the full dynamics.
To reconstruct the dynamics one needs to find intersection points
of leaves from different foliations which is more complicated than
adding vibration modes of linear system. However, such composability
is not at all possible with invariant manifolds or any other nonlinear
normal mode (NNM) definition \cite{KERSCHEN2009170,ShawPierre,Haller2016,Rosenberg66}.
Therefore, an invariant foliation seems to be the closest nonlinear
alternative of linear modal analysis. The concept of composition is
illustrated in figure \ref{fig:IntroFoliation}.

\begin{figure}
\begin{centering}
\includegraphics[scale=0.75]{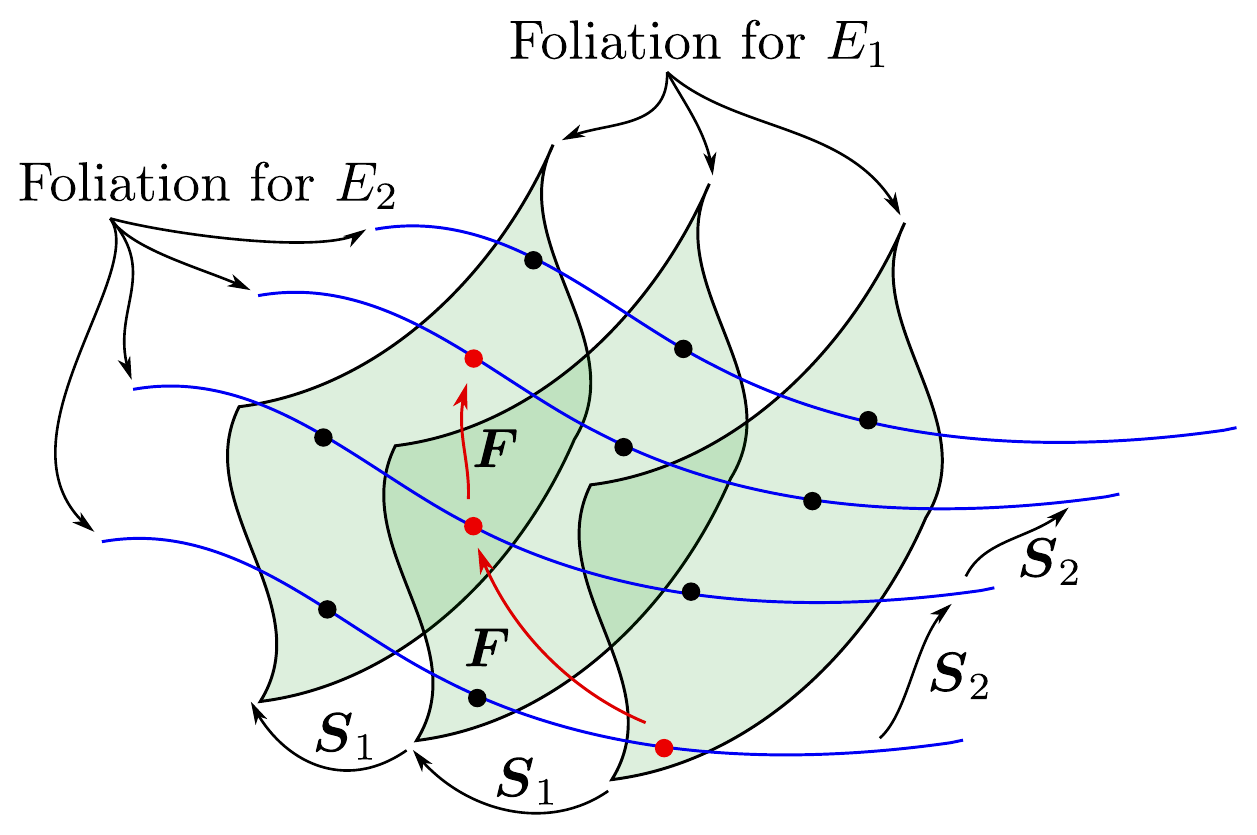}
\par\end{centering}
\caption{\label{fig:IntroFoliation}Two foliations act as a coordinate system.
An initial condition (red dots) is mapped forward by $\boldsymbol{F}$,
however each leaf of a foliation is brought forward by the lower dimensional
maps $\boldsymbol{S}_{1}$ and $\boldsymbol{S}_{2}$. Due to invariance
of the foliation, the full trajectory can be reconstructed from the
two maps $\boldsymbol{S}_{1}$ and $\boldsymbol{S}_{2}$ and the leaves
of the foliations.}
\end{figure}

Invariant foliations can be directly fitted to time-series data, because
the foliation acts as a projection, much like linear modes. This allows
for another parallel to be drawn with modal testing \cite{EwinsBook},
which identifies linear vibration modes from data. Direct fitting
of the manifold invariance equation to data is not possible, because
the likelihood of data points falling onto the manifold is zero. Instead,
in \cite{Szalai20160759} a two-step process was used to find invariant
manifolds in vibration data. First a high-dimensional black-box model
was identified and then the invariant manifold was extracted. In contrast,
a foliation covers all of the phase space where the data lives, hence
all available data can be used for fitting. Moreover, a leaf of a
foliation that is mapped into itself or equivalently in our case contains
the equilibrium, is an invariant manifold. Therefore finding two complementary
invariant foliations, one transversal to an invariant manifold, another
containing the invariant manifold as a leaf can substitute for calculating
the invariant manifold and ROM.

The condition for uniqueness of invariant foliations are different
from invariant manifolds. Only invariant manifolds about equilibria
that are sufficiently smooth are unique. Unique invariant manifolds
about equilibria, periodic or quasi-periodic orbits are called spectral
submanifolds (SSM) \cite{Haller2016}. The theory behind SSMs was
mainly developed in \cite{delaLlave1997}, generalised to infinite
dimensions in \cite{CabreLlave2003} and applied to mechanical systems
in \cite{Haller2016}. For an SSM to exist, non-resonance conditions
need to be satisfied and the dynamics must be smoother than a so-called
spectral quotient, which is calculated from the eigenvalues of the
Jacobian about an equilibrium. For an SSM to be interesting, it must
contain the slowest dynamics, so that it captures long-term behaviour,
rather than just transients (see R2 in \cite{HallerExact2016}). It
turns out that the spectral quotient of such an SSM is also the highest
and therefore the SSM requires the highest order of smoothness to
be unique. While the concept of smoothness is theoretically well-understood,
it is almost impossible to quantify numerically or determine from
data. This is one of the reasons why it is challenging to calculate
SSMs numerically (in contrast to series expansion \cite{PONSIOEN2018269})
in a reproducible manner. Invariant foliations, as explained below,
also need to satisfy non-resonance conditions to exist and be sufficiently
smooth to be unique. We call a unique invariant foliation tangential
to an invariant linear subspace about an equilibrium an invariant
spectral foliation (ISF). In contrast to SSMs, ISFs that capture the
long term dynamics require the lowest order of smoothness among all
ISFs. This however does not mean that the smoothness requirements
of SSMs can be circumvented by extracting an SSM as the leaf of the
ISF going through the origin. In order to obtain the slowest SSM,
one would need to calculate the fastest ISF, both of which require
the same high-order of smoothness for uniqueness.

The existing literature on invariant foliations is rich and difficult
to summarise without distracting too much from the purpose of the
paper (see e.g.~\cite{hirsch1970,BatesFoliations2000,ShubFoliation2012}).
However, the setting used here is also different from most of the
literature in that we are not dealing with stable or unstable fibres
and hyperbolicity is not an important aspect either. The closest results
in the literature are the remarks of de la Llave in section 7.3 of
\cite{delaLlave1997} and section 2 of \cite{CabreLlave2003}, that
generalise the parametrisation method to foliations. 

The plan of the paper is as follows. We start with introducing invariant
foliations and describe their properties. We then state and prove
theorems for the existence and uniqueness of ISFs both for discrete-time
systems and vector fields. Finally, we describe a simple method that
allows finding ISFs from time-series data, which is then tested on
a simple example alongside with two other approaches. 

\section{Invariant foliations}

Consider a dynamical system that is defined by the $C^{r}$ map $\boldsymbol{F}:\mathbb{R}^{n}\to\mathbb{R}^{n}$.
A trajectory of the dynamical system is obtained by recursively applying
$\boldsymbol{F}$ to the initial condition $\boldsymbol{x}_{0}$,
such that successive points along a trajectory are generated by
\begin{equation}
\boldsymbol{x}_{k+1}=\boldsymbol{F}\left(\boldsymbol{x}_{k}\right),\quad k=0,1,\ldots.\label{eq:MapTrajectory}
\end{equation}
We assume that the origin is a fixed point, that is $\boldsymbol{F}\left(\boldsymbol{0}\right)=\boldsymbol{0}$
and the Jacobian at the origin, $\boldsymbol{A}=D\boldsymbol{F}\left(\boldsymbol{0}\right)$,
is semisimple. The eigenvalues of $\boldsymbol{A}$ are denoted by
$\mu_{i}$, $i=1,\ldots,n$ and we have a full set of left and right
eigenvectors, $\boldsymbol{v}_{i}^{\star}$ and $\boldsymbol{v}_{i}$,
that satisfy $\boldsymbol{v}_{i}^{\star}\boldsymbol{A}=\mu_{i}\boldsymbol{v}_{i}^{\star}$
and $\boldsymbol{A}\boldsymbol{v}_{i}=\mu_{i}\boldsymbol{v}_{i}$,
respectively. For convenience we also assume that the eigenvectors
are scaled such that $\boldsymbol{v}_{i}^{\star}\boldsymbol{v}_{i}=1$.
Let us denote the linear subspace spanned by the first $\nu$ eigenvectors
as $E=\mathrm{span}\left\{ \boldsymbol{v}_{1},\ldots,\boldsymbol{v}_{\nu}\right\} $
and the dual subspace $E^{\star}=\mathrm{span}\left\{ \boldsymbol{v}_{1}^{\star},\ldots,\boldsymbol{v}_{\nu}^{\star}\right\} $.
Finally, we assume that $\boldsymbol{A}$ is a contraction, that is,
$\left|\mu_{i}\right|<1,\,\forall i=1,\ldots,n$.

We are interested in how codimension-$\nu$ sets about the origin
are brought into each other by $\boldsymbol{F}$. The manifold of
sets is parametrised by an $\nu$-dimensional parameter $\boldsymbol{z}\in\mathbb{R}^{\nu}$
and a single set at point $\boldsymbol{z}$ is denoted by $\mathcal{L}_{\boldsymbol{z}}$.
We assume that each $\mathcal{L}_{\boldsymbol{z}}$ is a differentiable
manifold and $\mathcal{L}_{\boldsymbol{z}}$ and $\mathcal{L}_{\tilde{\boldsymbol{z}}}$
are disjointed if $\boldsymbol{z}\neq\tilde{\boldsymbol{z}}$. In
technical terms this is called a codimension-$\nu$ \emph{foliation}
of $\mathbb{R}^{n}$ \cite{Lawson1974} and each $\mathcal{L}_{\boldsymbol{z}}$
is a \emph{leaf}. The foliation is a collection of leaves, that is
$\mathcal{F}=\left\{ \mathcal{L}_{\boldsymbol{z}}:\boldsymbol{z}\in\mathbb{R}^{\nu}\right\} $.

A foliation $\mathcal{F}$ is invariant under $\boldsymbol{F}$ if
there is a map $\boldsymbol{S}:\mathbb{R}^{\nu}\to\mathbb{R}^{\nu}$,
which brings the leaves into each other in the same way as the high-dimensional
dynamics, that is 
\begin{equation}
\boldsymbol{F}\left(\mathcal{L}_{\boldsymbol{z}}\right)\subset\mathcal{L}_{\boldsymbol{S}\left(\boldsymbol{z}\right)}.\label{eq:LeafInvariance}
\end{equation}
A foliation can be represented by a function $\boldsymbol{U}:\mathbb{R}^{n}\to\mathbb{R}^{\nu}$,
called \emph{submersion}, such that a leaf is the pre-image of the
parameter $\boldsymbol{z}$ under the submersion $\boldsymbol{U}$,
that is, 
\begin{equation}
\mathcal{L}_{\boldsymbol{z}}=\left\{ \boldsymbol{x}\in\mathbb{R}^{n}:\boldsymbol{U}\left(\boldsymbol{x}\right)=\boldsymbol{z}\right\} .\label{eq:LeafDefinition}
\end{equation}
Using definition (\ref{eq:LeafDefinition}), we find that the inclusion
(\ref{eq:LeafInvariance}) translates into an algebraic equation for
the submersion $\boldsymbol{U}$, 
\begin{equation}
\boldsymbol{U}\left(\boldsymbol{F}\left(\boldsymbol{x}\right)\right)=\boldsymbol{S}\left(\boldsymbol{U}\left(\boldsymbol{x}\right)\right),\label{eq:adjInvariance}
\end{equation}
which is called the \emph{invariance equation}. Similar to invariant
manifolds, we require a tangency condition to a linear subspace. To
consider the dynamics corresponding to the linear subspace $E^{\star}$,
we require that
\begin{equation}
\boldsymbol{U}\left(\boldsymbol{0}\right)=\boldsymbol{0}\;\text{and}\;\mathrm{span}\,D\boldsymbol{U}\left(\boldsymbol{0}\right)=E^{\star},\label{eq:adjTangency}
\end{equation}
which means that $D\boldsymbol{U}\left(\boldsymbol{0}\right)$ is
a set of $\nu$ linearly independent row vectors from the dual space
of $\mathbb{R}^{n}$ (row vectors) spanning the whole space $E^{\star}$.

Figure \ref{fig:InvariantFibres} shows the geometry of an invariant
foliation. Each leaf is mapped into another, in particular, the green
solid line representing $\mathcal{L}_{\boldsymbol{z}}$ is mapped
into the red solid line by $\boldsymbol{F}$. A leaf that corresponds
to a fixed point of $\boldsymbol{S}$ is an invariant manifold as
it is mapped into itself. In particular, we have $\boldsymbol{S}\left(\boldsymbol{0}\right)=\boldsymbol{0}$,
hence $\mathcal{L}_{\boldsymbol{0}}$ is an invariant manifold.
\begin{figure}
\begin{centering}
\includegraphics[width=0.5\linewidth]{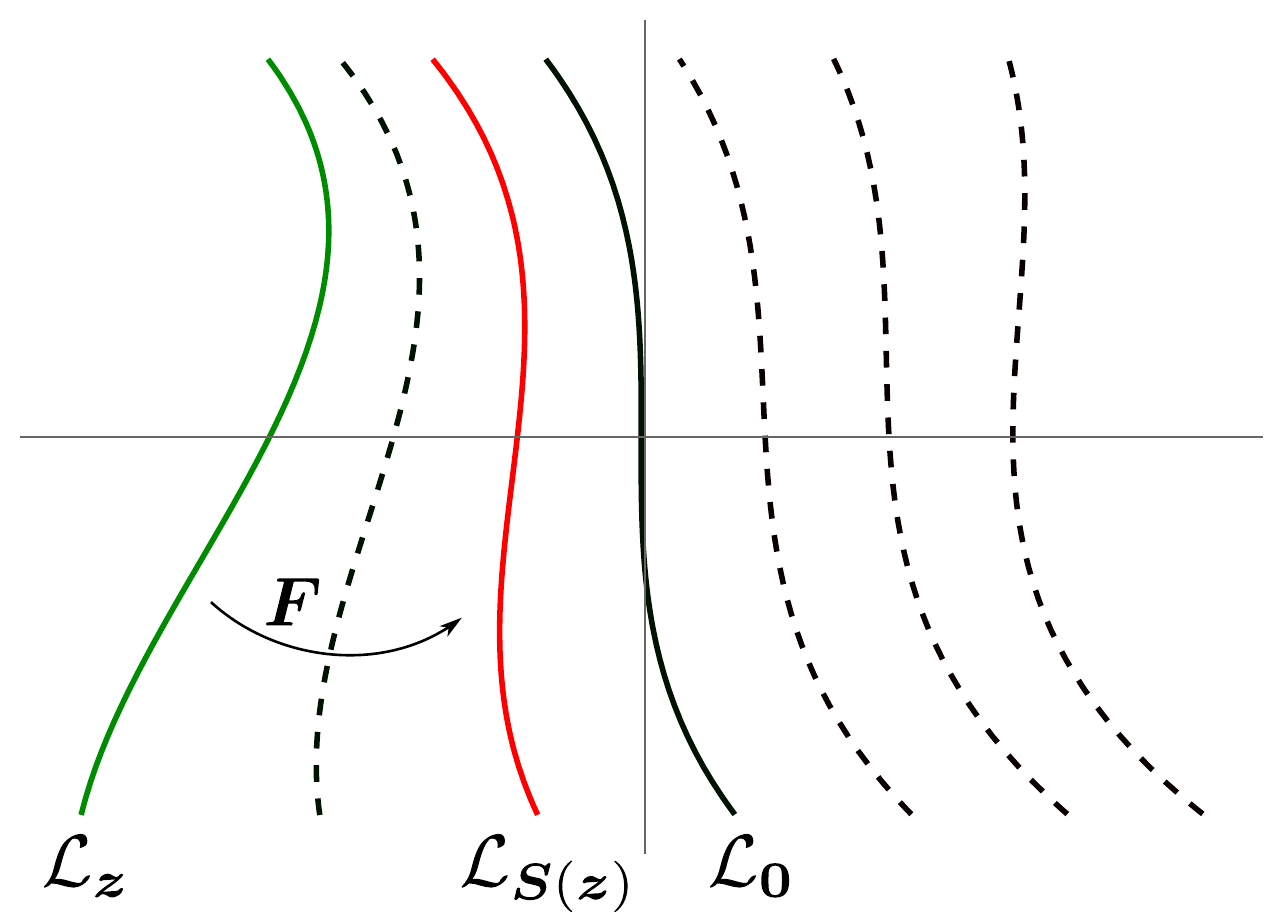}
\par\end{centering}
\caption{\label{fig:InvariantFibres}Invariant foliation. The leaf $\mathcal{L}_{\boldsymbol{z}}$
(green solid line) is mapped onto $\mathcal{L}_{\boldsymbol{S}\left(\boldsymbol{z}\right)}$
(red solid line) by $\boldsymbol{F}$. Leaf $\mathcal{L}_{\boldsymbol{0}}$
(black solid line) is an invariant manifold, because it contains the
origin and it is mapped onto itself by $\boldsymbol{F}$. Dashed lines
are other leaves.}
\end{figure}

The solution of the invariance equation (\ref{eq:adjInvariance})
with the tangency condition (\ref{eq:adjTangency}) is not unique
for a number of reasons. Firstly, assuming that there exist a pair
of functions $\boldsymbol{U}$ and $\boldsymbol{S}$ satisfying (\ref{eq:adjInvariance})
and (\ref{eq:adjTangency}) a large class of diffeomorphism $\boldsymbol{\Phi}:\mathbb{R}^{\nu}\to\mathbb{R}^{\nu}$
can be used, such that $\tilde{\boldsymbol{U}}=\boldsymbol{\Phi}\circ\boldsymbol{U}$
and $\tilde{\boldsymbol{S}}=\boldsymbol{\Phi}\circ\boldsymbol{U}\circ\boldsymbol{\Phi}^{-1}$
are also solutions of (\ref{eq:adjInvariance}) and (\ref{eq:adjTangency}).
However if two pairs of solutions of (\ref{eq:adjInvariance}) and
(\ref{eq:adjTangency}) are conjugate through a diffeomorphism $\boldsymbol{\Phi}$,
they represent the same invariant foliation $\mathcal{F}$. The kind
of non-uniqueness that is problematic when multiple solutions of (\ref{eq:adjInvariance})
and (\ref{eq:adjTangency}) are not conjugate and do not represent
the same foliation. To fix possible non-uniqueness, we impose extra
smoothness conditions on the submersion $\boldsymbol{U}$ in addition
to being differentiable and tangent to $E^{\star}$. We then call
the smoothest and unique foliation the \emph{invariant spectral foliation}
(ISF) corresponding to the linear subspace $E$.

\subsection{Vector fields}

In many applications the dynamics is defined by a vector field. Here
we recall that there is a one-to-one relationship between invariant
foliations of maps and vector fields \cite{Arnold1992ordinary}. Consider
the vector field $\boldsymbol{\dot{x}}=\boldsymbol{G}\left(\boldsymbol{x}\right)$,
which has a fundamental solution $\boldsymbol{\Phi}_{t}\left(\boldsymbol{x}\right)$,
such that 
\[
\frac{d}{dt}\boldsymbol{\Phi}_{t}\left(\boldsymbol{x}\right)=\boldsymbol{G}\left(\boldsymbol{\Phi}_{t}\left(\boldsymbol{x}\right)\right),\quad\boldsymbol{\Phi}_{0}\left(\boldsymbol{x}\right)=\boldsymbol{x}.
\]
Here $\boldsymbol{\Phi}_{t}$ is a one-parameter group, because $\boldsymbol{\Phi}_{t}\left(\boldsymbol{\Phi}_{s}\left(\boldsymbol{x}\right)\right)=\boldsymbol{\Phi}_{t+s}\left(\boldsymbol{x}\right)$
and $\boldsymbol{\Phi}_{0}\left(\boldsymbol{x}\right)=\boldsymbol{x}$.
If $\boldsymbol{G}$ is $C^{r}$ smooth then so is $\boldsymbol{\Phi}_{t}$.
We can now define the map $\boldsymbol{F}\left(\boldsymbol{x}\right)=\boldsymbol{\Phi}_{t}\left(\boldsymbol{x}\right)$,
which brings the invariance equation (\ref{eq:adjInvariance}) into
\begin{equation}
\boldsymbol{U}\left(\boldsymbol{\Phi}_{t}\left(\boldsymbol{x}\right)\right)=\boldsymbol{S}_{t}\left(\boldsymbol{U}\left(\boldsymbol{x}\right)\right).\label{eq:VFMapInvariance}
\end{equation}
The conjugate dynamics $\boldsymbol{S}$ must also be a one-parameter
group with $\boldsymbol{S}_{t+s}\left(\boldsymbol{x}\right)=\boldsymbol{S}_{t}\left(\boldsymbol{S}_{s}\left(\boldsymbol{x}\right)\right)$
and $\boldsymbol{S}_{0}\left(\boldsymbol{x}\right)=\boldsymbol{x}$
in order to satisfy the invariance equation, that is, 
\begin{align*}
\boldsymbol{U}\left(\boldsymbol{\Phi}_{t}\left(\boldsymbol{\Phi}_{s}\left(\boldsymbol{x}\right)\right)\right) & =\boldsymbol{S}_{t}\left(\boldsymbol{U}\left(\boldsymbol{\Phi}_{s}\left(\boldsymbol{x}\right)\right)\right)\\
\boldsymbol{U}\left(\boldsymbol{\Phi}_{t+s}\left(\boldsymbol{x}\right)\right) & =\boldsymbol{S}_{t}\left(\boldsymbol{S}_{s}\left(\boldsymbol{U}\left(\boldsymbol{x}\right)\right)\right)\\
\boldsymbol{U}\left(\boldsymbol{\Phi}_{t+s}\left(\boldsymbol{x}\right)\right) & =\boldsymbol{S}_{t+s}\left(\boldsymbol{U}\left(\boldsymbol{x}\right)\right).
\end{align*}
The infinitesimal generator of the group $\boldsymbol{S}$ is denoted
by $\boldsymbol{R}$, such that $\frac{d}{dt}\boldsymbol{S}_{t}\left(\boldsymbol{x}\right)=\boldsymbol{R}\left(\boldsymbol{S}_{t}\left(\boldsymbol{x}\right)\right)$.
On the other hand $\boldsymbol{U}$ must be independent of time, if
it is to define an invariant foliation. We now take the derivative
of the invariance equation (\ref{eq:VFMapInvariance}) with respect
to time and find
\begin{equation}
D\boldsymbol{U}\left(\boldsymbol{\Phi}_{t}\left(\boldsymbol{x}\right)\right)\boldsymbol{G}\left(\boldsymbol{\Phi}_{t}\left(\boldsymbol{x}\right)\right)=\boldsymbol{R}\left(\boldsymbol{S}_{t}\left(\boldsymbol{U}\left(\boldsymbol{x}\right)\right)\right).\label{eq:VFDeriInvariance}
\end{equation}
Setting $t=0$ in equation (\ref{eq:VFDeriInvariance}), we get the
invariance equation for vector fields in the form of
\begin{equation}
D\boldsymbol{U}\left(\boldsymbol{x}\right)\boldsymbol{G}\left(\boldsymbol{x}\right)=\boldsymbol{R}\left(\boldsymbol{U}\left(\boldsymbol{x}\right)\right).\label{eq:VFadjInvariance}
\end{equation}
The next example, which aims to illustrate non-uniqueness of foliations,
also shows that occasionally, it is easier to find an invariant foliation
using (\ref{eq:VFadjInvariance}) than using (\ref{eq:adjInvariance}).

\subsection{\label{subsec:UniqueExample}Example: smoothness and uniqueness of
foliations}

Let us consider the discrete-time map
\[
\begin{pmatrix}x_{k+1}\\
y_{k+1}
\end{pmatrix}=\begin{pmatrix}\mathrm{e}^{-\lambda}x_{k}\\
\mathrm{e}^{-\mu}y_{k}
\end{pmatrix},\,\lambda>0,\mu>0
\]
for which we can find an equivalent vector field in the form of
\begin{equation}
\begin{pmatrix}\dot{x}\\
\dot{y}
\end{pmatrix}=\begin{pmatrix}-\lambda x\\
-\mu y
\end{pmatrix},\label{eq:DELinExample}
\end{equation}
such that $x_{k}=x\left(k\right)$ and $y_{k}=y\left(k\right)$. The
solutions of system (\ref{eq:DELinExample}) lie on the curves $y\left(x\right)=c\mathrm{e}^{x\mu/\lambda}$,
$c\in\mathbb{R}$, $x\ge0$, as we only consider the right half-plane.
The invariance equation (\ref{eq:VFadjInvariance}), when (\ref{eq:DELinExample})
is substituted, becomes 
\[
-\lambda xD_{1}u\left(x,y\right)-\mu yD_{2}u\left(x,y\right)=r\left(u\left(x,y\right)\right),
\]
where $r$ describes the dynamics among the leaves of the invariant
foliation. Here, we have used non-bold, lower-case letters to represent
$\boldsymbol{U}=u$ and $\boldsymbol{R}=r$, because they assume scalar
values. Without restricting generality, we prescribe the parametrisation
of the foliation by setting $u\left(x,0\right)=x$, which implies
that $r\left(x\right)=-\lambda x$. We note that any other parametrisation
for which $\hat{u}\left(x,0\right)$ is a strictly monotonous (invertible)
and smooth function of $x$ can be brought into the special parametrisation
that we have just chosen, that is $u\left(x,y\right)=\hat{u}\left(\hat{u}^{-1}\left(x,0\right),y\right)$.
Using this parametrisation, the invariance equation then simplifies
to
\begin{equation}
-\lambda xD_{1}u\left(x,y\right)-\mu yD_{2}u\left(x,y\right)=-\lambda u\left(x,y\right).\label{eq:SimpleExInvariance}
\end{equation}
The solution of (\ref{eq:SimpleExInvariance}) is sought in the form
of $u\left(x,y\right)=xw\left(x,y\right)$, where $w$ has to satisfy
the somewhat simpler equation
\[
-\lambda xD_{1}w\left(x,y\right)-\mu yD_{2}w\left(x,y\right)=0.
\]
Using the method of characteristics and assuming the boundary condition
$w\left(1,y\right)=f\left(y\right)$ gives the general solution

\begin{equation}
u\left(x,y\right)=xf\left(x^{-\mu/\lambda}y\right),\label{eq:SimpleExU}
\end{equation}
where $f$ is an unknown, continuously differentiable function with
$f\left(0\right)=1$ due to the constraint on the parametrisation.

We now assume that $f$ is $m$ times differentiable, such that $f\left(x\right)=1+\sum_{k=1}^{m}a_{k}x^{k}+\mathcal{O}\left(x^{m+1}\right)$
and that $\lambda/\mu>m$. In this case the $k$-th order term of
$f$ leads to an order $1+k\left(1-\mu/\lambda\right)>1+k-k/m$ term
in $u$, that are continuously differentiable if and only if $k\le m$.
This implies that if $m<\lambda/\mu\le m+1$, function $f$ must assume
the form 
\[
f\left(x\right)=1+\sum_{k=1}^{m}a_{k}x^{k}
\]
for $u$ to be once differentiable. This means that the foliation
is non-unique and has $m$ free parameters. Repeating the same argument
but stipulating that the foliation must be $m$-times continuously
differentiable we find that $f=1$, which has no parameters and therefore
the invariant foliation becomes unique. Indeed, after differentiating
(\ref{eq:SimpleExU}) $m$-times, a $k$-th order term in $f$ results
in an order $1+k\left(1-\mu/\lambda\right)-m>1+k-m-k/m$ term in $D^{m}u$,
hence none of the terms apart from the constant one will lead to an
$m$-times differentiable $u$, and the only solution is $f=1$ meaning
that the unique submersion is $u\left(x,y\right)=x$. We also note
that in this example, the ISF is as smooth as the vector field, that
is analytic.

The $x$ variable represents the slow dynamics if $\lambda<\mu$.
In this case $\lambda/\mu<1$, which means that a differentiable foliation
is already unique. The result of this section is graphically illustrated
in figure \ref{fig:UniquenessExample}.

\begin{figure}
\begin{centering}
\includegraphics[width=0.33\linewidth]{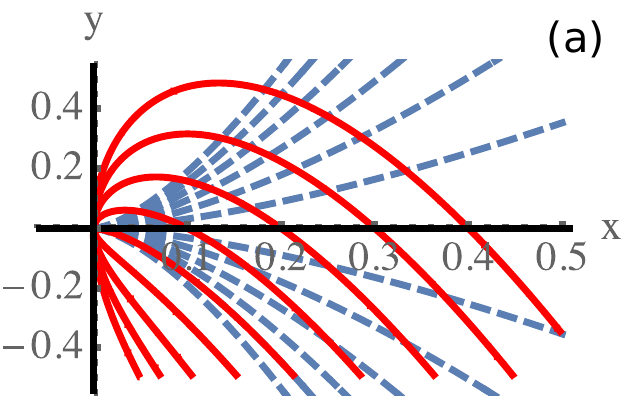}\includegraphics[width=0.33\linewidth]{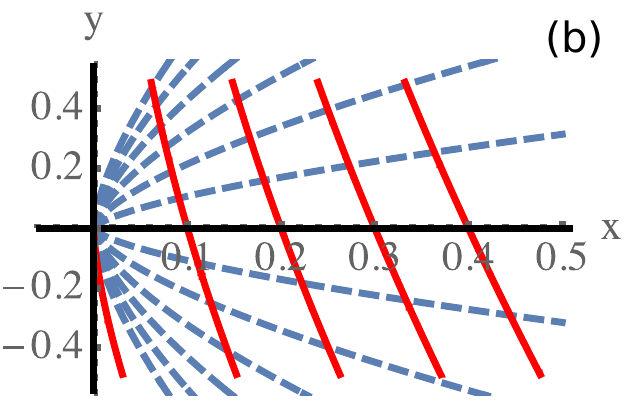}\includegraphics[width=0.33\linewidth]{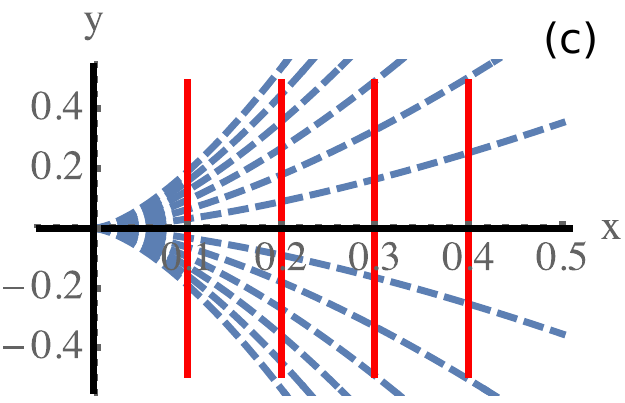}
\par\end{centering}
\caption{\label{fig:UniquenessExample}Uniqueness of foliations. Dashed blue
lines are trajectories of (\ref{eq:DELinExample}), red continuous
lines are the contours of $u\left(x,y\right)$ and represent the leaves
of the foliation. (a) $f=1+x/5$, $\lambda=2$, $\mu=3$, the resulting
$u$ does not define a differentiable foliation; (b) $f=1+x/5$, $\lambda=3$,
$\mu=2$, the foliation is once differentiable but not unique; (c)
$f=1$, $\lambda=2$, $\mu=3$ leads to the unique and differentiable
foliation.}

\end{figure}

\section{Existence and uniqueness of invariant foliations}

In this section we generalise the findings from the example in section
\ref{subsec:UniqueExample} and provide a sufficient condition for
the existence of a unique invariant foliation, i.e., an ISF. We start
with a definition.
\begin{definition}
The number
\[
\beth_{E^{\star}}=\frac{\min_{k=1\ldots\nu}\log\left|\mu_{k}\right|}{\max_{k=1\ldots n}\log\left|\mu_{k}\right|}
\]
is called the ISF spectral quotient of the linear subspace $E^{\star}$
about the origin.
\end{definition}
\begin{theorem}
\label{thm:MapFoliation}Assume that $\max_{k=1\ldots n}\left|\mu_{k}\right|<1$
and that there exists an integer $2\le\sigma\le r$, such that $\beth_{E^{\star}}<\sigma$.
Further assume that
\begin{equation}
\prod_{k=1}^{n}\mu_{k}^{m_{k}}\neq\mu_{j},\;j=1,\ldots,\nu\label{eq:MapNonResonance}
\end{equation}
for all integer $m_{k}\ge0$, $1\le k\le n$ with at least one $m_{l}\neq0$,
$\nu+1\le l\le n$ and with $2\le\sum_{k=0}^{n}m_{k}\le\sigma-1$.

Then the following are true:
\begin{enumerate}
\item In a sufficiently small neighbourhood of the origin there exists an
invariant foliation $\mathcal{F}$ tangent to the invariant linear
subspace $E^{\star}$ of the $C^{r}$ map $\boldsymbol{F}$. The foliation
$\mathcal{F}$ is unique among the $\sigma$-times differentiable
foliations and it is also $C^{r}$ smooth.
\item The conjugate dynamics of the invariant foliation $\mathcal{F}$,
given by the map $\boldsymbol{S}$ in equation (\ref{eq:adjInvariance})
can be represented by a polynomial of order $\sigma-1$. In its simplest
form $\boldsymbol{S}$ must include terms $\prod_{k=1}^{\nu}z_{k}^{m_{k}}$
in dimension $j$ for which
\begin{equation}
\prod_{k=1}^{\nu}\mu_{k}^{m_{k}}=\mu_{j},\;j=1,\ldots,\nu,\;2\le\sum_{k=0}^{\nu}m_{k}\le\sigma-1.\label{eq:MapInternalResonance}
\end{equation}
\end{enumerate}
\end{theorem}
\begin{proof}
The proof is carried out in appendix \ref{sec:ProofThm1}.
\end{proof}

\begin{remark}
\label{rem:ThmSuboptimal}From the conditions of theorem \ref{thm:MapFoliation}
it follows that $\beth_{E^{\star}}\ge1$. In case $\beth_{E^{\star}}=1$,
an invariant foliations is unique if it is at least twice differentiable.
This is however not a necessary condition, because for example (\ref{eq:DELinExample})
of section \ref{subsec:UniqueExample} once differentiability already
implied uniqueness.
\end{remark}
\begin{remark}
In contrast to SSMs, $D\boldsymbol{F}\left(\boldsymbol{0}\right)$
does not have to be invertible, only $D\boldsymbol{S}\left(\boldsymbol{0}\right)$
has to be invertible, that is, $\mu_{k}\neq0$ for $k=1\ldots\nu$.
If one were to extend the theory to Banach spaces, where typical dynamics
is not invertible (e.g.~delay equations, analytic semigroups, etc),
the lack of requirement on invertibility would allow wider application
of ISFs than SSMs. For example, in \cite{Kogelbauer2018}, the requirement
of invertibility demanded a special choice of damping added to a beam
model for an SSM to exist.
\end{remark}
\begin{remark}
For simplicity of presentation, the paper focusses on equilibria.
However, theorem \ref{thm:MapFoliation} is also applicable to periodic
orbits of vector fields, both autonomous and periodically forced,
where $\boldsymbol{F}$ is the Poincar\'e map associated with the
periodic orbit.
\end{remark}
The proof of theorem \ref{thm:MapFoliation} follows the same lines
that Cabr\'e et al.~\cite{CabreLlave2003} employ. First, a low-order
series expansion is carried out avoiding possible resonances. For
higher order terms, where no resonance is possible, Banach's contraction
mapping principle is applied to find a unique correction. Since the
series expansion allows a number of free parameters, we also show
that the choice of these parameters does not influence the geometry
of the foliation, only its parametrisation. This results in a unique
foliation. Differentiability follows from choosing $\sigma=r$.

Theorem \ref{thm:MapFoliation} also applies to $C^{r}$ vector fields
$\dot{\boldsymbol{x}}=\boldsymbol{G}\left(\boldsymbol{x}\right)$.
Again, we assume that the origin is the equilibrium, that is $\boldsymbol{G}\left(\boldsymbol{0}\right)=\boldsymbol{0}$
and that the Jacobian $\boldsymbol{B}=D\boldsymbol{G}\left(\boldsymbol{0}\right)$
is semisimple. The eigenvalues of $\boldsymbol{B}$ are denoted by
$\lambda_{i}$, $i=1,\ldots,n$ and we have a full set of left and
right eigenvectors, $\boldsymbol{v}_{i}^{\star}$ and $\boldsymbol{v}_{i}$,
that satisfy $\boldsymbol{v}_{i}^{\star}\boldsymbol{B}=\lambda_{i}\boldsymbol{v}_{i}^{\star}$
and $\boldsymbol{B}\boldsymbol{v}_{i}=\lambda_{i}\boldsymbol{v}_{i}$,
respectively. The invariant linear subspaces are defined as before:
$E=\mathrm{span}\left\{ \boldsymbol{v}_{1},\ldots,\boldsymbol{v}_{\nu}\right\} $
and $E^{\star}=\mathrm{span}\left\{ \boldsymbol{v}_{1}^{\star},\ldots,\boldsymbol{v}_{\nu}^{\star}\right\} $.

Using the spectral mapping theorem for the equivalence $\boldsymbol{A}=\exp\boldsymbol{B}\tau$,
$\tau>0$, we find that the ISF spectral quotient for a vector field
is 
\[
\beth_{E^{\star}}=\frac{\min_{k=1\ldots\nu}\Re\lambda_{k}}{\max_{k=1\ldots n}\Re\lambda_{k}}.
\]
Due to the equivalence between discrete-time dynamics and vector fields,
the following corollary is a direct consequence of theorem \ref{thm:MapFoliation}.
\begin{corollary}
\label{cor:VFFoliation}Assume that $\max_{k=1\ldots n}\Re\lambda_{k}<0$
and that there exists an integer $2\le\sigma\le r$, such that $\beth_{E^{\star}}<\sigma$.
Further assume that
\begin{equation}
\sum_{k=1}^{n}m_{k}\lambda_{k}\neq\lambda_{j},\;j=1,\ldots,\nu\label{eq:VFNonResonance}
\end{equation}
for all integer $m_{k}\ge0$, $1\le k\le n$ with at least one $m_{l}\neq0$,
$\nu+1\le l\le n$ and with $2\le\sum_{k=0}^{n}m_{k}\le\sigma-1$.

Then the following are true:
\begin{enumerate}
\item In a sufficiently small neighbourhood of the origin there exists an
invariant foliation $\mathcal{F}$ tangent to the invariant linear
subspace $E^{\star}$ of the $C^{r}$ vector field $\boldsymbol{G}$.
The foliation $\mathcal{F}$ is unique among the $\sigma$-times differentiable
foliations and it is also $C^{r}$ smooth.
\item The conjugate dynamics of the invariant foliation $\mathcal{F}$,
given by the vector field $\boldsymbol{R}$ in equation (\ref{eq:VFadjInvariance})
can be represented as a polynomial of order $\sigma-1$. In its simplest
form $\boldsymbol{R}$ must include terms $\prod_{k=1}^{\nu}z_{k}^{m_{k}}$
in dimension $j$ for which
\begin{equation}
\sum_{k=1}^{\nu}m_{k}\lambda_{k}=\lambda_{j},\;j=1,\ldots,\nu,\;2\le\sum_{k=0}^{\nu}m_{k}\le\sigma-1.\label{eq:VFInternalResonance}
\end{equation}
\end{enumerate}
\end{corollary}
\begin{definition}
We say that the invariant foliation has an internal resonance if there
exist non-negative integers $m_{k}$, $k=1,\ldots,\nu$ for which
(\ref{eq:MapInternalResonance}) (or (\ref{eq:VFInternalResonance})
for vector fields) holds.
\end{definition}

\section{Fitting a codimension-two ISF to data}

In this section we outline how to find the submersion $\boldsymbol{U}$
and conjugate map $\boldsymbol{S}$ from a time-series without identifying
map $\boldsymbol{F}$ first. The procedure is based on the proof of
theorem \ref{thm:MapFoliation} in appendix \ref{sec:ProofThm1},
which uses normalising conditions to find a unique solution of the
invariance equation (\ref{eq:adjInvariance}). Here, we make the normalising
conditions applicable to a wide class of representations of the submersion
$\boldsymbol{U}$ and conjugate map $\boldsymbol{S}$ and not just
polynomials.

\subsection{\label{subsec:ReducedDynamics}Normalising the solution of the invariance
equation}

The construction of the fitting process is centred around near internal
resonances. We assume a pair of complex conjugate eigenvalues $\mu_{1}=\overline{\mu}_{2}$
corresponding to the invariant linear subspace $E^{\star}$. If the
dynamics is slow on the ISF compared to the rest of the system, we
have $\left|\mu_{1}\right|\approx1$, which implies that 
\begin{equation}
\left.\begin{array}{rl}
\mu_{1} & \approx\mu_{1}^{p+1}\mu_{2}^{p}\\
\mu_{2} & \approx\mu_{1}^{p}\mu_{2}^{p+1}
\end{array}\right\} \label{eq:FITinternalRes}
\end{equation}
for integers $1\le p<\sigma$. According to theorem \ref{thm:MapFoliation},
we can choose to represent the dynamics on the ISF in complex coordinates
as 
\begin{equation}
\tilde{\boldsymbol{S}}\left(z,\overline{z}\right)=\begin{pmatrix}\begin{array}{l}
\mu_{1}z+\sum_{p=1}^{\left\lfloor \sigma/2\right\rfloor }a_{p}z^{p+1}\overline{z}^{p}\\
\mu_{2}\overline{z}+\sum_{p=1}^{\left\lfloor \sigma/2\right\rfloor }\overline{a}_{p}z^{p}\overline{z}^{p+1}
\end{array}\end{pmatrix},\label{eq:FITcplxRedMod}
\end{equation}
where $z,a_{p}\in\mathbb{C}$. The choice of terms in (\ref{eq:FITcplxRedMod})
avoids diverging terms in the submersion $\boldsymbol{U}$ when $\left|\mu_{1}\right|\approx1$
as illustrated by formula (\ref{eq:NoInternalResonance}) in the proof
of theorem \ref{thm:MapFoliation}. Using the transformation $z=z_{1}+iz_{2}$,
with $z_{1},z_{2}\in\mathbb{R}$, the dynamics on the ISF (\ref{eq:FITcplxRedMod})
can be written in real coordinates as 
\begin{equation}
\hat{\boldsymbol{S}}\left(z_{1},z_{2}\right)=\begin{pmatrix}\begin{array}{l}
z_{1}\sum_{p=0}^{\left\lfloor \sigma/2\right\rfloor }b_{p}\left(z_{1}^{2}+z_{2}^{2}\right)^{p}-z_{2}\sum_{p=0}^{\left\lfloor \sigma/2\right\rfloor }c_{p}\left(z_{1}^{2}+z_{2}^{2}\right)^{p}\\
z_{1}\sum_{p=0}^{\left\lfloor \sigma/2\right\rfloor }c_{p}\left(z_{1}^{2}+z_{2}^{2}\right)^{p}+z_{2}\sum_{p=0}^{\left\lfloor \sigma/2\right\rfloor }b_{p}\left(z_{1}^{2}+z_{2}^{2}\right)^{p}
\end{array}\end{pmatrix},\label{eq:FITrealRedMod}
\end{equation}
where $b_{0}=\Re\mu_{1}$, $c_{0}=\Im\mu_{1}$ and $b_{p}=\Re a_{p}$,
$c_{p}=\Im a_{p}$, $p=1,\ldots,\left\lfloor \sigma/2\right\rfloor $.
To generalise even further, and allow the limit $\left|\mu_{1}\right|\to1$,
we can also write that
\begin{equation}
\boldsymbol{S}\left(z_{1},z_{2}\right)=\begin{pmatrix}\begin{array}{l}
z_{1}f_{r}\left(z_{1}^{2}+z_{2}^{2}\right)-z_{2}f_{i}\left(z_{1}^{2}+z_{2}^{2}\right)\\
z_{1}f_{i}\left(z_{1}^{2}+z_{2}^{2}\right)+z_{2}f_{r}\left(z_{1}^{2}+z_{2}^{2}\right)
\end{array}\end{pmatrix},\label{eq:FITgeneralRedMod}
\end{equation}
where $f_{r}$ and $f_{i}$ are unknown functions. We note that theorem
\ref{thm:MapFoliation} does not cover the case $\left|\mu_{1}\right|=1$,
however the invariance equation can be solved by the asymptotic expansion
described in appendix \ref{subsec:MAPexpansion} up to any order of
accuracy even when $\left|\mu_{1}\right|=1$. This suggests that the
invariance equation can also be solved numerically up to any order
of accuracy, when (\ref{eq:MapNonResonance}) holds and near internal
resonances are taken into account, even though the existence of a
unique solution is not known.

The dynamics on the ISF can be further analysed by introducing the
polar parametrisation $z_{1}=r\cos\theta$ and $z_{2}=r\sin\theta$.
In these coordinates equation (\ref{eq:FITgeneralRedMod}) is transformed
into
\begin{equation}
\breve{\boldsymbol{S}}\left(r,\theta\right)=\begin{pmatrix}\begin{array}{l}
{\displaystyle r\sqrt{f_{r}^{2}\left(r^{2}\right)+f_{i}^{2}\left(r^{2}\right)}}\\
{\displaystyle \theta+\tan^{-1}\frac{f_{i}\left(r^{2}\right)}{f_{r}\left(r^{2}\right)}}
\end{array}\end{pmatrix}.\label{eq:FITpolarRedMod}
\end{equation}
For a similar analysis see \cite[section 6]{Szalai20160759}. The
radial dynamics in (\ref{eq:FITpolarRedMod}) is decoupled from the
angular motion, therefore we can identify that $r=0$ is the fixed
point, and all solutions of $f_{r}^{2}\left(r^{2}\right)+f_{i}^{2}\left(r^{2}\right)=1$
for $r$ with $r>0$ represent periodic orbits. We assume that each
iteration of $\boldsymbol{F}$ and of $\breve{\boldsymbol{S}}$ accounts
for a period of time $T$ and therefore the instantaneous angular
frequency of rotation about the fixed point is given by 
\begin{equation}
\omega_{E^{\star}}\left(r\right)=T^{-1}\tan^{-1}\frac{f_{i}\left(r^{2}\right)}{f_{r}\left(r^{2}\right)}.\label{eq:FITfrequency}
\end{equation}
We also define the instantaneous damping ratio by 
\begin{equation}
\zeta_{E^{\star}}\left(r\right)=-\frac{\log\sqrt{f_{r}^{2}\left(r^{2}\right)+f_{i}^{2}\left(r^{2}\right)}}{T\omega_{E^{\star}}\left(r\right)},\label{eq:FITdamping}
\end{equation}
which agrees with the damping ratio of the linear dynamics about the
equilibrium at $r=0$. Unfortunately we cannot easily determine what
vibration amplitude $r$ represents, because there is no unique closed
curve in the phase space that is mapped by the submersion $\boldsymbol{U}$
to the circle $r\times[0,2\pi)$. This means that we cannot define
a \emph{backbone curve} in the same way as in \cite[section 6]{Szalai20160759}.
Instead, we define a surrogate for the amplitude in section \ref{subsec:BackboneCurves}.

Similarly, the submersion of the ISF needs to be normalised, because
in case of an internal resonance it is not fully specified. We are
now looking for a $\tilde{\boldsymbol{U}}$, which together with $\tilde{\boldsymbol{S}}$
satisfies the invariance equation (\ref{eq:adjInvariance}), and also
takes into account the near internal resonances (\ref{eq:FITinternalRes}).
In order to uncover the constraints on $\tilde{\boldsymbol{U}}$ that
eliminate the terms corresponding to near internal resonances, we
write that 
\[
\tilde{\boldsymbol{U}}\left(\boldsymbol{x}\right)=\boldsymbol{U}\left(\boldsymbol{v}_{1}^{\star}\boldsymbol{x},\boldsymbol{v}_{2}^{\star}\boldsymbol{x},\ldots,\boldsymbol{v}_{n}^{\star}\boldsymbol{x}\right),
\]
where $\boldsymbol{U}=\left(U_{1},U_{2}\right)^{T}$ and $U_{1},U_{2}$
form a complex conjugate pair, which have real and imaginary parts,
such that $U_{1}=U_{r}+iU_{i}$. Note that $\boldsymbol{U}$ is the
same submersion that is used in appendix \ref{sec:ProofThm1}, where
$\boldsymbol{F}$ was assumed to have a diagonal Jacobian at the origin.
Similarly, we decompose $\tilde{\boldsymbol{U}}=\left(\tilde{U}_{1},\tilde{U}_{2}\right)^{T}$
and $\tilde{U}_{1}=\tilde{U}_{r}+i\tilde{U}_{i}$ and define $\hat{\boldsymbol{U}}=\left(\hat{U}_{1},\hat{U}_{2}\right)\overset{\mathit{def}}{=}\left(\tilde{U}_{r},\tilde{U}_{i}\right)$,
which together with $\hat{\boldsymbol{S}}$ or $\boldsymbol{S}$ of
equations (\ref{eq:FITrealRedMod}) and (\ref{eq:FITgeneralRedMod}),
respectively, must satisfy the invariance equation (\ref{eq:adjInvariance}).
Due to our assumptions, the left and right eigenvectors satisfy $\boldsymbol{v}_{j}^{\star}\boldsymbol{v}_{k}=\delta_{jk}$,
where $\delta_{jk}$ is the Kronecker delta, hence we can write that
\[
\tilde{\boldsymbol{U}}\left(\sum_{j=1}^{n}\boldsymbol{v}_{j}z_{j}\right)=\boldsymbol{U}\left(z_{1},z_{2},\ldots,z_{n}\right).
\]
As in appendix \ref{subsec:MAPexpansion}, we recognise that the terms
corresponding to internal resonances are
\[
z_{1}^{p+1}z_{2}^{p}\;\text{and}\;z_{1}^{p}z_{2}^{p+1},\;p\ge1
\]
in $U_{1}$ and $U_{2}$, respectively, whose coefficients need to
vanish as per equation (\ref{eq:YesInternalResonance}). To remove
these terms, we set $z_{1}=r\mathrm{e}^{i\theta}$, $z_{2}=r\mathrm{e}^{-i\theta}$
which leads to internally resonant terms $r^{2p+1}\mathrm{e}^{i\theta}$
and $r^{2p+1}\mathrm{e}^{-i\theta}$ that are the only terms with
$\mathrm{e}^{i\theta}$ and $\mathrm{e}^{-i\theta}$ components in
the Fourier expansion of $\boldsymbol{U}$. In particular for $U_{1}$
only the linear term ($r\mathrm{e}^{i\theta}$ for $p=0$) is allowed
to contribute to a non-zero coefficient of $\mathrm{e}^{i\theta}$,
which means that the first Fourier coefficient must be 
\begin{equation}
\int_{0}^{2\pi}\mathrm{e}^{-i\theta}\cdot U_{1}\left(r\mathrm{e}^{i\theta},r\mathrm{e}^{-i\theta},0,\ldots,0\right)=2\pi r,\label{eq:FITcplxConstraint}
\end{equation}
where we assumed the normalisation $D_{k}U_{1}\left(0,\ldots,0\right)=\delta_{1k}$.
Since $U_{1}$ and $U_{2}$ are complex conjugate pairs, there is
no need for a similar condition for $U_{2}$. Instead, we expand the
constraint (\ref{eq:FITcplxConstraint}) using the real valued submersion
$\hat{\boldsymbol{U}}$, such that the final form of the constraint
becomes 
\begin{equation}
\left.\begin{array}{rl}
{\displaystyle \int_{0}^{2\pi}\hat{U}_{1}\left(\boldsymbol{v}_{r}r\cos\theta-\boldsymbol{v}_{i}r\sin\theta\right)\cos\theta+\hat{U}_{2}\left(\boldsymbol{v}_{r}r\cos\theta-\boldsymbol{v}_{i}r\sin\theta\right)\sin\theta\mathrm{d}\theta} & =2\pi r\\
{\displaystyle \int_{0}^{2\pi}\hat{U}_{2}\left(\boldsymbol{v}_{r}r\cos\theta-\boldsymbol{v}_{i}r\sin\theta\right)\cos\theta-\hat{U}_{1}\left(\boldsymbol{v}_{r}r\cos\theta-\boldsymbol{v}_{i}r\sin\theta\right)\sin\theta\mathrm{d}\theta} & =0
\end{array}\right\} ,\label{eq:FITnonResConstraint}
\end{equation}
where $\boldsymbol{v}_{r}=\Re\boldsymbol{v}_{1}$ and $\boldsymbol{v}_{i}=\Im\boldsymbol{v}_{1}$.
In what follows the constraints (\ref{eq:FITnonResConstraint}) will
turn into penalty terms added to the loss function of the optimisation
problem, whose minimum is the approximate pair of functions $\boldsymbol{U}$
and $\boldsymbol{S}$.

\subsection{\label{subsec:PenaltyOptimisation}The optimisation problem}

Let us assume a set of data points, given by $\left\{ \left(\boldsymbol{x}_{k},\boldsymbol{y}_{k}\right),\,k=1,\ldots,N\right\} $,
with the constraint that $\boldsymbol{y}_{k}=\boldsymbol{F}\left(\boldsymbol{x}_{k}\right)$.
In practice, $\left(\boldsymbol{x}_{k},\boldsymbol{y}_{k}\right)$
may be part of a set of trajectories, such that $\boldsymbol{y}_{k}=\boldsymbol{x}_{k+1}$
for ranges of subsequent indices $K_{j}\le k<K_{j+1}$, $1=K_{1}<K_{2}<\cdots<K_{M}=N$.
We also assume that there is an approximate knowledge of the Jacobian
of $\boldsymbol{F}$ about the equilibrium. To find the Jacobian one
can use standard linear regression that fits a linear model to the
data in the neighbourhood of the equilibrium \cite{boyd_vandenberghe_2018}. 

We further assume parametric representations of the submersion $\boldsymbol{U}$
and the map $\boldsymbol{S}$, such that $\boldsymbol{U}\left(\boldsymbol{0}\right)=\boldsymbol{0}$
and $\boldsymbol{S}$ has the form of (\ref{eq:FITgeneralRedMod}).
In particular, we use the notation $\boldsymbol{U}\left(\boldsymbol{x}\right)=\boldsymbol{U}\left(\boldsymbol{x};\boldsymbol{\Theta}_{\boldsymbol{U}}\right)$
and $\boldsymbol{S}\left(\boldsymbol{z}\right)=\boldsymbol{S}\left(\boldsymbol{z};\boldsymbol{\Theta}_{\boldsymbol{S}}\right)$,
where $\boldsymbol{\Theta}_{\boldsymbol{U}}$ and $\boldsymbol{\Theta}_{\boldsymbol{S}}$
are the parameters we are looking for. Functions $\boldsymbol{U}$
and $\boldsymbol{S}$ must satisfy the invariance equation (\ref{eq:adjInvariance})
at each point along the time-series with the smallest possible residual
error $\boldsymbol{r}_{k}$, that is
\begin{align*}
\boldsymbol{U}\left(\boldsymbol{y}_{k};\boldsymbol{\Theta}_{\boldsymbol{U}}\right) & =\boldsymbol{S}\left(\boldsymbol{U}\left(\boldsymbol{x}_{k};\boldsymbol{\Theta}_{\boldsymbol{U}}\right);\boldsymbol{\Theta}_{\boldsymbol{S}}\right)+\boldsymbol{r}_{k}.
\end{align*}
An obvious strategy to minimise the residual $\boldsymbol{r}_{k}$,
is to use the least-squares method. In particular, we use the scaled
norm (\ref{eq:SigmaNorm}) from the proof of theorem \ref{thm:MapFoliation}
in appendix \ref{subsec:ContractionMapping}, which guarantees a unique
solution. The loss term from the invariance equation is then
\begin{equation}
L_{i}\left(\boldsymbol{\Theta}_{\boldsymbol{U}},\boldsymbol{\Theta}_{\boldsymbol{S}}\right)=\sum_{k=1}^{N}\left|\boldsymbol{x}_{k}\right|^{-2\sigma}\left|\boldsymbol{U}\left(\boldsymbol{y}_{k};\boldsymbol{\Theta}_{\boldsymbol{U}}\right)-\boldsymbol{S}\left(\boldsymbol{U}\left(\boldsymbol{x}_{k};\boldsymbol{\Theta}_{\boldsymbol{U}}\right);\boldsymbol{\Theta}_{\boldsymbol{S}}\right)\right|^{2}.\label{eq:FITinvarianceLoss}
\end{equation}
We also need to ensure that the normalising conditions (\ref{eq:FITnonResConstraint})
are satisfied. We choose a two-dimensional mesh in polar coordinates,
that is $r_{j}=r_{\mathrm{max}}j/N_{r}$, $\theta_{k}=2\pi k/N_{\theta}$
and $\boldsymbol{v}_{jk}=\boldsymbol{v}_{r}r_{j}\cos\theta_{k}-\boldsymbol{v}_{i}r_{j}\sin\theta_{k}$
and define 
\begin{multline}
L_{n}\left(\boldsymbol{\Theta}_{\boldsymbol{U}}\right)=\sum_{j=1}^{N_{r}}\left(r_{j}^{-1}\sum_{k=1}^{N_{\theta}}\left(U_{1}\left(\boldsymbol{v}_{jk};\boldsymbol{\Theta}_{\boldsymbol{U}}\right)\cos\theta_{k}+U_{2}\left(\boldsymbol{v}_{jk};\boldsymbol{\Theta}_{\boldsymbol{U}}\right)\sin\theta_{k}\right)-\frac{N_{\theta}}{2}\right)^{2}+\\
+\sum_{j=1}^{N_{r}}\left(r_{j}^{-1}\sum_{k=1}^{N_{\theta}}\left(U_{2}\left(\boldsymbol{v}_{jk};\boldsymbol{\Theta}_{\boldsymbol{U}}\right)\cos\theta_{k}-U_{1}\left(\boldsymbol{v}_{jk};\boldsymbol{\Theta}_{\boldsymbol{U}}\right)\sin\theta_{k}\right)\right)^{2}.\label{eq:FITnormaliseLoss}
\end{multline}
The value of $r_{\mathrm{max}}$ is proportional to $\max_{k}\left|\boldsymbol{x}_{k}\right|$.
Our version of the least-squares optimisation problem can be written
as
\begin{equation}
\boldsymbol{\Theta}_{\boldsymbol{U}},\boldsymbol{\Theta}_{\boldsymbol{S}}=\arg\min\left(L_{i}\left(\boldsymbol{\Theta}_{\boldsymbol{U}},\boldsymbol{\Theta}_{\boldsymbol{S}}\right)+\beta L_{n}\left(\boldsymbol{\Theta}_{\boldsymbol{U}}\right)\right),\label{eq:SigmaLeastSquares}
\end{equation}
where $\beta>0$ is sufficiently large so that $\boldsymbol{U}$ continues
to satisfy the normalising conditions (\ref{eq:FITnonResConstraint}).
The optimisation must be initialised such that 
\begin{equation}
D_{1}\boldsymbol{U}\left(\boldsymbol{0};\boldsymbol{\Theta}_{\boldsymbol{U}}\right)\approx\begin{pmatrix}\begin{array}{l}
\Re\boldsymbol{v}_{1}^{\star}\\
\Im\boldsymbol{v}_{1}^{\star}
\end{array}\end{pmatrix}\;\text{and}\;f_{r}\left(0\right)\approx\Re\mu_{1},\;f_{i}\left(0\right)\approx\Im\mu_{1}.\label{eq:FIT_IC}
\end{equation}

\begin{remark}
An alternative to the normalising conditions (\ref{eq:FITnonResConstraint})
is to fix the norm of $D_{1}\boldsymbol{U}\left(\boldsymbol{0};\boldsymbol{\Theta}_{\boldsymbol{U}}\right)$,
by defining 
\begin{equation}
L_{n}\left(\boldsymbol{U}\right)=\left(\left\Vert D_{1}\boldsymbol{U}\left(\boldsymbol{0};\boldsymbol{\Theta}_{\boldsymbol{U}}\right)\right\Vert ^{2}-1\right)^{2}.\label{eq:FITDUnormLoss}
\end{equation}
\end{remark}
In this case, the optimisation (\ref{eq:SigmaLeastSquares}) will
not yield a unique result for $\boldsymbol{\Theta}_{\boldsymbol{U}},\boldsymbol{\Theta}_{\boldsymbol{S}}$,
however according to theorem \ref{thm:MapFoliation} the foliation
defined by the resulting $\boldsymbol{U}$ should represent the unique
foliation. The non-uniqueness comes from the possible choices of terms
in $\boldsymbol{U}$ and $\boldsymbol{S}$ relative to each other
at near internal resonances as described in appendix \ref{subsec:MAPexpansion}.

\subsection{\label{subsec:DATAexpansionOpt}Polynomial representation for optimisation}

Here we use a polynomial representation to carry out the optimisation
given by equation (\ref{eq:SigmaLeastSquares}). We represent the
unknown functions as polynomials of finite order $\alpha$, such that
\begin{align*}
\boldsymbol{U}\left(\boldsymbol{x};\boldsymbol{U}^{\boldsymbol{m}_{1}},\ldots,\boldsymbol{U}^{\boldsymbol{m}_{\#\left[n,\alpha\right]}}\right) & =\sum_{\boldsymbol{m}\in M_{n,\alpha}}\boldsymbol{U}^{\boldsymbol{m}}\boldsymbol{x}^{\boldsymbol{m}},\\
\boldsymbol{S}\left(\boldsymbol{z};\boldsymbol{S}^{\boldsymbol{m}_{1}},\ldots,\boldsymbol{S}^{\boldsymbol{m}_{\#\left[2,\alpha\right]}}\right) & =\sum_{\boldsymbol{m}\in M_{2,\alpha}}\boldsymbol{S}^{\boldsymbol{m}}\boldsymbol{z}^{\boldsymbol{m}},
\end{align*}
where the finite set is $M_{n,\alpha}=\left\{ \boldsymbol{m}\in\mathbb{N}^{n}:1\le\sum_{k=1}^{n}m_{k}\le\alpha\right\} $,
the unique elements of $M_{n,\alpha}$ are denoted by $\boldsymbol{m}_{1},\boldsymbol{m}_{2},\ldots,\boldsymbol{m}_{\#\left[n,\alpha\right]}$
and $\#\left[n,\alpha\right]=\binom{n+\alpha}{n}-1$ is the cardinality
of $M_{n,\alpha}$. The scalar values $\boldsymbol{x}^{\boldsymbol{m}}$
are defined as
\begin{equation}
\boldsymbol{x}^{\boldsymbol{m}}=x_{1}^{m_{1}}\cdots x_{n}^{m_{n}}\label{eq:MonomialDef}
\end{equation}
and $\boldsymbol{U}^{\boldsymbol{m}},\boldsymbol{S}^{\boldsymbol{m}}\in\mathbb{R}^{2}$.
We further define that $\left|\boldsymbol{m}\right|=\sum_{k=1}^{n}m_{k}$.
The multi-index notation implies that coefficients of linear terms
have indices given by unit vectors 
\[
\boldsymbol{e}_{k}=\left(\underset{1}{0},\ldots,\underset{k-1}{0},\underset{k}{1},\underset{k+1}{0}\ldots,\underset{n\text{ or }\nu}{0}\right).
\]
Matrices are consequently denoted as multi-indexed vectors, that is,
the element of a matrix in the $j$-th row and $k$-th column is written
as $U_{j}^{\boldsymbol{e}_{k}}$ or just simply the $k$-th column
vector of a matrix is written as $\boldsymbol{U}^{\boldsymbol{e}_{k}}$.
In order to arrive at the form of $\boldsymbol{S}$ given by (\ref{eq:FITrealRedMod}),
we need to set 
\[
\left.\begin{array}{ll}
S_{1}^{\left(1+2p,2\left(k-p\right)\right)} & ={\displaystyle \binom{k}{p}b_{k}}\\
S_{1}^{\left(2p,1+2\left(k-p\right)\right)} & ={\displaystyle -\binom{k}{p}c_{k}}\\
S_{2}^{\left(1+2p,2\left(k-p\right)\right)} & ={\displaystyle \binom{k}{p}c_{k}}\\
S_{2}^{\left(2p,1+2\left(k-p\right)\right)} & ={\displaystyle \binom{k}{p}b_{k}}
\end{array}\right\} \quad0\le k\le\left\lfloor \alpha/2\right\rfloor ,\;0\le p\le k.
\]
Finally, as per the notation of section \ref{subsec:PenaltyOptimisation},
the parameter arrays are given by 
\begin{align*}
\boldsymbol{\Theta}_{\boldsymbol{S}} & =\left(b_{0},\ldots b_{\left\lfloor \alpha/2\right\rfloor },c_{0},\ldots c_{\left\lfloor \alpha/2\right\rfloor }\right),\\
\boldsymbol{\Theta}_{\boldsymbol{U}} & =\left(\boldsymbol{U}^{\boldsymbol{m}_{1}},\ldots,\boldsymbol{U}^{\boldsymbol{m}_{\#\left[n,\alpha\right]}}\right).
\end{align*}
The starting point of the optimisation is using the eigenvalues and
left eigenvectors of the Jacobian at the origin
\begin{equation}
U_{1}^{\boldsymbol{e}_{k}}=\left[\Re\boldsymbol{v}_{1}^{\star}\right]_{k},\;U_{2}^{\boldsymbol{e}_{k}}=\left[\Im\boldsymbol{v}_{1}^{\star}\right]_{k},\;b_{0}=\Re\mu_{1},\;c_{0}=\Im\mu_{1},\label{eq:FITParIni}
\end{equation}
while the rest of the parameters can be initialised either randomly
or to zero. During the optimisation the values (\ref{eq:FITParIni})
are allowed to change to fit the data, the initialisation ensures
that the ISF converges to the chosen linear subspace $E^{\star}.$

The polynomial representation of the objective function in the optimisation
problem (\ref{eq:SigmaLeastSquares}) can be written as 
\begin{multline}
\mathit{loss}\left(\boldsymbol{\Theta}_{\boldsymbol{U}},\boldsymbol{\Theta}_{\boldsymbol{S}}\right)=\beta L_{n}\left(\boldsymbol{\Theta}_{\boldsymbol{U}}\right)+\\
\quad+\sum_{k=1}^{N}\left|\boldsymbol{x}_{k}\right|^{-2\sigma}\left|\sum_{\boldsymbol{m}\in M_{n,\alpha}}\boldsymbol{U}^{\boldsymbol{m}}\boldsymbol{y}_{k}^{\boldsymbol{m}}-\sum_{\boldsymbol{p}\in M_{2,\alpha}}\boldsymbol{S}^{\boldsymbol{p}}\left(\sum_{\boldsymbol{m}\in M_{n,\alpha}}\boldsymbol{U}^{\boldsymbol{m}}\boldsymbol{x}_{k}^{\boldsymbol{m}}\right)^{\boldsymbol{p}}\right|^{2},\label{eq:PolyObjFunction}
\end{multline}
where
\begin{multline}
L_{n}\left(\boldsymbol{\Theta}_{\boldsymbol{U}}\right)=\sum_{j=1}^{N_{r}}\left(\sum_{\boldsymbol{m}\in M_{n,\alpha}}\sum_{k=1}^{N_{\theta}}\left(U_{1}^{\boldsymbol{m}}\boldsymbol{c}_{jk}^{\boldsymbol{m}}+U_{2}^{\boldsymbol{m}}\boldsymbol{s}_{jk}^{\boldsymbol{m}}\right)-\frac{N_{\theta}}{2}\right)^{2}+\\
+\sum_{j=1}^{N_{r}}\left(\sum_{\boldsymbol{m}\in M_{n,\alpha}}\sum_{k=1}^{N_{\theta}}\left(U_{2}^{\boldsymbol{m}}\boldsymbol{c}_{jk}^{\boldsymbol{m}}-U_{1}^{\boldsymbol{m}}\boldsymbol{s}_{jk}^{\boldsymbol{m}}\right)\right)^{2}\label{eq:FITnormaliseLoss-1}
\end{multline}
and 
\begin{align*}
\boldsymbol{c}_{jk}^{\boldsymbol{m}} & =r_{j}^{\left|\boldsymbol{m}\right|-1}\left(\boldsymbol{v}_{r}\cos\theta_{k}-\boldsymbol{v}_{i}\sin\theta_{k}\right)^{\boldsymbol{m}}\cos\theta_{k},\\
\boldsymbol{s}_{jk}^{\boldsymbol{m}} & =r_{j}^{\left|\boldsymbol{m}\right|-1}\left(\boldsymbol{v}_{r}\cos\theta_{k}-\boldsymbol{v}_{i}\sin\theta_{k}\right)^{\boldsymbol{m}}\sin\theta_{k}.
\end{align*}
The penalty term $L_{n}$ uses the approximate right eigenvectors
$\boldsymbol{v}_{r}\pm i\boldsymbol{v}_{i}$, which do not adapt during
the optimisation. We do not expect that this causes inaccuracies,
because this is just one possible way of normalising the submersion
$\boldsymbol{U}$ which still represents the unique ISF. An inaccuracy
of the a-priori estimated eigenvectors $\boldsymbol{v}_{r}\pm i\boldsymbol{v}_{i}$
however will affect the conjugate map $\boldsymbol{S}$.

Another issue is that for a large $\beta$ the penalty term can overshadow
the actual loss function $L_{i}$, which may cause inaccuracies. On
the other hand, for smaller values of $\beta$ the constraint (\ref{eq:FITnonResConstraint})
may not hold accurately. It is however much less important to satisfy
the constraint accurately than finding the minimum of $L_{i}$, because
the constraint (\ref{eq:FITnonResConstraint}) only affects the parametrisation
of the ISF and not its geometry. Therefore the value of $\beta$ can
be limited so that the minimum of the penalised loss function remains
close to the minimum of $L_{n}$. Alternatively, one can use constrained
optimisation, such as sequential quadratic programming \cite{nocedal2000numerical}
to take (\ref{eq:FITnonResConstraint}) into account with full numerical
accuracy.

From experience with other model identification studies, we believe
that accuracy can be improved if not just two consecutive points,
but multiple points along a trajectory are taken into account. This
leads to a so-called 'multiple shooting' technique \cite{Bock1983},
which will be part of a further investigation.

In our implementation we use the \noun{Optim.jl} \cite{mogensen2018optim}
package of the \noun{Julia} programming language and choose the BFGS
method to find an optimal solution. This only requires the gradient
of $\mathit{loss}$, which can be calculated by automatic differentiation.

\section{Analysis of ISFs}

\subsection{Reconstructing the dynamics}

Two or more carefully selected ISFs can act as a nonlinear coordinate
system of the state space and therefore can be used to reconstruct
the dynamics of $\boldsymbol{F}$. Let $E_{j}^{\star}$, $j=1,\ldots,q$
be invariant linear subspaces, satisfying the conditions of theorem
\ref{thm:MapFoliation}, such that 
\begin{equation}
\left.\begin{array}{rl}
E_{j}^{\star}\cap E_{k}^{\star} & =\left\{ \mathbf{0}\right\} ,\;\forall j\neq k\\
E_{1}^{\star}\oplus E_{2}^{\star}\cdots\oplus E_{q}^{\star} & =\mathbb{R}^{n}
\end{array}\right\} \label{eq:IntersectCond}
\end{equation}
and let the corresponding submersions of the ISFs be denoted by $\boldsymbol{U}^{j}$.
Further, assume trajectories $\boldsymbol{x}_{k}$, and $\boldsymbol{z}_{j,k}$,
$k\in\mathbb{N}$ satisfying $\boldsymbol{x}_{k+1}=\boldsymbol{F}\left(\boldsymbol{x}_{k}\right)$,
$\boldsymbol{z}_{j,k+1}=\boldsymbol{S}^{j}\left(\boldsymbol{z}_{j,k}\right)$
with matching initial conditions, that is, $\boldsymbol{z}_{j,0}=\boldsymbol{U}^{j}\left(\boldsymbol{x}_{0}\right)$.
Because of the invariance of ISFs, the trajectories satisfy the equation
\begin{equation}
\begin{pmatrix}\boldsymbol{z}_{1,k}\\
\vdots\\
\boldsymbol{z}_{q,k}
\end{pmatrix}=\widehat{\boldsymbol{U}}\left(\boldsymbol{x}_{k}\right)\overset{\mathit{def}}{=}\begin{pmatrix}\boldsymbol{U}^{1}\left(\boldsymbol{x}_{k}\right)\\
\vdots\\
\boldsymbol{U}^{q}\left(\boldsymbol{x}_{k}\right)
\end{pmatrix},\;k=1,2,\ldots\;.\label{eq:FullToRed}
\end{equation}
Due to our assumptions about $E_{j}^{\star}$, we can invert $\widehat{\boldsymbol{U}}$
in a neighbourhood of the origin and therefore there exist a function
$\boldsymbol{h}$, such that 
\begin{equation}
\widehat{\boldsymbol{U}}\left(\boldsymbol{h}\left(\boldsymbol{z}_{1},\ldots,\boldsymbol{z}_{q}\right)\right)=\begin{pmatrix}\boldsymbol{z}_{1}\\
\vdots\\
\boldsymbol{z}_{2}
\end{pmatrix}.\label{eq:RestoreDef}
\end{equation}
Using $\boldsymbol{h}$, the equivalence of the trajectories is expressed
as
\begin{equation}
\boldsymbol{x}_{k}=\boldsymbol{h}\left(\boldsymbol{z}_{1,k},\ldots,\boldsymbol{z}_{q,k}\right),\;k=1,2,\ldots\;.\label{eq:RedToFull}
\end{equation}
Equation (\ref{eq:RedToFull}) can be used to reconstruct the full
dynamics of the system from the lower order conjugate dynamics of
the ISFs. The equivalence is the same between the trajectories of
the vector fields $\boldsymbol{G}$, $\boldsymbol{R}^{j}$, except
that the subscript $k$ is replaced by time $t$.

Function $\boldsymbol{h}$ can be obtained by a fixed point iteration
in a small neighbourhood of the origin. Let us denote $\boldsymbol{C}=D\widehat{\boldsymbol{U}}\left(\boldsymbol{0}\right)$
and decompose $\widehat{\boldsymbol{U}}\left(\boldsymbol{x}\right)=\boldsymbol{C}\boldsymbol{x}+\widehat{\boldsymbol{U}}_{N}\left(\boldsymbol{x}\right)$,
such that $\widehat{\boldsymbol{U}}_{N}\left(\boldsymbol{x}\right)=\mathcal{O}\left(\left|\boldsymbol{x}\right|^{2}\right)$.
Due to our assumptions (\ref{eq:IntersectCond}), we infer that $\boldsymbol{C}$
is invertible, therefore the iteration makes sense
\begin{equation}
\boldsymbol{h}_{l+1}\left(\boldsymbol{z}_{1},\ldots,\boldsymbol{z}_{q}\right)=\boldsymbol{C}^{-1}\begin{pmatrix}\boldsymbol{z}_{1}\\
\vdots\\
\boldsymbol{z}_{q}
\end{pmatrix}-\boldsymbol{C}^{-1}\widehat{\boldsymbol{U}}_{N}\left(\boldsymbol{h}_{l}\left(\boldsymbol{z}_{1},\ldots,\boldsymbol{z}_{q}\right)\right),\;\boldsymbol{h}_{0}\left(\boldsymbol{z}_{1},\ldots,\boldsymbol{z}_{q}\right)=\boldsymbol{0}.\label{eq:InverseSubmersion}
\end{equation}
The iteration (\ref{eq:InverseSubmersion}) converges in a neighbourhood
of the origin, where $\boldsymbol{C}^{-1}\widehat{\boldsymbol{U}}_{N}$
is a contraction \cite{agarwal2018fixed}. For polynomials of a given
order the iteration always converges in finite number of steps if
the resulting polynomial is truncated to a finite order at each iteration.

Finding function $\boldsymbol{h}$ recovers all the SSMs of the system
at the same time. Let $j\in\left\{ 1,\ldots,q\right\} $. It is quick
to show that
\begin{equation}
\boldsymbol{W}^{j}\left(\boldsymbol{z}_{j}\right)=\boldsymbol{h}\left(\boldsymbol{0},\ldots,\boldsymbol{0},\boldsymbol{z}_{j},\boldsymbol{0},\ldots,\boldsymbol{0}\right)\label{eq:SSMimmersion}
\end{equation}
is the immersion of the SSM and $\boldsymbol{S}^{j}$ is the SSM conjugate
dynamics. Indeed, applying $\boldsymbol{W}^{j}$ to the invariance
equation (\ref{eq:adjInvariance}) from both sides gives
\[
\boldsymbol{W}^{j}\circ\boldsymbol{U}^{j}\circ\boldsymbol{F}\circ\boldsymbol{W}^{j}=\boldsymbol{W}^{j}\circ\boldsymbol{S}^{j}\circ\boldsymbol{U}^{j}\circ\boldsymbol{W}^{j},
\]
where we notice that $\boldsymbol{U}^{j}\circ\boldsymbol{W}^{j}$
is the identity, by construction, and $\boldsymbol{W}^{j}\circ\boldsymbol{U}^{j}$
is a projection, and also the identity on the range of $\boldsymbol{F}\circ\boldsymbol{W}^{j}$.
Therefore we are left with the SSM invariance equation
\[
\boldsymbol{F}\circ\boldsymbol{W}^{j}=\boldsymbol{W}^{j}\circ\boldsymbol{S}^{j},
\]
which proves our statement.

\begin{remark}
\label{rem:LeavesFromInverse}The leaves of the ISF can be explicitly
constructed using the function $\boldsymbol{h}$, that is 
\[
\mathcal{L}_{\boldsymbol{z}}^{j}=\left\{ \boldsymbol{h}\left(\boldsymbol{c}_{1},\ldots,\boldsymbol{c}_{j-1},\boldsymbol{z},\boldsymbol{c}_{j+1},\ldots,\boldsymbol{c}_{q}\right):\boldsymbol{c}_{l}\in\mathbb{R}^{\nu_{l}},l=1,\ldots,q,l\neq j\right\} .
\]
However the information about the foliation $\mathcal{F}^{j}$ is
already encoded in the submersion $\boldsymbol{U}^{j}$, hence finding
a full set of foliations satisfying (\ref{eq:IntersectCond}) and
then calculating $\boldsymbol{h}$ is inefficient. In the next section,
we develop a more efficient technique to find explicit expressions
for $\mathcal{L}_{\boldsymbol{z}}^{j}$.
\end{remark}

\subsection{\label{subsec:LeavesCalc}The leaves of an ISF}

Each leaf of an ISF is given implicitly by (\ref{eq:LeafDefinition}).
It is however possible to describe a leaf explicitly as a forward
image of a manifold immersion without relying on the inefficient construction
of remark \ref{rem:LeavesFromInverse}. An explicit expressions for
a leaf allows us to find an SSM as $\mathcal{L}_{\boldsymbol{0}}$
or visualise the leaves of the foliation as surfaces (or lines). It
will also aid us to define backbone curves in section \ref{subsec:BackboneCurves}.

We construct the family of immersions $\boldsymbol{W}_{\boldsymbol{z}}:\mathbb{R}^{n-\nu}\to\mathbb{R}^{n}$
from a submersion $\boldsymbol{U}:\mathbb{R}^{n}\to\mathbb{R}^{\nu}$,
such that a leaf within a foliation is given by 
\begin{equation}
\mathcal{L}_{\boldsymbol{z}}=\left\{ \boldsymbol{W}_{\boldsymbol{z}}\left(\boldsymbol{y}\right):\boldsymbol{y}\in\mathbb{R}^{n-\nu}\right\} .\label{eq:LeafAsGraph}
\end{equation}
To achieve this we are solving the under-determined equation 
\begin{equation}
\boldsymbol{z}=\boldsymbol{\boldsymbol{U}}\left(\boldsymbol{W}_{\boldsymbol{z}}\left(\boldsymbol{y}\right)\right)\label{eq:LeafImmersionEquation}
\end{equation}
under additional constraints, which allows a unique solution. We assume
that the immersion has the form
\begin{equation}
\boldsymbol{W}_{\boldsymbol{z}}\left(\boldsymbol{y}\right)=\boldsymbol{V}_{\perp}\boldsymbol{y}+\boldsymbol{V}_{\parallel}\boldsymbol{g}\left(\boldsymbol{z},\boldsymbol{y}\right),\label{eq:LeafImmersionDef}
\end{equation}
where $\boldsymbol{g}:\mathbb{R}^{\nu}\times\mathbb{R}^{n-\nu}\to\mathbb{R}^{\nu}$
is an unknown function. First we choose matrices $\boldsymbol{V}_{\perp}$
and $\boldsymbol{V}_{\parallel}$, such that 
\begin{equation}
D\boldsymbol{U}\left(\boldsymbol{0}\right)\boldsymbol{V}_{\perp}=\boldsymbol{0},\;D\boldsymbol{U}\left(\boldsymbol{0}\right)\boldsymbol{V}_{\parallel}=\boldsymbol{I},\;\boldsymbol{V}_{\perp}^{T}\boldsymbol{V}_{\parallel}=\boldsymbol{0}\;\text{and}\;\boldsymbol{V}_{\perp}^{T}\boldsymbol{V}_{\perp}=\boldsymbol{I}.\label{eq:LeafProjConstr}
\end{equation}
This choice constrained by (\ref{eq:LeafProjConstr}) allows for a
unique solution of $\boldsymbol{g}$ in formula (\ref{eq:LeafImmersionDef})
through the defining equation (\ref{eq:LeafImmersionEquation}). Note
that the linear subspace $E_{\parallel}$ spanned by $\boldsymbol{V}_{\parallel}$
can also be defined as 
\begin{equation}
E_{\parallel}=\left\{ \arg\min_{\boldsymbol{x}}\left|D\boldsymbol{U}\left(\boldsymbol{0}\right)\boldsymbol{x}-\boldsymbol{\xi}\right|:\boldsymbol{\xi}\in\mathbb{R}^{2}\right\} .\label{eq:Eparallel}
\end{equation}
The construction of $\boldsymbol{W}_{\boldsymbol{z}}$ is illustrated
in figure \ref{fig:LeafImmersion}.
\begin{figure}
\begin{centering}
\includegraphics[scale=0.5]{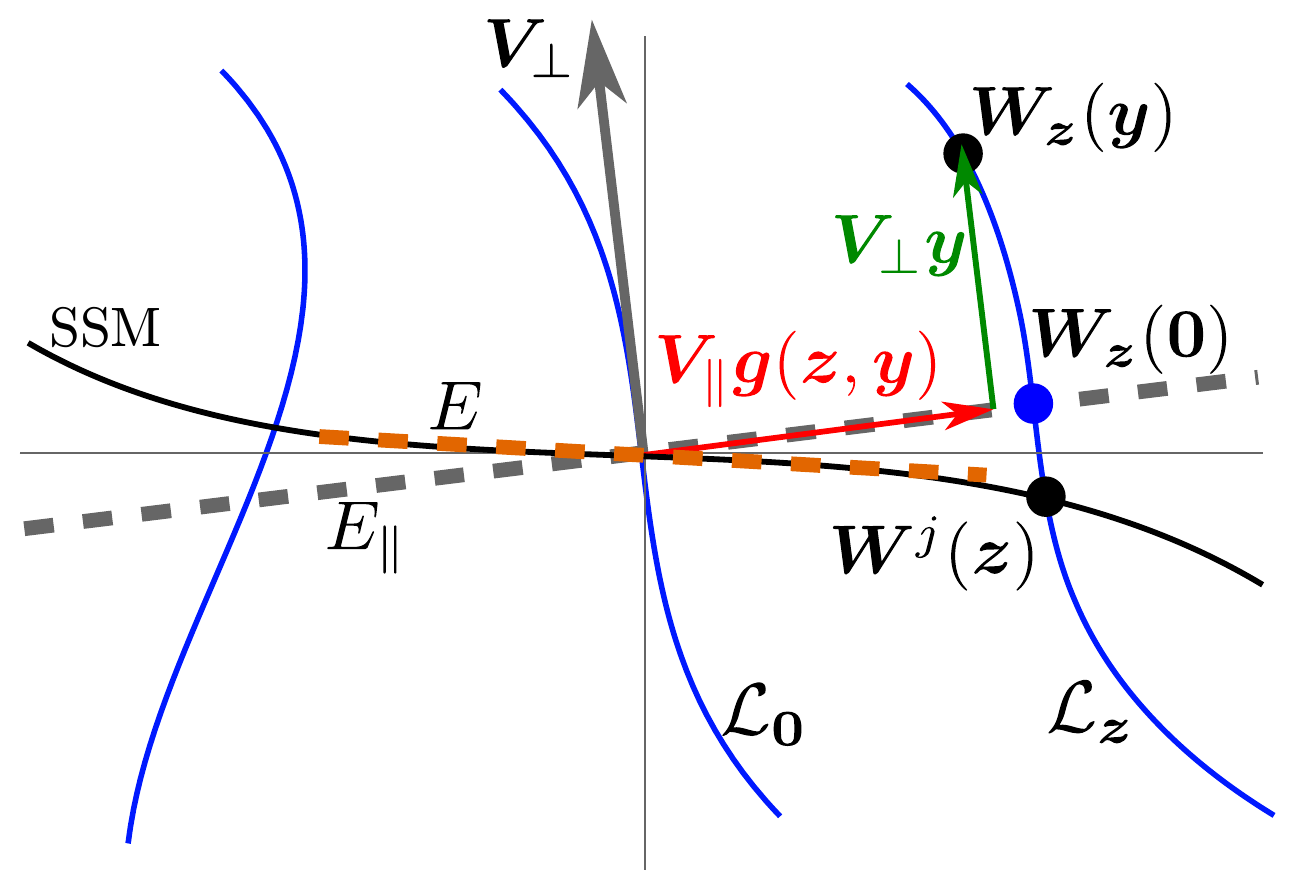}
\par\end{centering}
\caption{\label{fig:LeafImmersion}Finding the immersion $\boldsymbol{W}_{\boldsymbol{z}}$
of a leaf $\mathcal{L}_{\boldsymbol{z}}$ in the form of equation
(\ref{eq:LeafImmersionDef}). The leaf $\mathcal{L}_{\boldsymbol{z}}$
is represented as a graph over the linear subspace spanned by $\boldsymbol{V}_{\perp}$.
This representation breaks down at points where the tangent space
of $\mathcal{L}_{\boldsymbol{z}}$ is parallel with $E_{\parallel}$.
The SSM tangent to the linear subspace $E$ is also illustrated, which
coincides with $E_{\parallel}$ under the conditions outlined in remark
\ref{rem:EparNearInvariance}. When $E_{\parallel}=E$, $\boldsymbol{W}_{\boldsymbol{z}}\left(\boldsymbol{0}\right)$
is linearly asymptotic to the SSM at the origin.}
\end{figure}
We also decompose the submersion $\boldsymbol{U}$ into a linear and
nonlinear part, such that $\boldsymbol{U}\left(\boldsymbol{x}\right)=D\boldsymbol{U}\left(\boldsymbol{0}\right)\boldsymbol{x}+\boldsymbol{U}_{N}\left(\boldsymbol{x}\right)$,
then expand equation (\ref{eq:LeafImmersionEquation}) into
\begin{equation}
\boldsymbol{z}=\boldsymbol{g}\left(\boldsymbol{z},\boldsymbol{y}\right)+\boldsymbol{U}_{N}\left(\boldsymbol{V}_{\perp}\boldsymbol{y}+\boldsymbol{V}_{\parallel}\boldsymbol{g}\left(\boldsymbol{z},\boldsymbol{y}\right)\right).\label{eq:LeafGraphPreIt}
\end{equation}
Equation (\ref{eq:LeafGraphPreIt}) can be rearranged into a contraction
mapping iteration \cite{agarwal2018fixed}, that is
\begin{equation}
\boldsymbol{g}_{j+1}\left(\boldsymbol{z},\boldsymbol{y}\right)=\boldsymbol{z}-\boldsymbol{U}_{N}\left(\boldsymbol{V}_{\perp}\boldsymbol{y}+\boldsymbol{V}_{\parallel}\boldsymbol{g}_{j}\left(\boldsymbol{z},\boldsymbol{y}\right)\right),\;\boldsymbol{g}_{0}\left(\boldsymbol{z},\boldsymbol{y}\right)=\boldsymbol{z}.\label{eq:LeafGraphIteration}
\end{equation}
Due to $\boldsymbol{U}_{N}\left(\boldsymbol{x}\right)=\mathcal{O}\left(\left|\boldsymbol{x}\right|^{2}\right)$,
the iteration is indeed a contraction within a sufficiently small
neighbourhood of the origin. If $\boldsymbol{U}$ is a polynomial
of order $\alpha$, and we seek $\boldsymbol{g}$ as another polynomial
of order $\alpha$, then the iteration (\ref{eq:LeafGraphIteration})
finishes in $\alpha$ steps.

We now show that $\boldsymbol{V}_{\perp}$ and $\boldsymbol{V}_{\parallel}$
can be found using singular value decomposition \cite{Linalg}. The
singular value decomposition in our case can be written as 
\begin{equation}
\begin{pmatrix}D\boldsymbol{U}\left(\boldsymbol{0}\right)\\
\boldsymbol{0}_{\left(n-\nu\right)\times n}
\end{pmatrix}=\begin{pmatrix}\boldsymbol{\varUpsilon}_{\parallel} & \boldsymbol{0}\\
\boldsymbol{0} & \boldsymbol{\varUpsilon}_{\perp}
\end{pmatrix}\begin{pmatrix}\boldsymbol{\varSigma} & \boldsymbol{0}\\
\boldsymbol{0} & \boldsymbol{0}
\end{pmatrix}\begin{pmatrix}\tilde{\boldsymbol{V}}_{\parallel} & \tilde{\boldsymbol{V}}_{\perp}\end{pmatrix}^{T},\label{eq:ULsvd}
\end{equation}
where $\boldsymbol{\varSigma}$ is diagonal, $\boldsymbol{\varUpsilon}_{\parallel}$,
$\boldsymbol{\varUpsilon}_{\perp}$ and $\begin{pmatrix}\tilde{\boldsymbol{V}}_{\parallel} & \tilde{\boldsymbol{V}}_{\perp}\end{pmatrix}^{T}$
are orthonormal matrices. We now multiply (\ref{eq:ULsvd}) by $\begin{pmatrix}\tilde{\boldsymbol{V}}_{\parallel} & \tilde{\boldsymbol{V}}_{\perp}\end{pmatrix}$
from the left to check the constraints (\ref{eq:LeafProjConstr})
and we find that 
\[
D\boldsymbol{U}\left(\boldsymbol{0}\right)\tilde{\boldsymbol{V}}_{\parallel}=\boldsymbol{\varUpsilon}_{\parallel}\boldsymbol{\varSigma},\;D\boldsymbol{U}\left(\boldsymbol{0}\right)\tilde{\boldsymbol{V}}_{\perp}=\boldsymbol{0},
\]
where $\boldsymbol{\varUpsilon}_{\parallel}\boldsymbol{\varSigma}$
is invertible. Therefore, we find that $\boldsymbol{V}_{\parallel}=\tilde{\boldsymbol{V}}_{\parallel}\left(\boldsymbol{\varUpsilon}_{\parallel}\boldsymbol{\varSigma}\right)^{-1}$
and $\boldsymbol{V}_{\perp}=\tilde{\boldsymbol{V}}_{\perp}$.
\begin{remark}
\label{rem:EparNearInvariance}Here we justify our choice of representation
(\ref{eq:LeafImmersionDef}) for lightly damped mechanical systems.
For other kinds of systems a different representation may be necessary.
First we note that the range of matrix $\boldsymbol{V}_{\perp}$ is
always invariant under the Jacobian $\boldsymbol{A}$, however the
range of $\boldsymbol{V}_{\parallel}$ (i.e. $E_{\parallel}$) is
not. The range of $\boldsymbol{V}_{\perp}$ is invariant if there
exists a matrix $\boldsymbol{P}_{\perp}$, such that $\boldsymbol{A}\boldsymbol{V}_{\perp}=\boldsymbol{V}_{\perp}\boldsymbol{P}_{\perp}$.
In the most general case, we have the decomposition 
\begin{equation}
\boldsymbol{A}\boldsymbol{V}_{\perp}=\boldsymbol{V}_{\perp}\boldsymbol{P}_{\perp}+\boldsymbol{V}_{\parallel}\boldsymbol{P}_{\parallel}.\label{eq:VperpInvariance}
\end{equation}
Applying $D\boldsymbol{U}\left(\boldsymbol{0}\right)$ from the left
to (\ref{eq:VperpInvariance}) and noticing that 
\[
D\boldsymbol{U}\left(\boldsymbol{0}\right)\boldsymbol{A}\boldsymbol{V}_{\perp}=D\boldsymbol{S}\left(\boldsymbol{0}\right)D\boldsymbol{U}\left(\boldsymbol{0}\right)\boldsymbol{V}_{\perp}=\boldsymbol{0},
\]
we find that $\boldsymbol{P}_{\parallel}=\boldsymbol{0}$, which proves
the invariance of $\boldsymbol{V}_{\perp}$. A similar calculation
can be carried out for $\boldsymbol{V}_{\parallel}$ by using the
decomposition
\begin{equation}
\boldsymbol{A}\boldsymbol{V}_{\parallel}=\boldsymbol{V}_{\perp}\boldsymbol{Q}_{\perp}+\boldsymbol{V}_{\parallel}\boldsymbol{Q}_{\parallel}.\label{eq:VparNonInvariance}
\end{equation}
Applying $\boldsymbol{V}_{\perp}^{T}$ to (\ref{eq:VparNonInvariance}),
we find that $\boldsymbol{Q}_{\perp}=\boldsymbol{V}_{\perp}^{T}\boldsymbol{A}\boldsymbol{V}_{\parallel}$.
If $\boldsymbol{V}_{\perp}$ is invariant under $\boldsymbol{A}^{T}$,
we have $\boldsymbol{Q}_{\perp}=\boldsymbol{0}$, which implies that
$E_{\parallel}$ coincides with $E$.

We now assume an undamped mechanical system with an equilibrium at
the origin, such that the Jacobian of the vector field $\dot{\boldsymbol{x}}=\boldsymbol{G}\left(\boldsymbol{x}\right)$
at the origin has the form
\begin{equation}
\boldsymbol{B}=\begin{pmatrix}\boldsymbol{0} & \boldsymbol{I}\\
-\boldsymbol{K} & \boldsymbol{0}
\end{pmatrix}=D\boldsymbol{G}\left(\boldsymbol{0}\right),\label{eq:UndampedJacobian}
\end{equation}
where the stiffness matrix $\boldsymbol{K}$ is symmetric and positive
definite. If $\left(\boldsymbol{v},\lambda\boldsymbol{v}\right)^{T}$
is a right eigenvector of $\boldsymbol{B}$, then $\left(\lambda\boldsymbol{v}^{T},\boldsymbol{v}^{T}\right)$
is a left eigenvector both corresponding to the same eigenvalue $\lambda$,
where $\boldsymbol{v}$ is a real valued vector. Therefore if the
eigenvector $\left(\boldsymbol{v},\lambda\boldsymbol{v}\right)^{T}$
being in the range of $\boldsymbol{V}_{\perp}$ implies that the eigenvector
$\left(\boldsymbol{v},\overline{\lambda}\boldsymbol{v}\right)^{T}$
is also in the range of $\boldsymbol{V}_{\perp}$, then $\boldsymbol{V}_{\perp}$
is invariant under $\boldsymbol{B}^{T}$. This is because the pair
of vectors $\left(\boldsymbol{v},\lambda\boldsymbol{v}\right)$, $\left(\boldsymbol{v},\overline{\lambda}\boldsymbol{v}\right)$
and $\left(\lambda\boldsymbol{v},\boldsymbol{v}\right)$, $\left(\overline{\lambda}\boldsymbol{v},\boldsymbol{v}\right)$
span the same linear subspace. In other words, if $\lambda_{j}^{2}-\lambda_{k}^{2}\neq0$
holds for $k=1,\cdots,\nu$ and $j=\nu+1,\cdots n$, then pairs of
complex conjugate eigenvectors are part of the range of $\boldsymbol{V}_{\perp}$,
which makes $\boldsymbol{V}_{\perp}$ invariant under $\boldsymbol{B}^{T}$
and further implies that $E_{\parallel}$ coincides with $E$. Using
the relation that $\boldsymbol{A}=D\boldsymbol{F}\left(\boldsymbol{0}\right)=\exp\boldsymbol{B}\tau$,
where $\tau$ is the sampling period, we find that if $\mu_{j}/\mu_{k}\neq1$
for $k=1,\cdots,\nu$ and $j=\nu+1,\cdots n$, then $E_{\parallel}=E$.
If light damping is introduced, into the mechanical system, such that
\[
\boldsymbol{B}=\begin{pmatrix}\boldsymbol{0} & \boldsymbol{I}\\
-\boldsymbol{K} & -\boldsymbol{C}
\end{pmatrix},
\]
where $\left\Vert C\right\Vert $ is small, $E_{\parallel}$ remains
close to $E$ due to the continuity of eigenvectors with respect to
the underlying matrix.
\end{remark}

\subsection{\label{subsec:BackboneCurves}The backbone and damping curves of
an ISF}

We can accurately identify the dynamics on an ISF and determine its
instantaneous damping ratio (\ref{eq:FITdamping}) and angular frequency
(\ref{eq:FITfrequency}). It is however not possible to attach a unique
amplitude to a leaf within a foliation. In this section we go around
this restriction and define a surrogate for the amplitude, which measures
the distance of a leaf from the equilibrium. This is extracted purely
from the submersion $\boldsymbol{U}$, therefore it will not measure
the amplitude, but some approximation of it as explained in remark
\ref{rem:LightlyDampedBackbone}.

In section \ref{subsec:ReducedDynamics} we have parametrised the
ISF in polar coordinates as $\boldsymbol{z}=\left(r\cos\theta,r\sin\theta\right)$.
Then in section \ref{subsec:LeavesCalc} we described the leaves of
an ISF as an immersion. Picking a point on the leaf $\mathcal{L}_{\boldsymbol{z}}$
and taking its norm can act as an instantaneous amplitude. The simplest
option is to pick the intersection point $\mathcal{L}_{\boldsymbol{z}}\cap E_{\parallel}$,
which is $\boldsymbol{W}_{\boldsymbol{z}}\left(\boldsymbol{0}\right)$
as per definition (\ref{eq:LeafImmersionDef}) and illustrated in
figure \ref{fig:LeafImmersion}. Using the same polar parametrisation
that describes the instantaneous natural frequency and damping, we
define our surrogate for the amplitude as
\begin{equation}
\Delta_{E^{\star}}\left(r\right)=\sup_{\theta\in[0,2\pi)}\left|\boldsymbol{W}_{\left(r\cos\theta,r\sin\theta\right)}\left(\boldsymbol{0}\right)\right|.\label{eq:SurrogateAmplitude}
\end{equation}

\begin{definition}
We call the parametrised curve 
\begin{equation}
\mathscr{B}_{E^{\star}}=\left\{ \omega_{E^{\star}}\left(r\right),\Delta_{E^{\star}}\left(r\right):0\le r<r_{\mathrm{max}}\right\} \label{eq:ISFbackbone}
\end{equation}
the \emph{ISF backbone curve} of the dynamics associated with the
codimension-two ISF corresponding to the linear subspace $E^{\star}$.
\end{definition}
We can similarly construct a curve that describes instantaneous damping.
\begin{definition}
We call the parametrised curve 
\begin{equation}
\mathscr{D}_{E^{\star}}=\left\{ \zeta_{E^{\star}}\left(r\right),\Delta_{E^{\star}}\left(r\right):0\le r<r_{\mathrm{max}}\right\} \label{eq:ISFdamping}
\end{equation}
the \emph{ISF damping curve} of the dynamics associated with the codimension-two
ISF corresponding to the linear subspace $E^{\star}$.
\end{definition}
If a full set of ISFs are calculated that satisfy the conditions (\ref{eq:IntersectCond}),
and one is willing to solve equation (\ref{eq:RestoreDef}) for the
function $\boldsymbol{h}$ or its values for a set of arguments, then
the SSM backbone and damping curves can also be calculated. The amplitude
of a vibration represented by the conjugate dynamics $\boldsymbol{S}^{j}$
on the corresponding SSM is given by 
\[
\Delta_{E_{j}}\left(r\right)=\sup_{\theta\in[0,2\pi)}\left|\boldsymbol{W}^{j}\left(r\cos\theta,r\sin\theta\right)\right|,
\]
where $\boldsymbol{W}^{j}$ is defined by (\ref{eq:SSMimmersion}).
This allows us to make the following definitions.
\begin{definition}
We call the parametrised curve 
\begin{equation}
\mathscr{B}_{E_{j}}=\left\{ \omega_{E_{j}}\left(r\right),\Delta_{E_{j}}\left(r\right):0\le r<r_{\mathrm{max}}\right\} \label{eq:SSMbackbone}
\end{equation}
the \emph{SSM backbone curve} of the dynamics associated with the
two-dimensional SSM corresponding to the linear subspace $E$.
\end{definition}
We can similarly construct a curve that describes instantaneous damping.
\begin{definition}
We call the parametrised curve 
\begin{equation}
\mathscr{D}_{E_{j}}=\left\{ \zeta_{E_{j}}\left(r\right),\Delta_{E_{j}}\left(r\right):0\le r<r_{\mathrm{max}}\right\} \label{eq:SSMdamping}
\end{equation}
the \emph{SSM damping curve} of the dynamics associated with the two-dimensional
SSM corresponding to the linear subspace $E$.
\end{definition}
\begin{remark}
The backbone and damping curves are not unique, they depend on the
choice of parametrisation of the ISF or SSM. This is illustrated by
the fact that $\boldsymbol{S}$ can be chosen linear if there are
no internal resonances in the strict sense of (\ref{eq:MapInternalResonance}),
which is the case of most damped systems. For linear $\boldsymbol{S}$
the damping and backbone curves are straight lines, which is not the
expected result for a nonlinear system. In \cite{Szalai20160759}
a special parametrisation was chosen, such that all near resonances
are fully represented in the conjugate dynamics on the SSM, which
made the backbone curves unique. We use an equivalent normalisation
in the optimisation problem (\ref{eq:SigmaLeastSquares}), which results
in unique backbone and damping curves. However the alternative normalising
loss function (\ref{eq:FITDUnormLoss}) can leave near internally
resonant terms in the submersion $\boldsymbol{U}$, which leads to
non-unique representations of the unique ISF. The amount of variation
in the submersion $\boldsymbol{U}$ and map $\boldsymbol{S}$ can
be reduced if during optimisation various terms of the submersion
$\boldsymbol{U}$ assume similar magnitudes as the nonlinear terms
of $\boldsymbol{S}$. This strategy leads to smaller variations as
the linear damping vanishes and the near internal resonances are getting
closer to strict internal resonances. Therefore the uncertainty in
the location of the backbone curve will also vanish as damping vanishes,
making the backbone curve unique in the limit, if the limit exists.
We must stress that this argument only mentions the linear damping,
that is, only $\zeta_{E^{\star}}\left(0\right)\to0$ is assumed, therefore
the damping curve need not vanish.
\end{remark}
\begin{remark}
\label{rem:LightlyDampedBackbone}In general, there is no connection
between the ISF and SSM backbone curves, except for lightly damped
mechanical systems. According to remark \ref{rem:EparNearInvariance}
for undamped mechanical systems $E$ and $E_{\parallel}$ coincide,
hence due to the construction of $\boldsymbol{W}_{\boldsymbol{z}}$,
the surrogate amplitude $\Delta_{E^{\star}}$ is linearly asymptotic
to the SSM amplitude $\Delta_{E_{j}}$ at the equilibrium. If small
damping is introduced, $E$ and $E_{\parallel}$ remain close to each
other. This implies that $\Delta_{E^{\star}}$ remains nearly linearly
asymptotic to the SSM amplitude $\Delta_{E_{j}}$, and the ISF and
SSM backbone curves stay close to each other near the equilibrium.
\end{remark}

\section{Examples}

We illustrate the application of the theory on two examples, one based
on a mathematical model, the other is purely data driven.

\subsection{\label{subsec:ShawPierre}Shaw-Pierre example}

We use a modified two-degree-of-freedom oscillator studied by Shaw
and Pierre \cite{ShawPierre}, which has appeared in \cite{Szalai20160759}.
The modification makes the damping matrix proportional to the stiffness
matrix in the linearised problem. The first-order equations of motion
are 
\begin{equation}
\left.\begin{array}{rl}
\dot{x}_{1} & =v_{1},\\
\dot{x}_{2} & =v_{2},\\
\dot{v}_{1} & =-cv_{1}-k_{0}x_{1}-\kappa x_{1}^{3}-k_{0}(x_{1}-x_{2})-c(v_{1}-v_{2}),\\
\dot{v}_{2} & =-cv_{2}-k_{0}x_{2}-k_{0}(x_{2}-x_{1})-c(v_{2}-v_{1}).
\end{array}\right\} \label{eq:ShawPierreModel}
\end{equation}
where the parameters are $c=0.003$, $k_{0}=1$, and $\kappa=0.5$.
The natural frequencies and damping ratios are 
\[
\omega_{1}=\sqrt{k_{0}},\qquad\omega_{2}=\sqrt{3k_{0}},\qquad\zeta_{1}=\frac{c}{2\sqrt{k_{0}}},\qquad\zeta_{2}=\frac{\sqrt{3}c}{2\sqrt{k_{0}}},
\]
yielding the complex eigenvalues 
\[
\lambda_{1,2}=-\frac{c}{2}\pm i\sqrt{k_{0}\left(1-\frac{c^{2}}{4k_{0}}\right)},\quad\lambda_{3,4}=-\frac{3c}{2}\pm i\sqrt{3k_{0}\left(1-\frac{3c^{2}}{4k_{0}}\right)},
\]
where we have assumed that both modes are underdamped, i.e., $c<2\sqrt{k_{0}/3}.$
The spectral quotients corresponding to these natural frequencies
are
\begin{align*}
\beth_{E_{1}^{\star}} & =1, & \beth_{E_{2}^{\star}}=3.
\end{align*}

The data for this problem was generated from 100 trajectories of 16
points each with time step $T=0.8$. The initial conditions for each
trajectory was uniformly drawn from a cube of width $0.4$ about the
origin and scaled, such that $\boldsymbol{x}_{0}\mapsto\boldsymbol{x}_{0}/\left|\boldsymbol{x}_{0}\right|^{2}$.
This ensures a higher density of data about the origin and that $\max\left|\boldsymbol{x}_{k}\right|\le0.2$.
Testing data was also created by the same procedure in order to check
whether we overfit the data. The fitting procedure used $\sigma=2$
and $\sigma=3$ values with order 3, 5 and 7 polynomials representing
the submersion $\boldsymbol{U}$ and dynamics $\hat{\boldsymbol{S}}$
in equation (\ref{eq:PolyObjFunction}). The optimisation was carried
out using the first-order BFGS method. The parameters for the penalty
term (\ref{eq:FITnormaliseLoss-1}) were $N_{r}=10$, $N_{\theta}=24$
and $r_{\mathrm{max}}=0.2$. The accuracy of fitting can be seen in
table \ref{fig:SPresidual}, which also shows that as the order of
polynomials grows, the ratio between of testing and training residual
slightly increases.
\begin{table}
\begin{centering}
\begin{tabular}{l|l|l|l|l}
 & training $E_{1}$ & training $E_{2}$ & testing $E_{1}$ & testing $E_{2}$\tabularnewline
\hline 
DATA O(3) $\sigma=2$ & $1.1800\times10^{-5}$ & $3.6622\times10^{-5}$ & $1.5712\times10^{-5}$ & $4.5403\times10^{-5}$\tabularnewline
DATA O(3) $\sigma=3$ & $1.2877\times10^{-5}$ & $3.7610\times10^{-5}$ & $1.7158\times10^{-5}$ & $4.9812\times10^{-5}$\tabularnewline
DATA O(5) $\sigma=2$ & $3.7609\times10^{-6}$ & $9.3560\times10^{-6}$ & $6.3557\times10^{-6}$ & $1.4281\times10^{-5}$\tabularnewline
DATA O(5) $\sigma=3$ & $4.2710\times10^{-7}$ & $4.1703\times10^{-6}$ & $1.1541\times10^{-6}$ & $1.0405\times10^{-5}$\tabularnewline
DATA O(7) $\sigma=2$ & $4.0612\times10^{-6}$ & $9.7153\times10^{-6}$ & $6.7263\times10^{-6}$ & $1.5472\times10^{-5}$\tabularnewline
DATA O(7) $\sigma=3$ & $8.3854\times10^{-8}$ & $6.4913\times10^{-7}$ & $5.1314\times10^{-7}$ & $3.2731\times10^{-6}$\tabularnewline
\end{tabular}
\par\end{centering}
\caption{\label{fig:SPresidual}The residual of the fitting procedure is calculated
as $\mathit{res}=\frac{1}{N}\sum_{k=1}^{N}\left|\boldsymbol{x}_{k}\right|^{-1}\left|\boldsymbol{U}\left(\boldsymbol{y}_{k}\right)-\boldsymbol{S}\left(\boldsymbol{U}\left(\boldsymbol{x}_{k}\right)\right)\right|$,
which are compared for the training and testing data. DATA O($n$)
means that order-$n$ polynomial was fitted to the generated data.}

\end{table}

In figure \ref{fig:ShawPierreBackbone} various ISF backbone and damping
curves are compared to each other and to the SSM backbone and damping
curves. We treat the order-7 SSM calculation as a reference. It can
be seen that the ISF backbone curves are very close to the SSM backbone
curve. The ISF damping curves seemingly display a larger variation,
however that is due to the scale of the horizontal axis, the relative
error is small.

\begin{figure}
\begin{centering}
\includegraphics[width=0.45\linewidth]{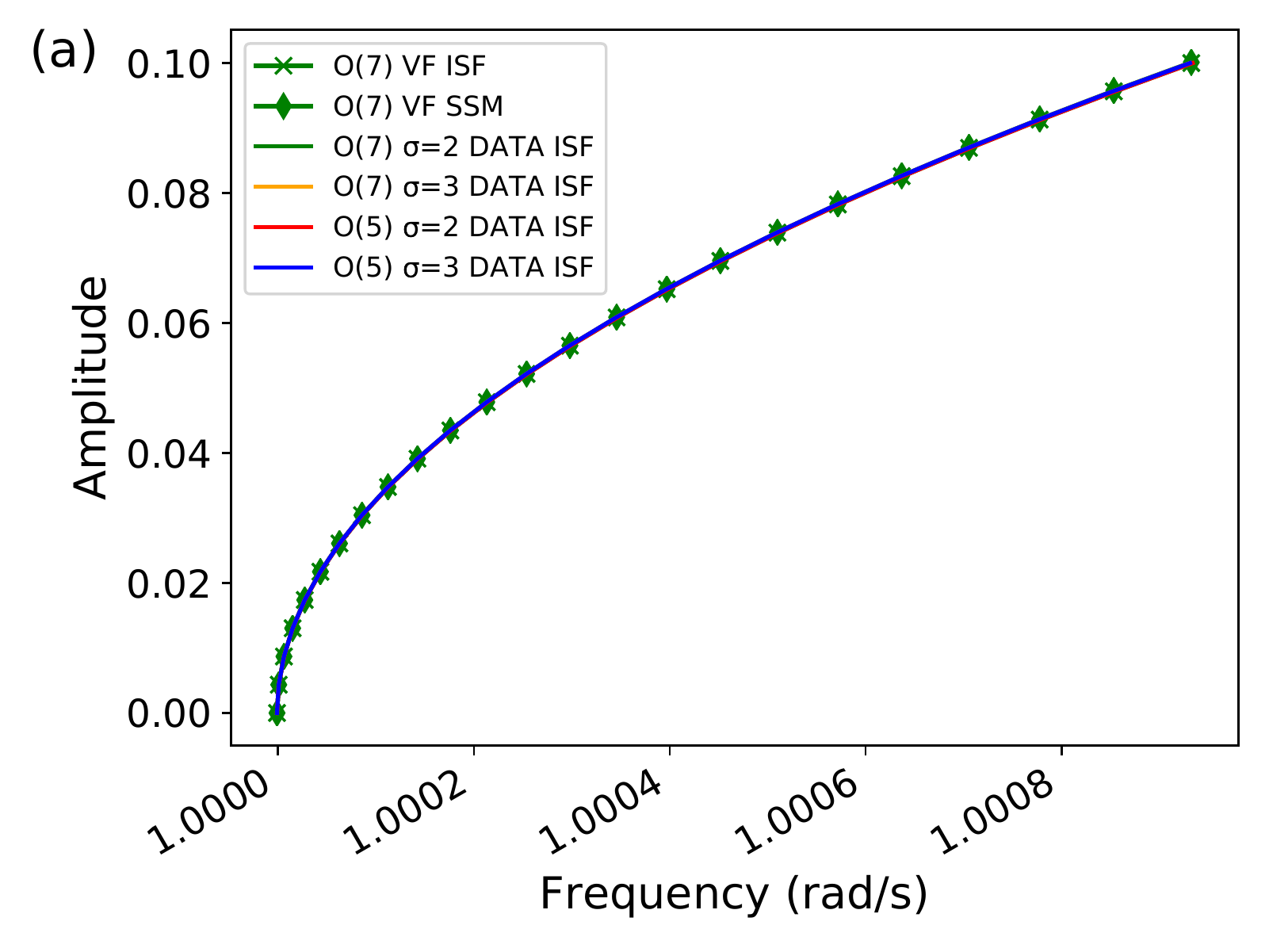}\includegraphics[width=0.45\linewidth]{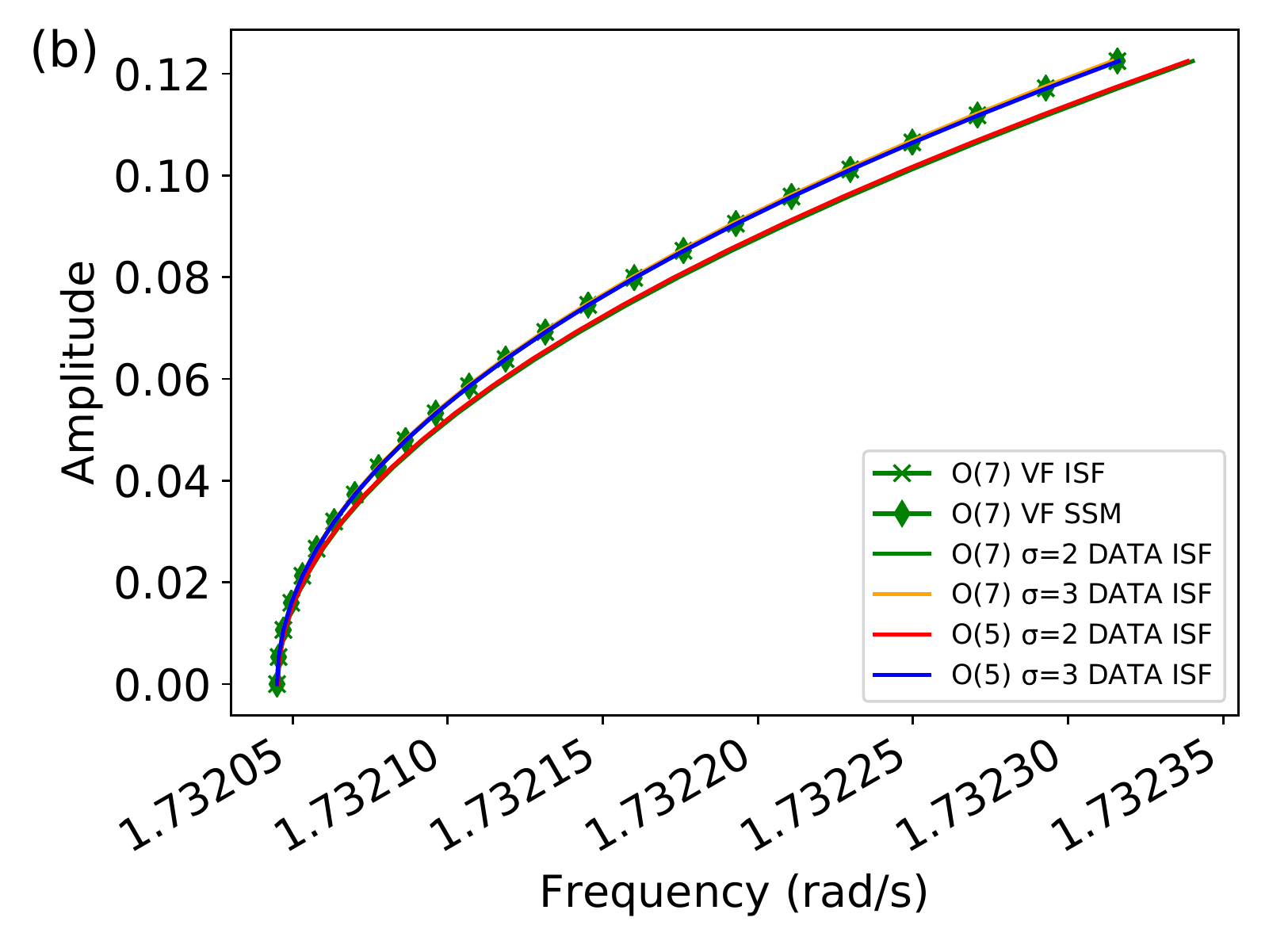}\\
\includegraphics[width=0.45\linewidth]{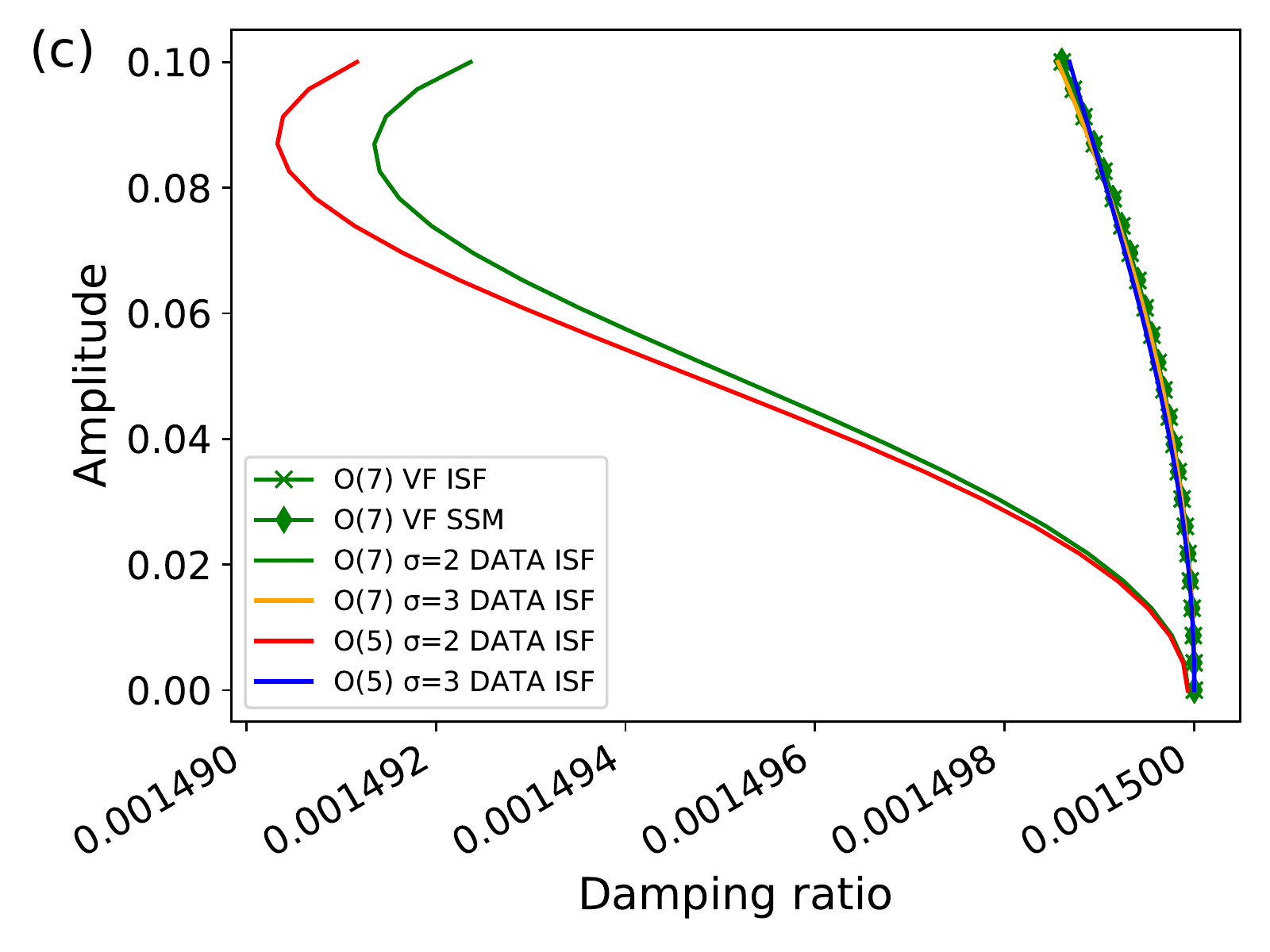}\includegraphics[width=0.45\linewidth]{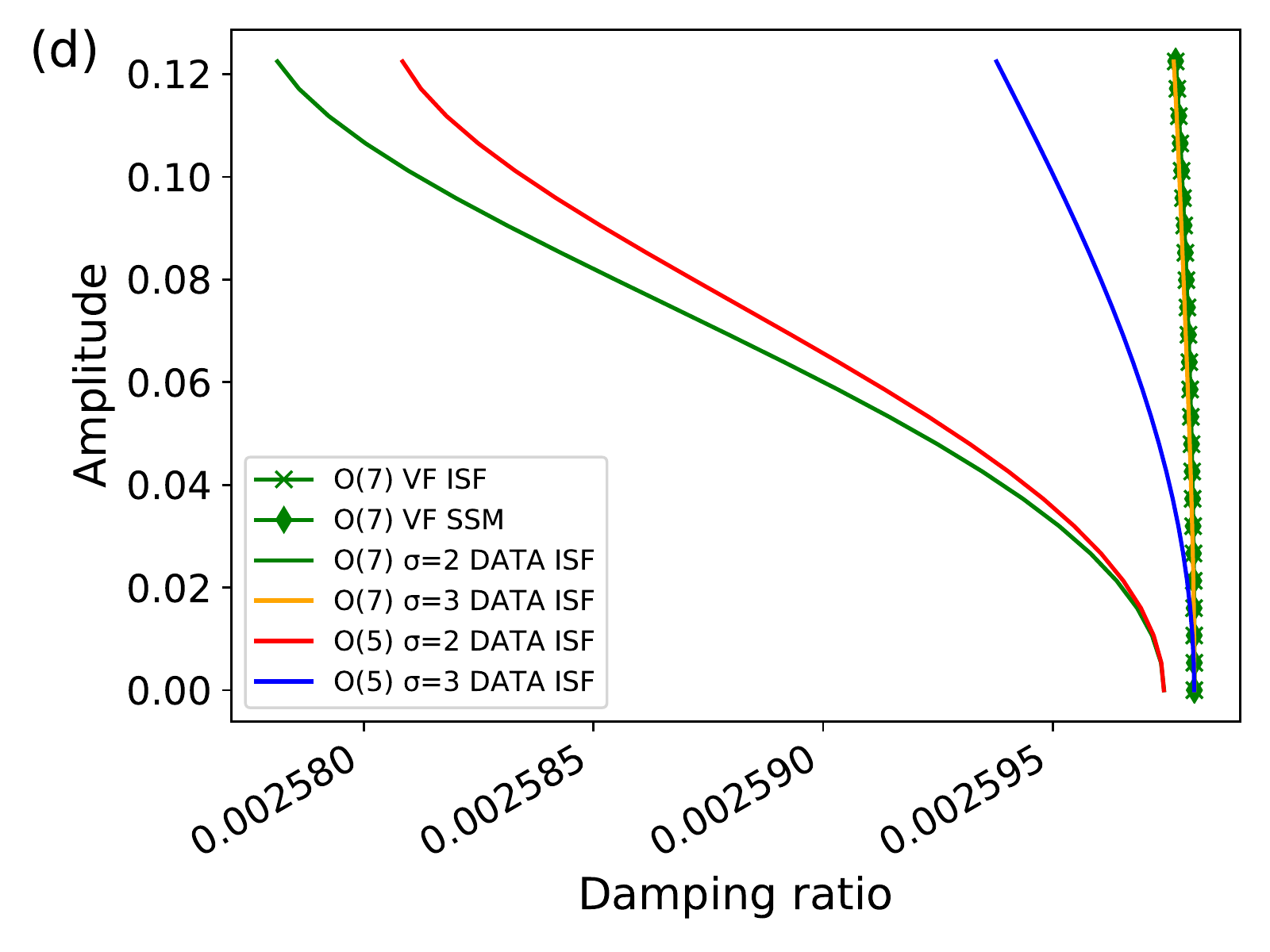}
\par\end{centering}
\caption{\label{fig:ShawPierreBackbone}Backbone and damping curves of equation
(\ref{eq:ShawPierreModel}). The curves were identified using SSMs,
series expanded ISFs calculated from the vector field and identified
from data. The relative error of the backbone and damping curves are
roughly the same, but due to the scaling of the figure the damping
curves appear less accurate. VF means that the vector field was used
to calculate the result, DATA means that the ISF was directly fitted
to the data, $O(\alpha)$ indicates that order $\alpha$ polynomials
were used.}
\end{figure}

The calculated ISFs can be used to reconstruct the full dynamics.
The accuracy of this reconstruction is illustrated in figure \ref{fig:ShawPierreAccuracy}
for a single trajectory. We compare the sampled trajectory $\boldsymbol{x}_{k}=\boldsymbol{x}\left(kT\right)$,
which is the solution of the differential equation (\ref{eq:ShawPierreModel})
to the reconstructed dynamics using the map $\boldsymbol{h}$ as defined
by (\ref{eq:RestoreDef}). The initial conditions for the reduced
order models are set by $\boldsymbol{z}_{j,0}=\boldsymbol{U}^{j}\left(\boldsymbol{x}_{0}\right)$
and then iterated under the reduced models, such that $\boldsymbol{z}_{j,k+1}=\boldsymbol{S}^{j}\left(\boldsymbol{z}_{j,k}\right)$.
First, we evaluate the inaccuracies of the fitting of the invariance
equation by
\begin{equation}
\mathrm{err}_{k}^{\mathit{fw}}=\left|\boldsymbol{x}_{k}\right|^{-1}\left|\left(\boldsymbol{z}_{1,k}-\boldsymbol{U}^{1}\left(\boldsymbol{x}_{k}\right),\boldsymbol{z}_{2,k}-\boldsymbol{U}^{2}\left(\boldsymbol{x}_{k}\right)\right)\right|,\label{eq:FWerror-1}
\end{equation}
where the subscript $\mathit{fw}$ refers to forward prediction. The
result of this can be seen in figure \ref{fig:ShawPierreAccuracy}(a).
Second, we use equation (\ref{eq:RedToFull}) to reconstruct the dynamics
from the two ISFs and compare the reconstructed trajectories to the
solution of the differential equation (\ref{eq:ShawPierreModel}).
The relative reconstruction error is calculated as
\begin{equation}
\mathrm{err}_{k}^{\mathit{bw}}=\left|\boldsymbol{x}_{k}\right|^{-1}\left|\boldsymbol{x}_{k}-\boldsymbol{h}\left(\boldsymbol{z}_{1,k},\boldsymbol{z}_{2,k}\right)\right|\label{eq:BWerror-1}
\end{equation}
and illustrated in figure \ref{fig:ShawPierreAccuracy}(b). When comparing
figures \ref{fig:ShawPierreAccuracy}(a) and \ref{fig:ShawPierreAccuracy}(b),
one can see that the accuracy of satisfying the invariance equation
is better than the accuracy of the reconstruction, which is due to
the added inaccuracy of the post-processing step that produces the
map $\boldsymbol{h}$. It is also clear that the error in the invariance
equation increases about one order of magnitude over the 32 steps
of the comparison, while the reconstruction error remains roughly
constant at least for order-3 and 5 polynomials. We note that the
described behaviour is consistent with other trajectories, however
the absolute magnitude of the errors will increase as $\left|\boldsymbol{x}_{0}\right|$
increases. The dependence of the errors on $\left|\boldsymbol{x}_{0}\right|$
can be controlled by the value of $\sigma$. However, the error also
depends on the distribution of the data within the state space, which
we may not have control over. We note that the errors in figure \ref{fig:ShawPierreAccuracy}(b)
can be reduced to the errors displayed in figure \ref{fig:ShawPierreAccuracy}(a)
if equation (\ref{eq:RestoreDef}) is solved using a Newton's method
instead of the iteration (\ref{eq:InverseSubmersion}) (data not shown
as it is indistinguishable from figure \ref{fig:ShawPierreAccuracy}(a)).
\begin{figure}
\begin{centering}
\includegraphics[width=0.45\linewidth]{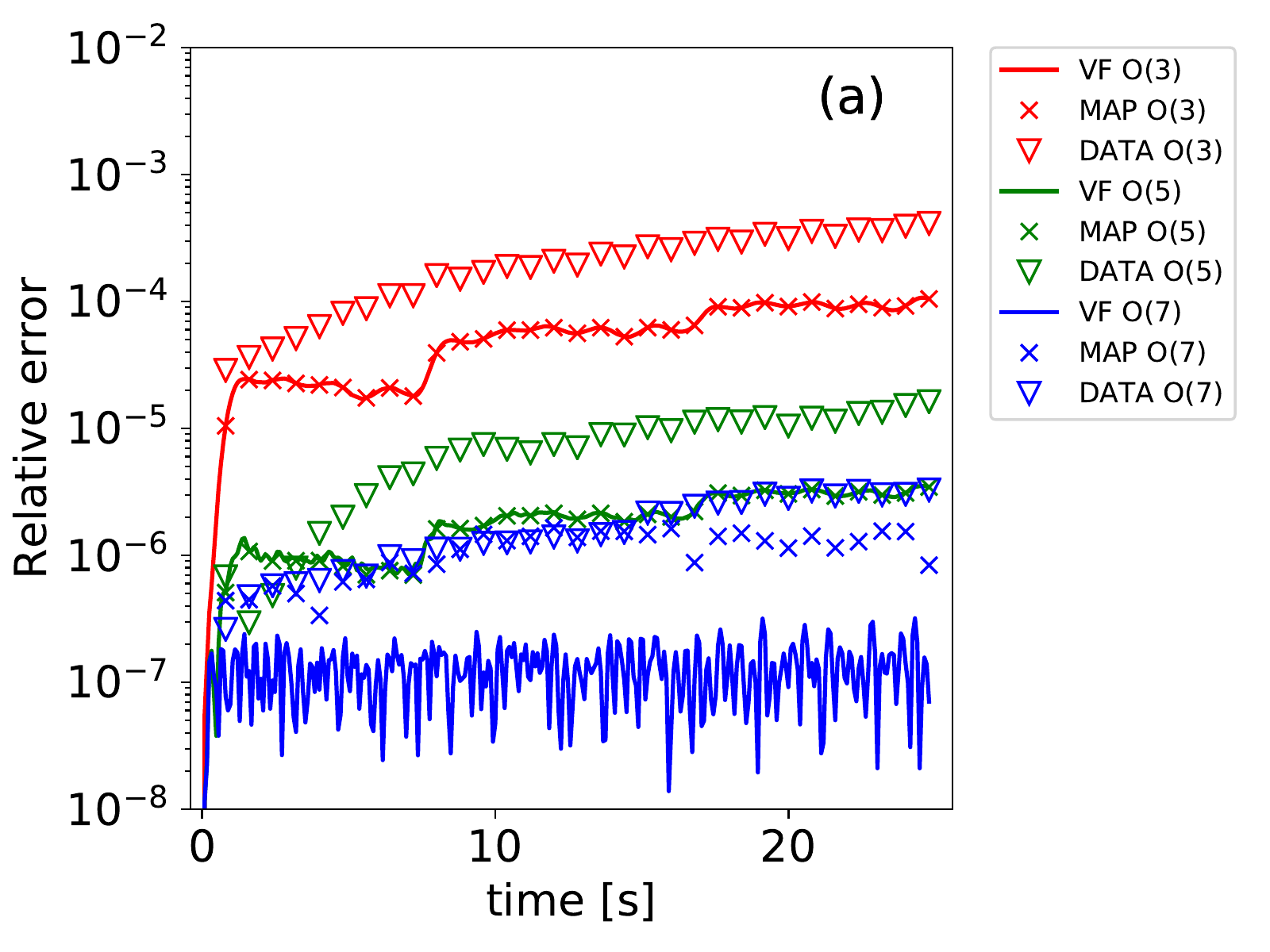}\includegraphics[width=0.45\linewidth]{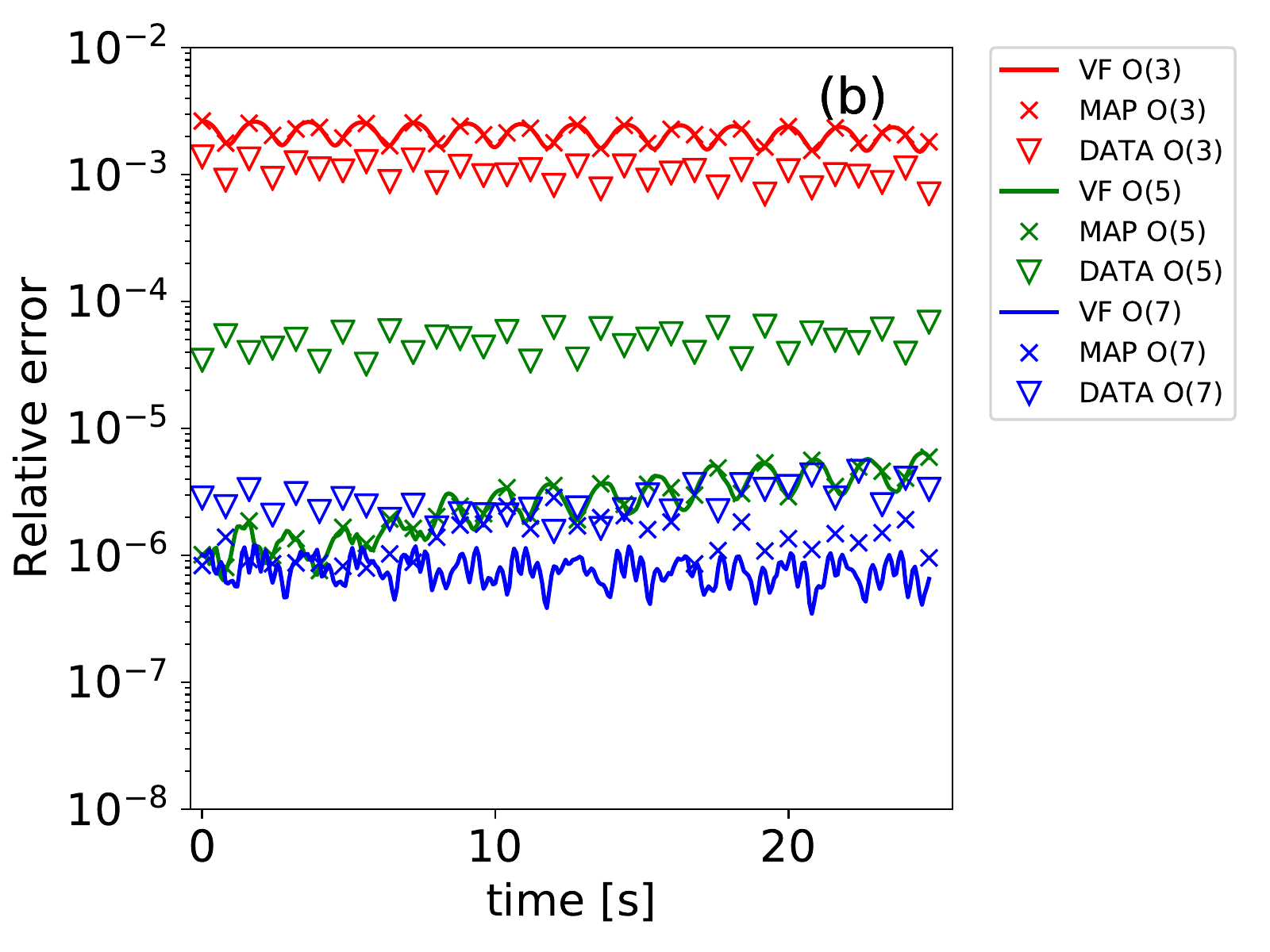}
\par\end{centering}
\caption{\label{fig:ShawPierreAccuracy}Reconstruction error as a function
of time. Solid lines correspond to the ISF obtained directly from
the vector field (\ref{eq:ShawPierreModel}). The $\times$ markers
denote the error from the ISFs of the identified map and the $\triangledown$
corresponds to the directly identified ISFs. The same comparison is
carried out for orders $\alpha=3$ (black) $\alpha=5$ (green) and
$\alpha=7$ (red) polynomial expansions. The scaling order parameter
is $\sigma=3$ and the initial condition is\textbf{ $\boldsymbol{x}=\left(1.1088\times10^{-4},1.9023\times10^{-5},-0.0739,-0.0126\right)$}.
(a) errors of reconstruction using equation (\ref{eq:FWerror-1})
and (b) using (\ref{eq:BWerror-1}).}
\end{figure}

\subsection{Clamped-clamped beam}

Here we analyse the free-decay vibration of a clamped-clamped beam.
The data was collected by Ehrhardt and Allen \cite{EHRHARDT2016612}
using the device depicted in figure \ref{fig:CCBeam}. The data contains
three tracks of velocity information, measured at the midpoint of
the beam, which correspond to the first three vibration modes of the
structure. The initial conditions were set by applying a carefully
tuned forcing that compensates for the damping within the structure,
and intends to recover the sustained vibration that would have occurred
if the structure did not have damping. Such vibration is thought to
be near an SSM \cite{Szalai20160759}, which makes it unusual as impact
hammer tests would not single out specific modes of vibration. The
data was re-sampled with time period $T=0.97656$ ms. We use the same
phase-space reconstruction through delay-embedding of velocity data
as in \cite{Szalai20160759}, where full justification is given for
the choice of phase space dimensionality.

We have fitted ISFs to all three modes of vibration captured by the
data, however we only show the first and third backbone curves, which
can be compared to the analysis in \cite{EHRHARDT2016612}. We have
used order 3 polynomials for the submersion $\boldsymbol{U}$, order
3, 5 and 7 polynomials for the conjugate map $\boldsymbol{S}$ and
set $\sigma=1$ throughout the calculation. Setting a higher value
of $\sigma$ would over-emphasise the importance of the data near
the equilibrium and therefore the backbone curves would follow less
accurately the actual frequency variations at higher amplitudes. We
have found that using higher order polynomials for the submersion
$\boldsymbol{U}$ makes the composite map $\widehat{\boldsymbol{U}}$
of equation (\ref{eq:RestoreDef}) non-invertible close to the equilibrium,
which is a likely symptom of over fitting. The parameters for the
penalty term (\ref{eq:FITnormaliseLoss-1}) were $N_{r}=12$, $N_{\theta}=24$
and $r_{\max}=0.7$. The residuals of the fitting process are gathered
in table \ref{tab:CCbeamResiduals}. When a polynomial model is first
fitted to the data, just as in \cite{Szalai20160759}, and the ISF
is directly calculated from the model, the residuals are high, because
the ISF only asymptomatically satisfies the invariance equation \ref{eq:adjInvariance}
about the equilibrium with respect to the fitted model. When the ISF
is directly fitted to the data, the loss function (\ref{eq:PolyObjFunction})
is minimised, which is closely related to the residual.

The fitting procedure also recovers the natural frequencies and the
damping ratios of the linear modes of the structure. The identified
natural frequencies can be seen in table \ref{tab:CCbeamFreqs}, which
show very little variation from the linearly identified values in
\cite{EHRHARDT2016612}. The damping ratios in table \ref{tab:CCbeamDamp}
show a wider variation, and there seems to be a systematic error of
a factor $2\ldots3$. The ISF spectral quotients are also shown in
table \ref{tab:CCbeamDamp}, which indicate that all ISFs are unique
if they are twice differentiable, when considering the results of
the order 5 and 7 fittings.
\begin{figure}[th]
\begin{centering}
\includegraphics[width=0.8\linewidth]{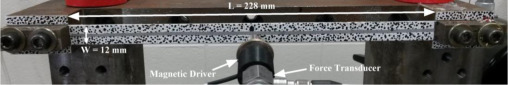}
\par\end{centering}
\caption{The clamped-clamped beam, whose free-vibration was measured, which
in turn was used to carry out our analysis. Reproduced from \cite{EHRHARDT2016612}.\label{fig:CCBeam}}
\end{figure}

\begin{table}
\begin{centering}
\begin{tabular}{c|c|c|c|}
 & mode 1 & mode 2 & mode 3\tabularnewline
\hline 
MAP O(3) & $7.3574\times10^{4}$  & $8.3796\times10^{2}$  & $2.2908\times10^{2}$ \tabularnewline
DATA O(3) & $4.0804\times10^{-2}$  & $1.6425\times10^{-2}$  & $6.6652\times10^{-3}$ \tabularnewline
DATA O(5) & $2.4294\times10^{-2}$  & $1.0230\times10^{-2}$  & $3.6994\times10^{-3}$ \tabularnewline
DATA O(7) & $1.9384\times10^{-2}$  & $8.7304\times10^{-3}$  & $2.8668\times10^{-3}$ \tabularnewline
\end{tabular}
\par\end{centering}
\caption{\label{tab:CCbeamResiduals}Residuals of the fitting process, calculated
as $\mathit{res}=\frac{1}{N}\sum_{k=1}^{N}\left|\boldsymbol{x}_{k}\right|^{-1}\left|\boldsymbol{U}\left(\boldsymbol{y}_{k}\right)-\boldsymbol{S}\left(\boldsymbol{U}\left(\boldsymbol{x}_{k}\right)\right)\right|$.
O($\alpha$) means that the conjugate map $\boldsymbol{S}$ is an
order-$\alpha$ polynomial, while the submersion $\boldsymbol{U}$
is always an order-3 polynomial.}
\end{table}
\begin{table}

\begin{centering}
\begin{tabular}{c|c|c|c|}
 & $\omega_{1}$ & $\omega_{2}$ & $\omega_{3}$\tabularnewline
\hline 
MAP O(3) & $2.8644\times10^{2}$  & $1.0466\times10^{3}$  & $2.3124\times10^{3}$ \tabularnewline
DATA O(3) & $2.9854\times10^{2}$  & $1.0423\times10^{3}$  & $2.3148\times10^{3}$ \tabularnewline
DATA O(5) & $2.8714\times10^{2}$  & $1.0431\times10^{3}$  & $2.3121\times10^{3}$ \tabularnewline
DATA O(7) & $2.8561\times10^{2}$  & $1.0433\times10^{3}$  & $2.3115\times10^{3}$ \tabularnewline
Ref \cite{EHRHARDT2016612} & $2.8777\times10^{2}$ & $1.0782\times10^{3}$ & $2.3354\times10^{3}$\tabularnewline
\end{tabular}
\par\end{centering}
\caption{\label{tab:CCbeamFreqs}Natural frequencies of the three ISFs are
compared to the estimates in \cite{EHRHARDT2016612}. MAP means a
polynomial model fit was carried out first, DATA means that the ISF
was directly fitted to the data, $O(\alpha)$ indicates that order-$\alpha$
polynomials were used for the map $\boldsymbol{S}$.}

\end{table}
\begin{table}
\begin{centering}
\begin{tabular}{c|c|c|c|c|c|c|}
 & $\zeta_{1}$ & $\zeta_{2}$ & $\zeta_{3}$ & $\beth_{1}$ & $\beth_{2}$ & $\beth_{3}$\tabularnewline
\hline 
MAP O(3) & $4.0957\times10^{-2}$  & $1.0379\times10^{-2}$  & $1.6871\times10^{-3}$  & $3.007$  & $2.784$  & $1.000$\tabularnewline
DATA O(3) & $2.9466\times10^{-4}$  & $2.6684\times10^{-3}$  & $1.9934\times10^{-3}$  & $1.000$  & $31.62$  & $52.45$\tabularnewline
DATA O(5) & $1.3519\times10^{-2}$  & $3.4533\times10^{-3}$  & $1.8234\times10^{-3}$  & $1.078$  & $1.000$  & $1.170$\tabularnewline
DATA O(7) & $1.3930\times10^{-2}$  & $3.3550\times10^{-3}$  & $1.7797\times10^{-3}$  & $1.137$  & $1.000$  & $1.175$\tabularnewline
Ref \cite{EHRHARDT2016612} & $3.8\times10^{-3}$ & $1.2\times10^{-3}$ & $9.0\times10^{-4}$ & $1.127$ & $1.333$ & $1.0$\tabularnewline
\end{tabular}
\par\end{centering}
\caption{\label{tab:CCbeamDamp}Damping ratios and ISF spectral quotients estimated
by polynomial fitting are compared to \cite{EHRHARDT2016612}. MAP
means a polynomial model fit was carried out first, DATA means that
the ISF was directly fitted to the data, $O(\alpha)$ indicates that
order $\alpha$ polynomials were used for the map $\boldsymbol{S}$.}
\end{table}

Using the fitted ISFs, we have calculated the backbone curves corresponding
to the first and third vibration modes. The ISF calculations are compared
to the force appropriation results and free decay analysis of \cite{EHRHARDT2016612},
denoted by 'Forcing' and 'Decay' in figure \ref{fig:CCbeamBackbone},
respectively. We have also calculated the SSMs and ISFs indirectly,
from a third order polynomial model that is fitted to the data, as
in \cite{Szalai20160759}, which is denoted by 'MAP' in figure \ref{fig:CCbeamBackbone}.
To obtain the backbone curves, we have applied the post-processing
steps in sections \ref{subsec:LeavesCalc} and \ref{subsec:BackboneCurves}.
In figures \ref{fig:CCbeamBackbone}(a,b) the leaves of the foliation
were recovered as polynomials as described in section \ref{subsec:LeavesCalc},
however in figures \ref{fig:CCbeamBackbone}(c,d), equation (\ref{eq:LeafGraphPreIt})
was solved for $\boldsymbol{g}$ in a pointwise manner with fixed
$\boldsymbol{y},\boldsymbol{z}$ values using Newton's method, which
gives accurate results. In figures \ref{fig:CCbeamBackbone}(e,f),
we have used Newton's method to solve the even more accurate equation
(\ref{eq:RestoreDef}), and calculated the SSM backbone curves from
the collection of three ISFs. It can be seen in figures \ref{fig:CCbeamBackbone}(c,e),
that Newton's method is unable to find a solution for higher vibration
amplitudes in the vicinity of previous iterations. This indicates
for figure \ref{fig:CCbeamBackbone}(c) that some leaves of the foliation
become tangential to $E_{\parallel}$ (defined by (\ref{eq:Eparallel}))
or the leaves of the three ISFs do not always intersect, in case of
figure \ref{fig:CCbeamBackbone}(e). We believe that the latter problem
can be partly blamed on the lack of data outside the neighbourhoods
of the three SSMs. The fitting method is arbitrarily picking the submersions
in these regions of the phase space, which can be highly distorted.
This problem was not encountered in section \ref{subsec:ShawPierre},
where the data was better distributed in the phase space.

\begin{figure}
\begin{centering}
\includegraphics[width=0.49\linewidth]{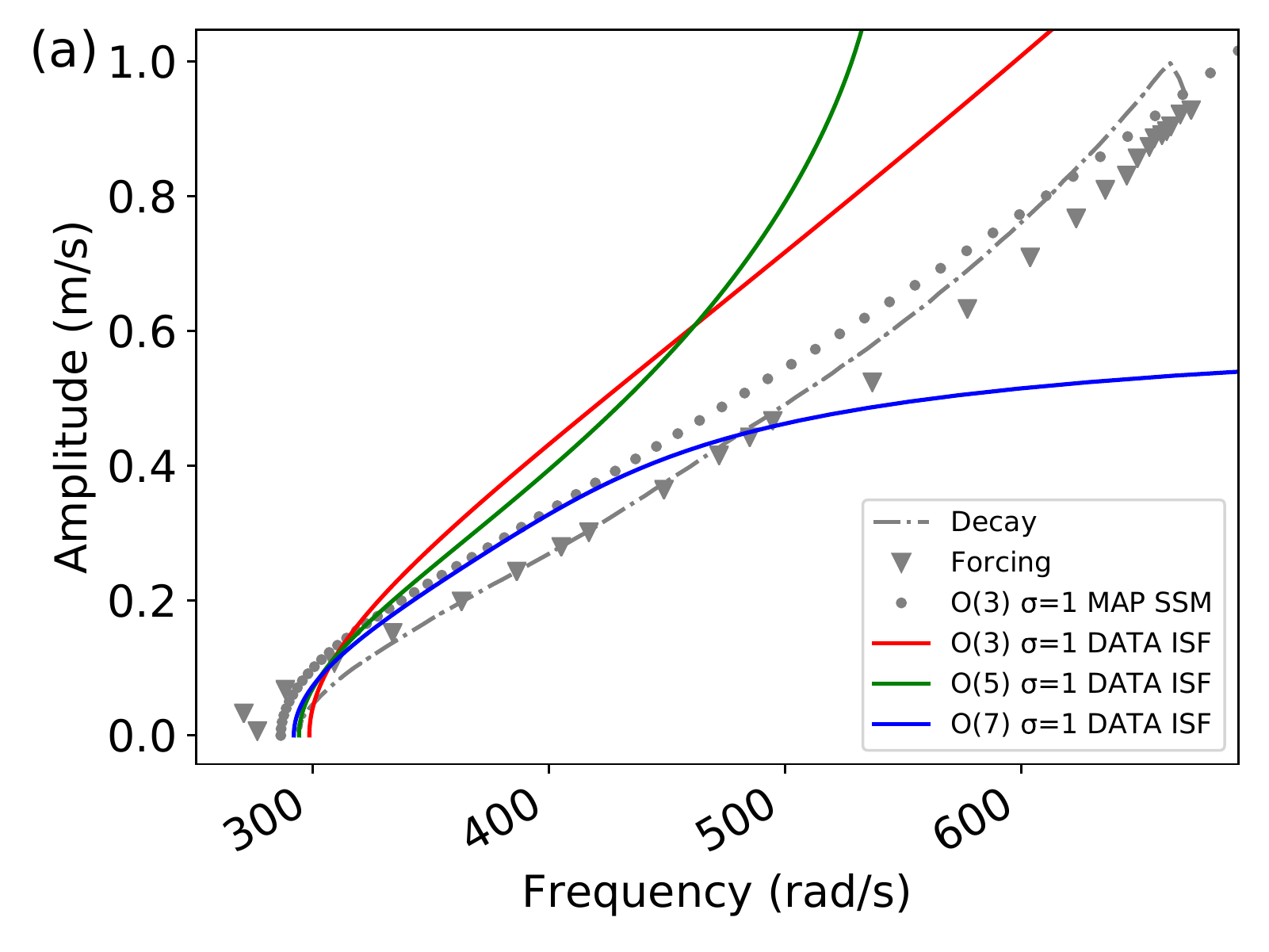}\includegraphics[width=0.49\linewidth]{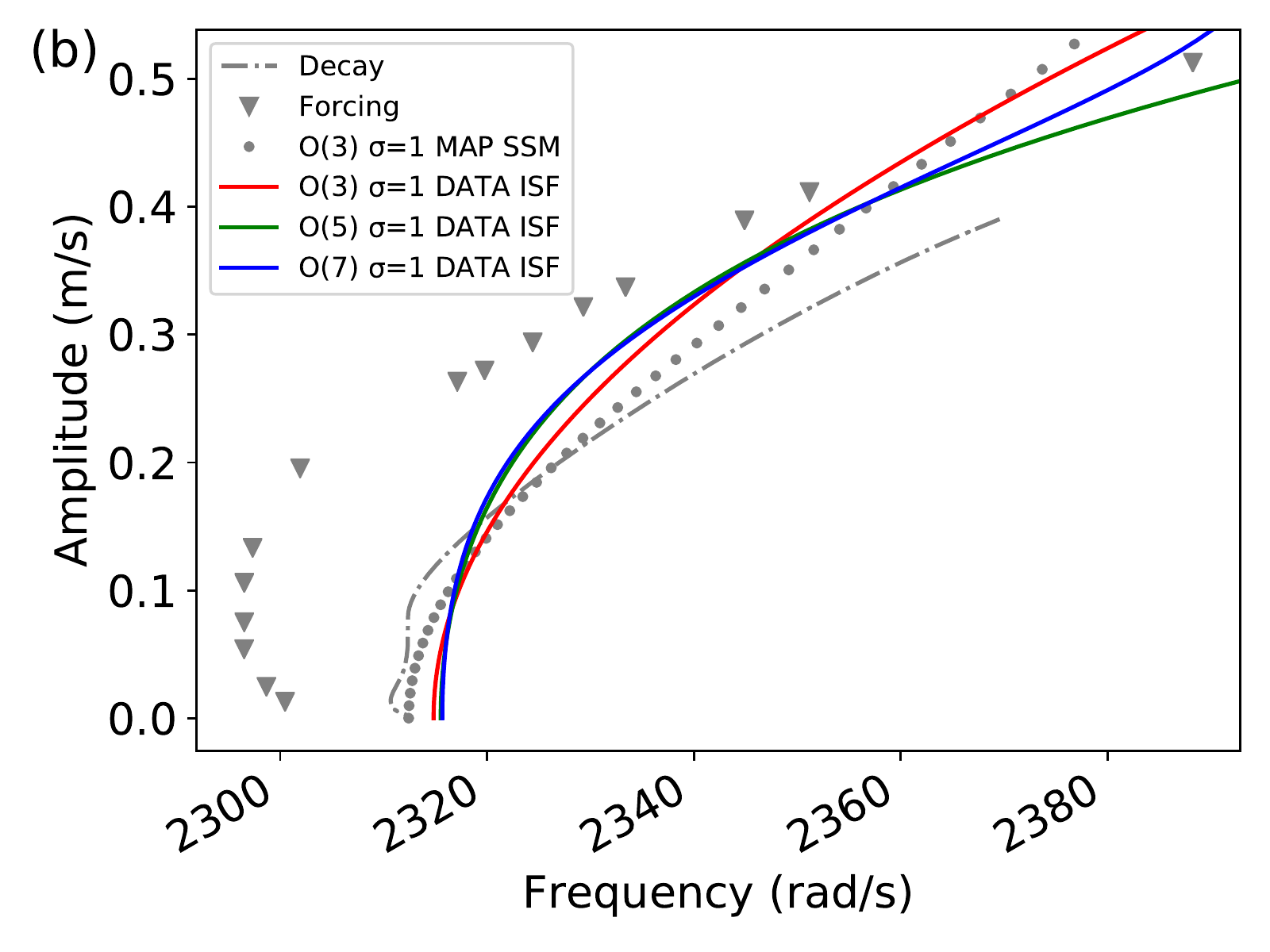}\\
\includegraphics[width=0.49\linewidth]{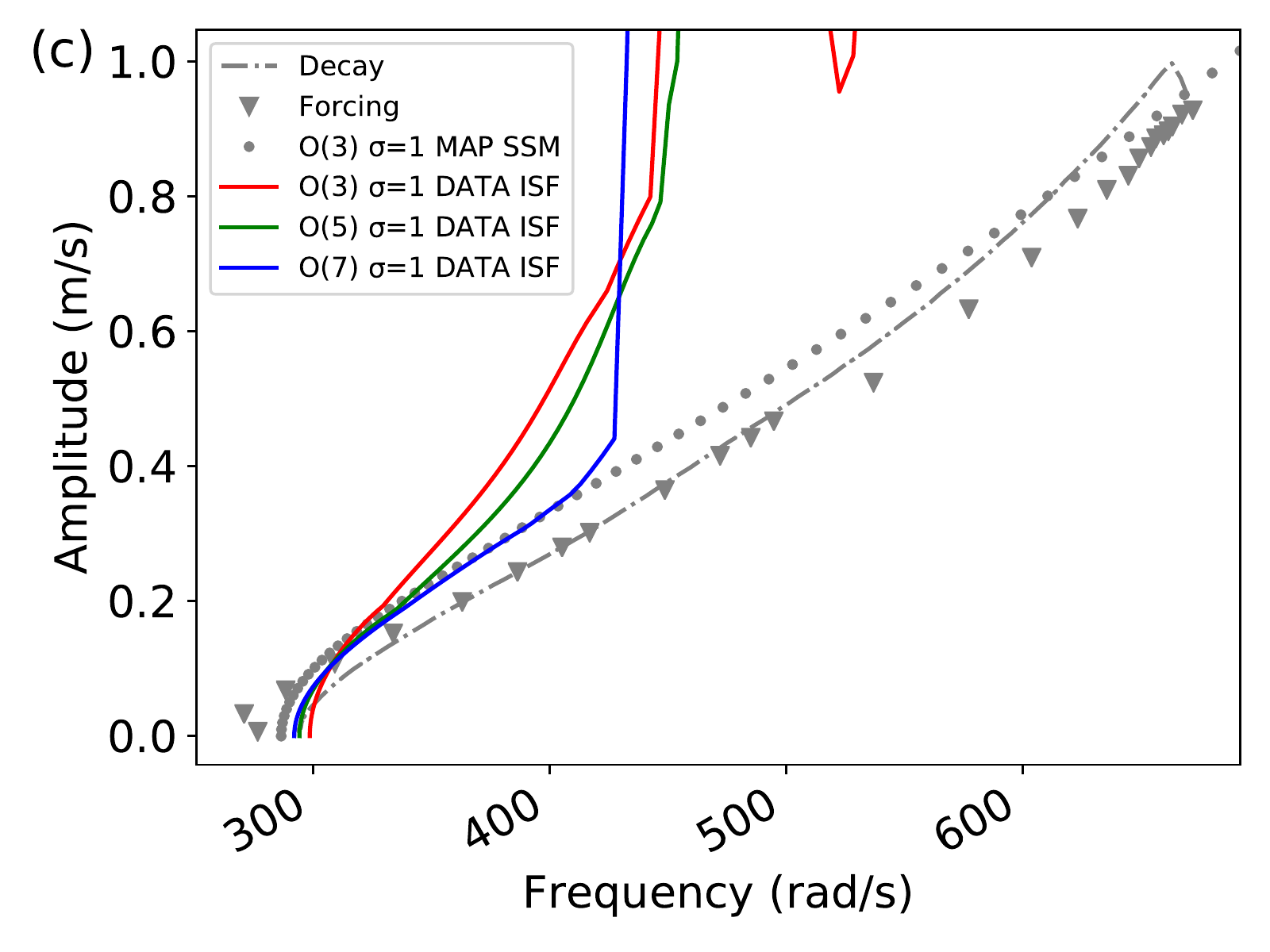}\includegraphics[width=0.49\linewidth]{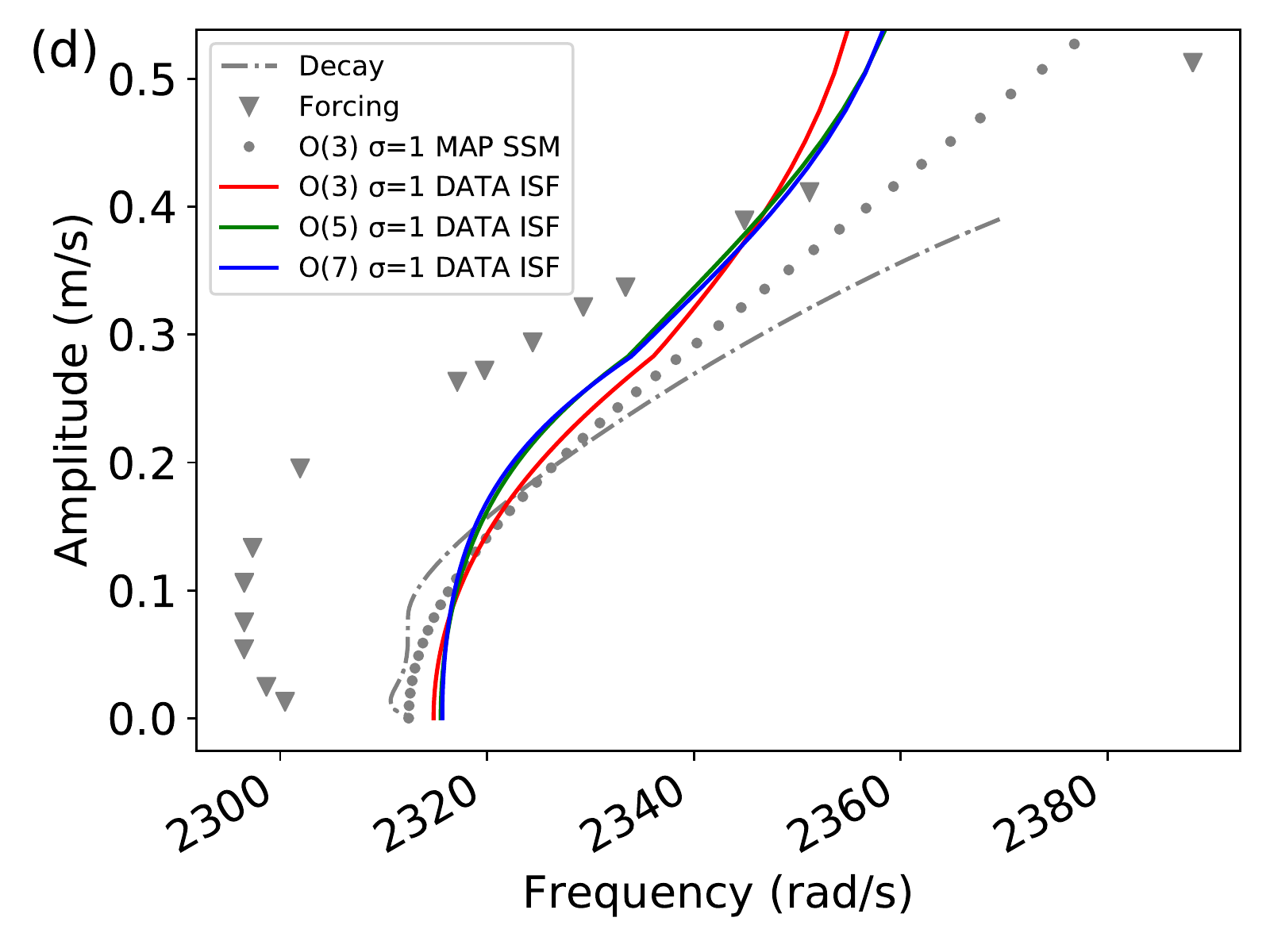}\\
\includegraphics[width=0.49\linewidth]{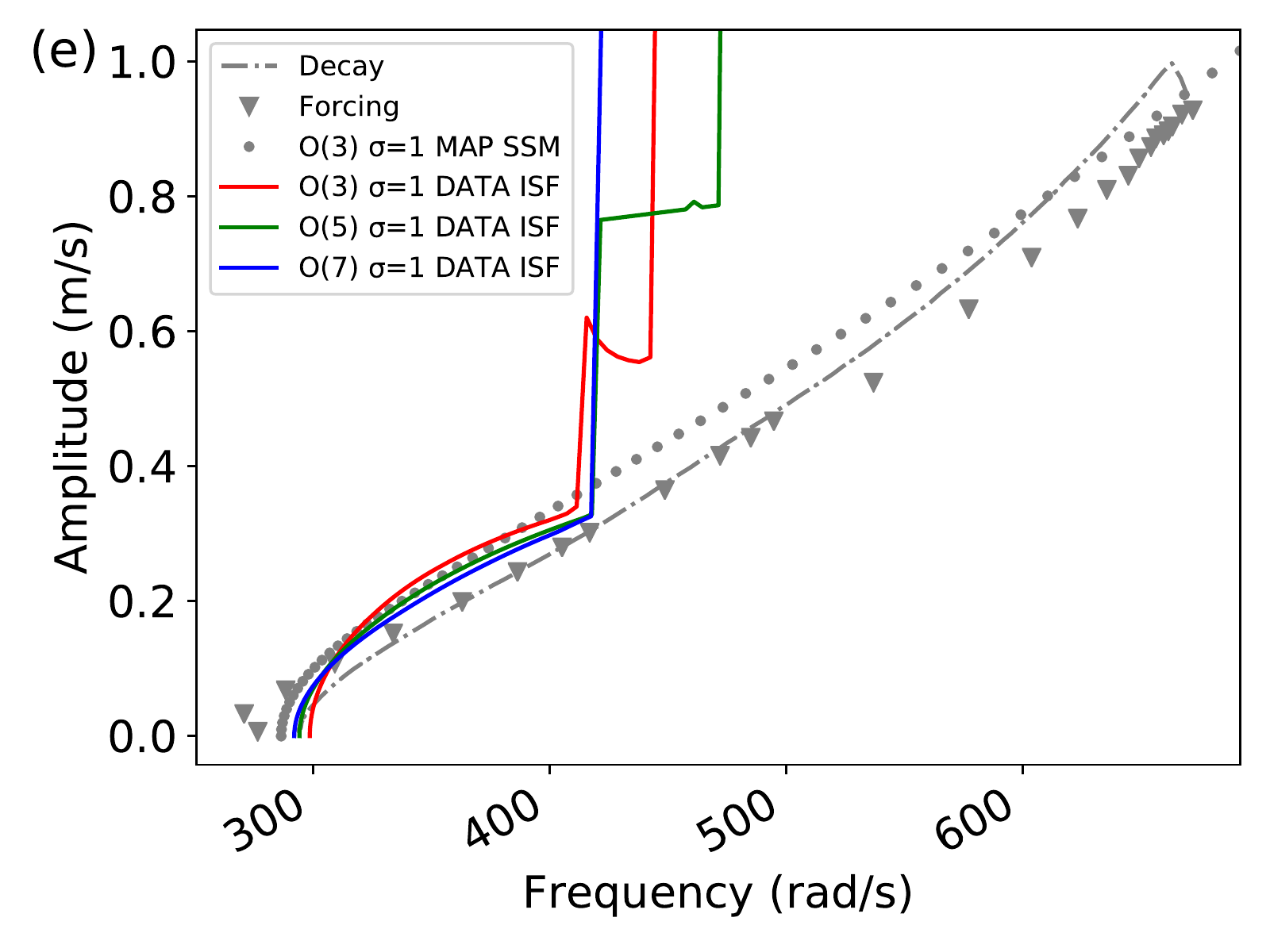}\includegraphics[width=0.49\linewidth]{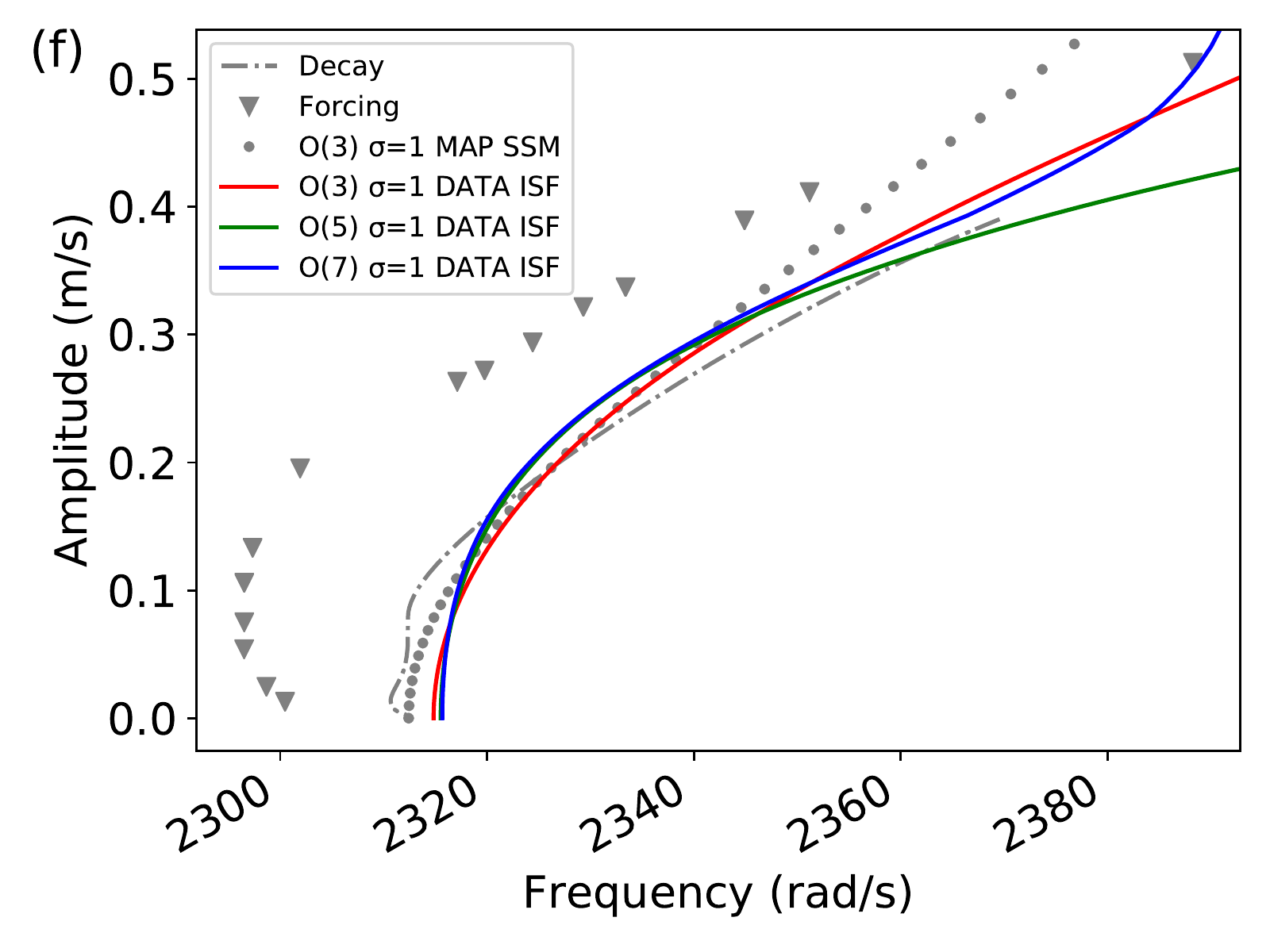}
\par\end{centering}
\caption{\label{fig:CCbeamBackbone}Backbone curves for the clamped-clamped
beam. Definition (\ref{eq:ISFbackbone}) and $\boldsymbol{W}_{\boldsymbol{z}}$
calculated by the polynomial iteration (\ref{eq:LeafGraphIteration})
were used in (a,b). Definition (\ref{eq:ISFbackbone}) and $\boldsymbol{W}_{\boldsymbol{z}}$
calculated by Newton's method from (\ref{eq:LeafGraphPreIt}) were
used in (c,d). Definition (\ref{eq:SSMbackbone}) was used in (e,f)
with $\boldsymbol{W}^{j}$ calculated using Newton's method. MAP means
a polynomial model fit was carried out first, DATA means that the
ISF was directly fitted to the data, $O(\alpha)$ indicates that order
$\alpha$ polynomials were used.}

\end{figure}

\section{Discussion and conclusions}

The paper has introduced invariant spectral foliations (ISF) as a
tool to derive reduced order models (ROM) of dynamic systems about
equilibria. The major advantage of this approach over other methods
is that the full dynamics can be reconstructed from a set of ISFs.
ISFs can also be fitted to data directly as opposed to SSMs. Direct
fitting ensures that the invariance equation is satisfied for the
data points with maximum accuracy, without using intermediate representations,
such as a black-box model. We have shown that the indirect fitting
of ISFs can result in high residuals compared to direct fitting. The
major disadvantage of an ISF is that is requires a submersion (function)
that depends on the same number of parameters as the dimensionality
of the phase-space. Therefore for high-dimensional problems a polynomial
representation will not be suitable, because the number of parameters
required to represent the ISF will be in the order of $\nu\binom{n+\alpha}{n}$,
where $n$ is the dimension of the phase space, $\nu$ is the co-dimension
of the ISF and $\alpha$ is the order of the polynomial. SSMs in contrast
are represented by immersions that depend on small number of parameters
even though they map into high dimensional spaces, so the required
number of parameters are about $n\binom{\nu+\alpha}{\nu}$, where
$\nu$ is the dimension of the SSM. Therefore SSMs can be efficiently
represented. However finding SSMs requires a model, be it black-box
or physical, that might also be difficult to represent efficiently.
In essence, the problem of dimensionality occurs at different levels
of representations with ISFs and SSMs. One promising approach to represent
submersions with minimum number of parameters is to use deep neural
networks \cite{Blcskei2017OptimalAW,grohs2019deep}, or other kinds
of nonlinear approximation methods \cite{devore1998}, that allow
to represent high dimensional functions with reasonable efficiency
as opposed to polynomials. The challenge with nonlinear approximations,
in particular with neural networks is that they can be difficult to
fit to data, because the distance between parameters that provide
small improvements in accuracy can be large and therefore not easy
to find \cite{petersen2018topological}. Nevertheless, deep neural
networks have enabled great advances in many fields of engineering
and therefore this approach will be explored elsewhere.

One important aspect of any calculation or prediction is whether it
is repeatable. In particular, the mathematically defined object should
be the same regardless of what numerical method is used to calculate
it. This aspect is determined by the uniqueness properties of the
mathematical object one wants to calculate. One particularly desirable
feature of an ISF is that its representation only needs to be once
differentiable when it is calculated for the slowest dynamics (as
in section \ref{subsec:UniqueExample}) or twice differentiable as
per theorem \ref{thm:MapFoliation} to be unique. This is in contrast
with SSMs, where the SSM containing the slowest dynamics must be many
times differentiable (as given by the SSM spectral quotient). The
required order of smoothness gets higher if there is a time-scale
separation with an increasingly fast dynamics in the system. As smoothness
is difficult to quantify numerically, calculating unique SSMs can
be a challenge. In this aspect, ISFs offer a theoretical advantage
over SSMs, which needs to be verified in practical examples.

The practical aspects of finding an ISF remains to be investigated.
One question is what type of data is required to obtain an accurate
ISF. To answer this question a rigorous statistical analysis on how
the accuracy of the ISF depends on the amount, the type and the uncertainty
of data is necessary. For mechanical systems impact hammer tests seem
to be a simple way to obtain free-decay vibration data. However such
data might need to be pre-processed in order to achieve a certain
distribution of data points within the phase space. In some cases
it might not be possible to obtain data from certain parts of the
phase space, similar to the clamped-clamped beam example, hence the
effect of missing data also needs to be explored.

\appendix

\section{\label{sec:ProofThm1}Proof of theorem \ref{thm:MapFoliation}}

\subsection{\label{subsec:MAPexpansion}Polynomial expansion}

In this section we find an approximate solution to the invariance
equation (\ref{eq:adjInvariance}) in the form of a power series.
Here we employ a complexification \cite{Arnold1992ordinary} of the
vector space $\mathbb{R}^{n}$ which is isomorphic to $\mathbb{C}^{n}$
and therefore the complexified map $\boldsymbol{F}$ may not be defined
on the whole of $\mathbb{C}^{n}$, because $\boldsymbol{F}$ is not
an entire function. We now apply a linear transformation $\boldsymbol{T}$
to the map $\boldsymbol{F}$, such that its Jacobian about the equilibrium
becomes a diagonal matrix and the map assumes the form
\begin{equation}
\boldsymbol{F}\left(\boldsymbol{x}\right)=\boldsymbol{A}\boldsymbol{x}+\begin{pmatrix}\begin{array}{l}
\boldsymbol{N}_{1}\left(\boldsymbol{x}\right)\\
\boldsymbol{N}_{2}\left(\boldsymbol{x}\right)
\end{array}\end{pmatrix},\label{eq:SplitMap}
\end{equation}
where,
\[
\boldsymbol{A}=\begin{pmatrix}\boldsymbol{A}_{1} & \boldsymbol{0}\\
\boldsymbol{0} & \boldsymbol{A}_{2}
\end{pmatrix},\;\text{\ensuremath{\boldsymbol{A}_{1}}}=\begin{pmatrix}\mu_{1}\\
 & \ddots\\
 &  & \mu_{\nu}
\end{pmatrix},\;\text{\ensuremath{\boldsymbol{A}_{2}}}=\begin{pmatrix}\mu_{\nu+1}\\
 & \ddots\\
 &  & \mu_{n}
\end{pmatrix}.
\]
If there is a complex conjugate pair of eigenvalues $\mu_{k}=\overline{\mu}_{k+1}$
the corresponding components of the state variable must also be complex
conjugate $x_{k}=\overline{x}_{k+1}$. Similarly, if $\mu_{k}$ is
real, $x_{k}$ must also be real. 

To arrive at an approximate solution of the invariance equation (\ref{eq:adjInvariance}),
we represent the unknowns as power series in the form of
\begin{align*}
\boldsymbol{S}\left(\boldsymbol{z}\right) & =\sum_{0<\left|\boldsymbol{m}\right|<\sigma}\boldsymbol{S}^{\boldsymbol{m}}\boldsymbol{z}^{\boldsymbol{m}}+\mathcal{O}\left(\left|\boldsymbol{z}\right|^{\sigma+1}\right),\\
\boldsymbol{U}\left(\boldsymbol{x}\right) & =\sum_{0<\left|\boldsymbol{m}\right|<\sigma}\boldsymbol{U}^{\boldsymbol{m}}\boldsymbol{x}^{\boldsymbol{m}}+\mathcal{O}\left(\left|\boldsymbol{x}\right|^{\sigma+1}\right),
\end{align*}
where the powers are interpreted as
\[
\boldsymbol{x}^{\boldsymbol{m}}=x_{1}^{m_{1}}\cdots x_{n}^{m_{n}},\qquad\boldsymbol{z}^{\boldsymbol{m}}=z_{1}^{m_{1}}\cdots z_{\nu}^{m_{\nu}},
\]
with $\boldsymbol{m}\in\mathbb{N}^{n}$ or $\boldsymbol{m}\in\mathbb{N}^{\nu}$,
respectively. The notation used here is explained in section \ref{subsec:DATAexpansionOpt}.
The equation for the linear terms in the invariance equation becomes
\[
\sum_{k=1}^{n}U_{j}^{\boldsymbol{e}_{k}}A_{k}^{\boldsymbol{e}_{l}}=\sum_{k=1}^{\nu}S_{j}^{\boldsymbol{e}_{k}}U_{k}^{\boldsymbol{e}_{l}},
\]
which does not have a unique solution, however we choose $U_{j}^{\boldsymbol{e}_{k}}=\delta_{jk}$,
$S_{j}^{\boldsymbol{e}_{k}}=\mu_{k}\delta_{jk}$. This is a normalising
constraint that will be taken into account when we prove the uniqueness
of the foliation $\mathcal{F}$. The equations for the order $\left|\boldsymbol{m}\right|$
terms in the invariance equation are written as
\begin{align}
U_{j}^{\boldsymbol{m}}\left(\prod_{k=1}^{\nu}\mu_{k}^{m_{k}}\right)\boldsymbol{x}^{\boldsymbol{m}} & =\mu_{j}U_{j}^{\boldsymbol{m}}\boldsymbol{x}^{\boldsymbol{m}}+S_{j}^{\boldsymbol{m}}\boldsymbol{x}^{\boldsymbol{m}}+H_{j}^{\boldsymbol{m}}\boldsymbol{x}^{\boldsymbol{m}}, & \negthickspace\negthickspace\negthickspace\negthickspace\negthickspace\negthickspace\negthickspace\negthickspace\negthickspace\negthickspace m_{\nu+1}=\cdots=m_{n}=0,\label{eq:homologyInternal}\\
U_{j}^{\boldsymbol{m}}\left(\prod_{k=1}^{n}\mu_{k}^{m_{k}}\right)\boldsymbol{x}^{\boldsymbol{m}} & =\mu_{j}U_{j}^{\boldsymbol{m}}\boldsymbol{x}^{\boldsymbol{m}}+H_{j}^{\boldsymbol{m}}\boldsymbol{x}^{\boldsymbol{m}}, & \negthickspace\negthickspace\negthickspace\negthickspace\negthickspace\negthickspace\negthickspace\negthickspace\negthickspace\negthickspace\exists l\in\left\{ \nu+1,\ldots,n\right\} :m_{l}\neq0,\label{eq:homologyExternal}
\end{align}
where $H_{j}^{\boldsymbol{m}}$ are the terms which are composed of
lower order terms of $\boldsymbol{U}$ and $\boldsymbol{S}$ and known
$\left|\boldsymbol{m}\right|$-th order terms of $\boldsymbol{F}$.
The equations are written for two different kinds of exponents. Equation
(\ref{eq:homologyInternal}) is for exponents that exist for both
$\boldsymbol{U}$ and $\boldsymbol{S}$ and therefore part of the
conjugate dynamics, equation (\ref{eq:homologyExternal}) is for exponents
that only identify terms in $\boldsymbol{U}$ and they correspond
to dynamics that occurs inside the leaves. Equations (\ref{eq:homologyInternal})
and (\ref{eq:homologyExternal}) are solved recursively starting with
$\left|\boldsymbol{m}\right|=2$ and then in increasing order for
$\left|\boldsymbol{m}\right|>2$.

Equation (\ref{eq:homologyInternal}) can be solved under any circumstances,
but there are multiple solutions. The terms $U_{j}^{\boldsymbol{m}}$
and $S_{j}^{\boldsymbol{m}}$ can be chosen relative to each other.
If there is an internal resonance or near internal resonance, that
is $\prod_{k=1}^{\nu}\mu_{k}^{m_{k}}\approx\mu_{j}$ we can choose
the solution 
\begin{equation}
U_{j}^{\boldsymbol{m}}=0,\quad S_{j}^{\boldsymbol{m}}=H_{j}^{\boldsymbol{m}},\label{eq:YesInternalResonance}
\end{equation}
otherwise we can also choose
\begin{equation}
U_{j}^{\boldsymbol{m}}=\frac{1}{\prod_{k=1}^{\nu}\mu_{k}^{m_{k}}-\mu_{j}}H_{j}^{\boldsymbol{m}},\quad S_{j}^{\boldsymbol{m}}=0\label{eq:NoInternalResonance}
\end{equation}
or some other combination of $U_{j}^{\boldsymbol{m}}$, $S_{j}^{\boldsymbol{m}}$.
The choice made here is another normalising condition and as we see
it will not affect the uniqueness of the foliation. Equation (\ref{eq:homologyExternal})
has a unique solution if $\prod_{k=1}^{n}\mu_{k}^{m_{k}}\neq\mu_{j}$,
which is
\begin{equation}
U_{j}^{\boldsymbol{m}}=\frac{1}{\prod_{k=1}^{n}\mu_{k}^{m_{k}}-\mu_{j}}H_{j}^{\boldsymbol{m}},\label{eq:NoExternalResonance}
\end{equation}
otherwise no solution exists, unless $H_{j}^{\boldsymbol{m}}$ vanishes,
which is unlikely.

Now we examine what is the order of expansion after which no resonances
are possible. All resonances are excluded when
\begin{equation}
\prod_{k=1}^{n}\mu_{k}^{m_{k}}\neq\mu_{j},\;j=1\ldots\nu.\label{eq:NoExtresAll}
\end{equation}
We describe two cases when (\ref{eq:NoExtresAll}) is satisfied. The
first is given by
\begin{align}
\left|\prod_{k=1}^{n}\mu_{k}^{m_{k}}\right| & <\left|\mu_{j}\right|,\label{eq:NoExtresLess}\\
\left|\boldsymbol{m}\right|\max_{k=1\ldots n}\log\left|\mu_{k}\right| & <\min_{j=1\ldots\nu}\log\left|\mu_{j}\right|,\nonumber 
\end{align}
which holds when $\max_{k=1\ldots n}\log\left|\mu_{k}\right|<0$ and
$\left|\boldsymbol{m}\right|>\beth_{E}$. Note that $\max_{k=1\ldots n}\log\left|\mu_{k}\right|=0$
implies that $\min_{j=1\ldots\nu}\log\left|\mu_{j}\right|>0$, which
are mutually exclusive conditions. Similarly, assuming $\max_{k=1\ldots n}\log\left|\mu_{k}\right|>0$
yields that $\left|\boldsymbol{m}\right|<\beth_{E}$, which is an
upper bound and therefore not applicable. The second case of (\ref{eq:NoExtresAll})
occurs when 
\begin{align*}
\left|\prod_{k=1}^{n}\mu_{k}^{m_{k}}\right| & >\left|\mu_{j}\right|,\\
\sum_{k=1\ldots n}m_{k}\log\left|\mu_{k}\right| & >\max_{j=1\ldots\nu}\log\left|\mu_{j}\right|,
\end{align*}
which yields a meaningful condition if $\min_{k=1\ldots n}\log\left|\mu_{k}\right|>0$,
that is 
\begin{equation}
\left|\boldsymbol{m}\right|>\frac{\max_{j=1\ldots\nu}\log\left|\mu_{j}\right|}{\min_{k=1\ldots n}\log\left|\mu_{k}\right|}.\label{eq:NoExtresGreater}
\end{equation}
However case (\ref{eq:NoExtresGreater}) is equivalent to (\ref{eq:NoExtresLess})
when the inverse $\boldsymbol{F}^{-1}$ is considered. 

In summary, we do not need to consider resonances for $\left|\boldsymbol{m}\right|>\beth_{E}$
when $\max_{k=1\ldots n}\log\left|\mu_{k}\right|<0$. Therefore we
choose the smallest $\sigma\in\mathbb{N}^{+}$, such that $\beth_{E}<\sigma$
and denote the truncated series expanded solution of the invariance
equation by $\boldsymbol{U}^{\le}$ and $\boldsymbol{S}^{\le}$.

\subsection{\label{subsec:ContractionMapping}Contraction mapping}

Here we show that once an order $\sigma-1$ asymptotic solution of
the invariance equation is found, then there is a unique $C^{\sigma}$
correction of the asymptotic solution $\boldsymbol{U}^{\le}$ and
$\boldsymbol{S}^{\le}$, so that the invariance equation (\ref{eq:adjInvariance})
is exactly satisfied. For the following argument we re-scale the map
$\boldsymbol{F}$, such that $\boldsymbol{F}_{\gamma}\left(\boldsymbol{x}\right)=\gamma^{-1}\boldsymbol{F}_{\gamma}\left(\gamma\boldsymbol{x}\right)$,
where $\gamma>0$. We have now solved the invariance equation up to
order $\sigma-1$ and therefore the invariance equation (\ref{eq:adjInvariance})
has an order $\sigma$ residual when the approximation is substituted.

Let us now fix $\boldsymbol{S}=\boldsymbol{S}^{\le}$ and decompose
the exact solution into $\boldsymbol{U}=\boldsymbol{U}^{\le}+\boldsymbol{U}^{>}$,
where $\boldsymbol{U}^{>}$ is a $C^{\sigma}$ function. To obtain
an iterative solution, we apply the inverse $\boldsymbol{S}^{-1}$
to the invariance equation (\ref{eq:adjInvariance}) and find that
\begin{equation}
\boldsymbol{U}^{>}\left(\boldsymbol{x}\right)=\boldsymbol{S}^{-1}\left(\boldsymbol{U}^{\le}\left(\boldsymbol{F}_{\gamma}\left(\boldsymbol{x}\right)\right)+\boldsymbol{U}^{>}\left(\boldsymbol{F}_{\gamma}\left(\boldsymbol{x}\right)\right)\right)-\boldsymbol{U}^{\le}\left(\boldsymbol{x}\right).\label{eq:FixedPointEquation}
\end{equation}
Equation (\ref{eq:FixedPointEquation}) can be re-cast as a fixed
point iteration using a nonlinear operator. This nonlinear operator
is defined as

\begin{equation}
\left[\mathcal{T}\left(\boldsymbol{U}^{>}\right)\right]\left(\boldsymbol{x}\right)=\boldsymbol{S}^{-1}\left(\boldsymbol{U}^{\le}\left(\boldsymbol{F}_{\gamma}\left(\boldsymbol{x}\right)\right)+\boldsymbol{U}^{>}\left(\boldsymbol{F}_{\gamma}\left(\boldsymbol{x}\right)\right)\right)-\boldsymbol{U}^{\le}\left(\boldsymbol{x}\right).\label{eq:FixedPointOperator}
\end{equation}
We define operator $\mathcal{T}$ on the space 
\[
\boldsymbol{X}^{\sigma}=\left\{ \boldsymbol{U}\in C^{\sigma}\left(B^{n},\mathbb{R}^{\nu}\right):D^{k}\boldsymbol{U}\left(\boldsymbol{0}\right)=\boldsymbol{0},k=0,1,\ldots,\sigma-1\right\} ,
\]
where $B^{n}=\left\{ \boldsymbol{x}\in\mathbb{R}^{n}:\left|\boldsymbol{x}\right|\le1\right\} $
is the closed unit ball. The norm on $\boldsymbol{X}^{\sigma}$ is
chosen to be
\begin{equation}
\left\Vert \boldsymbol{U}\right\Vert _{\sigma}=\sup_{\left|\boldsymbol{x}\right|\le1}\left|\boldsymbol{x}\right|^{-\sigma}\left|\boldsymbol{U}\left(\boldsymbol{x}\right)\right|,\label{eq:SigmaNorm}
\end{equation}
which makes $\boldsymbol{X}^{\sigma}$ a Banach space.

Next, we show that operator $\mathcal{T}$ is a contraction. In particular,
we show that there exists a constant $L<1$, such that 
\begin{equation}
\left\Vert \mathcal{T}\left(\boldsymbol{U}_{2}\right)-\mathcal{T}\left(\boldsymbol{U}_{1}\right)\right\Vert _{\sigma}\le L\left\Vert \boldsymbol{U}_{2}-\boldsymbol{U}_{1}\right\Vert _{\sigma}\label{eq:TLipschitz}
\end{equation}
and that $\left\Vert \mathcal{T}\left(\boldsymbol{U}\right)\right\Vert _{\sigma}\le1$
for all $\left\Vert \boldsymbol{U}\right\Vert _{\sigma}\le1$. The
latter criteria can be demonstrated through the Lipschitz condition
(\ref{eq:TLipschitz}) using the estimate
\begin{align*}
\left\Vert \mathcal{T}\left(\boldsymbol{U}\right)\right\Vert _{\sigma}=\left\Vert \mathcal{T}\left(\boldsymbol{U}\right)-\mathcal{T}\left(\boldsymbol{0}\right)+\mathcal{T}\left(\boldsymbol{0}\right)\right\Vert _{\sigma} & \le\left\Vert \mathcal{T}\left(\boldsymbol{U}\right)-\mathcal{T}\left(\boldsymbol{0}\right)\right\Vert _{\sigma}+\left\Vert \mathcal{T}\left(\boldsymbol{0}\right)\right\Vert _{\sigma}\\
 & \le L\left\Vert \boldsymbol{U}\right\Vert _{\sigma}+\left\Vert \mathcal{T}\left(\boldsymbol{0}\right)\right\Vert _{\sigma},
\end{align*}
which means that we need to have $\left\Vert \mathcal{T}\left(\boldsymbol{0}\right)\right\Vert _{\sigma}<1-L$
for $\mathcal{T}$ to be a contraction.

We start by estimating $\left\Vert \mathcal{T}\left(\boldsymbol{0}\right)\right\Vert _{\sigma}$.
Due to the polynomial approximation, we have 
\begin{align*}
\boldsymbol{S}^{-1}\left(\boldsymbol{U}^{\le}\left(\boldsymbol{F}\left(\boldsymbol{x}\right)\right)\right)-\boldsymbol{U}^{\le}\left(\boldsymbol{x}\right) & =\mathcal{O}\left(\boldsymbol{x}^{\sigma}\right),
\end{align*}
which implies that there exists $M>0$ such that
\[
\left|\boldsymbol{S}^{-1}\left(\boldsymbol{U}^{\le}\left(\boldsymbol{F}\left(\boldsymbol{x}\right)\right)\right)-\boldsymbol{U}^{\le}\left(\boldsymbol{x}\right)\right|\le M\left|\boldsymbol{x}\right|^{\sigma}.
\]
Using the scaled nonlinear map $\boldsymbol{F}_{\gamma}$ and the
scaled approximate solutions, the estimate now scales with $\gamma$
as
\begin{multline*}
\left|\boldsymbol{S}_{\gamma}^{-1}\left(\boldsymbol{U}_{\gamma}^{\le}\left(\boldsymbol{F}_{\gamma}\left(\boldsymbol{x}\right)\right)\right)-\boldsymbol{U}_{\gamma}^{\le}\left(\boldsymbol{x}\right)\right|\\
=\left|\gamma^{-1}\boldsymbol{S}^{-1}\left(\boldsymbol{U}^{\le}\left(\boldsymbol{F}\left(\gamma\boldsymbol{x}\right)\right)\right)-\gamma^{-1}\boldsymbol{U}^{\le}\left(\gamma\boldsymbol{x}\right)\right|\le M\gamma^{\sigma-1}\left|\boldsymbol{x}\right|^{\sigma}.
\end{multline*}
This result implies that $\left\Vert \mathcal{T}\left(\boldsymbol{0}\right)\right\Vert \le\gamma^{\sigma-1}M.$
Since $\sigma\ge2$, the norm $\left\Vert \mathcal{T}\left(\boldsymbol{0}\right)\right\Vert $
can be made arbitrarily small and therefore any Lipschitz constant
$0\le L<1$ is sufficient to demonstrate a unique fixed point of $\mathcal{T}$.

Next we need to show that there is a Lipschitz constant $0\le L<1$
in equation (\ref{eq:TLipschitz}). We use the fundamental theorem
of calculus to make a calculation similar to 
\[
f\left(x_{2}\right)-f\left(x_{1}\right)=\int_{0}^{1}f^{\prime}\left(x_{1}+s\left(x_{2}-x_{1}\right)\right)\mathrm{d}s\left(x_{2}-x_{1}\right).
\]
For our operator $\mathcal{T}$, we write that
\begin{multline*}
\boldsymbol{U}_{2}^{>}\left(\boldsymbol{x}\right)-\boldsymbol{U}_{1}^{>}\left(\boldsymbol{x}\right)=\\
=\int_{0}^{1}D\boldsymbol{S}^{-1}\left(\boldsymbol{U}^{\le}\left(\boldsymbol{F}_{\gamma}\left(\boldsymbol{x}\right)\right)+\boldsymbol{U}_{1}^{>}\left(\boldsymbol{F}_{\gamma}\left(\boldsymbol{x}\right)\right)+s\left(\boldsymbol{U}_{2}^{>}\left(\boldsymbol{F}_{\gamma}\left(\boldsymbol{x}\right)\right)-\boldsymbol{U}_{1}^{>}\left(\boldsymbol{F}_{\gamma}\left(\boldsymbol{x}\right)\right)\right)\right)\mathrm{d}s\times\\
\qquad\times\left(\boldsymbol{U}_{2}^{>}\left(\boldsymbol{F}_{\gamma}\left(\boldsymbol{x}\right)\right)-\boldsymbol{U}_{1}^{>}\left(\boldsymbol{F}_{\gamma}\left(\boldsymbol{x}\right)\right)\right).
\end{multline*}
Due to the scaling, $\boldsymbol{S}$ becomes linear as $\gamma\to0$,
hence we can find $\epsilon_{1}\left(\gamma\right)$ such that $\lim_{\gamma\to0}\epsilon_{1}\left(\gamma\right)=0$
and 
\[
\sup_{\left|\boldsymbol{z}\right|\le1}\left|D\boldsymbol{S}^{-1}\left(\boldsymbol{z}\right)\right|=\left(\min_{k=1\ldots\nu}\mu_{k}\right)^{-1}+\epsilon_{1}\left(\gamma\right).
\]
This implies the estimate
\[
\left|\boldsymbol{U}_{2}^{>}\left(\boldsymbol{x}\right)-\boldsymbol{U}_{1}^{>}\left(\boldsymbol{x}\right)\right|\le\left(\left(\min_{k=1\ldots\nu}\mu_{k}\right)^{-1}+\epsilon_{1}\left(\gamma\right)\right)\left|\boldsymbol{U}_{2}^{>}\left(\boldsymbol{F}_{\gamma}\left(\boldsymbol{x}\right)\right)-\boldsymbol{U}_{1}^{>}\left(\boldsymbol{F}_{\gamma}\left(\boldsymbol{x}\right)\right)\right|.
\]
In the next step, we are estimating the effect of the inner function
$\boldsymbol{F}_{\gamma}$ of $\boldsymbol{U}_{1,2}^{>}$ by way of
the $\sigma$-norm, that is
\begin{align*}
\left|\boldsymbol{x}\right|^{-\sigma}\left|\boldsymbol{U}_{2}^{>}\left(\boldsymbol{x}\right)-\boldsymbol{U}_{1}^{>}\left(\boldsymbol{x}\right)\right| & \le\left(\left(\min_{k=1\ldots\nu}\left|\mu_{k}\right|\right)^{-1}+\epsilon_{1}\left(\gamma\right)\right)\left|\boldsymbol{x}\right|^{-\sigma}\left|\boldsymbol{U}_{2}^{>}\left(\boldsymbol{F}_{\gamma}\left(\boldsymbol{x}\right)\right)-\boldsymbol{U}_{1}^{>}\left(\boldsymbol{F}_{\gamma}\left(\boldsymbol{x}\right)\right)\right|\\
 & \le\left(\left(\min_{k=1\ldots\nu}\left|\mu_{k}\right|\right)^{-1}+\epsilon_{1}\left(\gamma\right)\right)\left|\boldsymbol{F}_{\gamma}\left(\boldsymbol{x}\right)\right|^{\sigma}\left\Vert \boldsymbol{U}_{2}^{>}-\boldsymbol{U}_{1}^{>}\right\Vert _{\sigma}.
\end{align*}
Similar to the previous estimates, there exists $\epsilon_{2}\left(\gamma\right)$
such that $\lim_{\gamma\to0}\epsilon_{2}\left(\gamma\right)=0$ and
\[
\sup_{\left|\boldsymbol{x}\right|\le1}\left|\boldsymbol{F}_{\gamma}\left(\boldsymbol{x}\right)\right|=\max_{k=1\ldots n}\left|\mu_{k}\right|+\epsilon_{2}\left(\gamma\right).
\]
Putting together the previous estimates we find that 
\begin{equation}
\left\Vert \mathcal{T}\left(\boldsymbol{U}_{2}\right)-\mathcal{T}\left(\boldsymbol{U}_{1}\right)\right\Vert _{\sigma}\le\left(\left(\min_{k=1\ldots\nu}\left|\mu_{k}\right|\right)^{-1}+\epsilon_{1}\left(\gamma\right)\right)\left(\max_{k=1\ldots n}\left|\mu_{k}\right|+\epsilon_{2}\left(\gamma\right)\right)^{\sigma}\left\Vert \boldsymbol{U}_{2}^{>}-\boldsymbol{U}_{1}^{>}\right\Vert _{\sigma}.\label{eq:ContractionEstimate}
\end{equation}
Since $\epsilon_{1}\to0$ and $\epsilon_{2}\to0$ as $\gamma\to0$,
we only need to show that
\begin{equation}
\left(\max_{k=1\ldots n}\left|\mu_{k}\right|\right)^{\sigma}\left(\min_{k=1\ldots\nu}\left|\mu_{k}\right|\right)^{-1}<1,\label{eq:ContractionCrit}
\end{equation}
so that the Lipschitz constant in equation (\ref{eq:ContractionEstimate})
is less than one. After rearranging the criterion (\ref{eq:ContractionCrit})
we find that
\[
\sigma\log\max_{k=1\ldots n}\left|\mu_{k}\right|<\log\min_{k=1\ldots\nu}\left|\mu_{k}\right|.
\]
Therefore operator $\mathcal{T}$ is a contraction if there exists
$\sigma\ge2$ such that one of the following cases apply:
\begin{align}
\max_{k=1\ldots n}\left|\mu_{k}\right| & <1\quad\implies & \beth_{E} & <\sigma,\label{eq:ContractionLess}\\
\max_{k=1\ldots n}\left|\mu_{k}\right| & =1\quad\implies & \min_{k=1\ldots\nu}\left|\mu_{k}\right| & >1,\label{eq:ContractionEqual}\\
\max_{k=1\ldots n}\left|\mu_{k}\right| & >1\quad\implies & \beth_{E} & >\sigma,\label{eq:ContractionGreater}
\end{align}
where 
\[
\beth_{E}=\frac{\min_{k=1\ldots\nu}\log\left|\mu_{k}\right|}{\max_{k=1\ldots n}\log\left|\mu_{k}\right|}.
\]
It turns out that only case (\ref{eq:ContractionLess}) is possible,
because case (\ref{eq:ContractionEqual}) stipulates mutually exclusive
conditions and (\ref{eq:ContractionGreater}) would allow resonances.
In fact, when $\max_{k=1\ldots n}\left|\mu_{k}\right|>1$ holds one
needs to consider the inverse map $\boldsymbol{F}^{-1}$ instead,
when applying theorem \ref{thm:MapFoliation}.

\subsection{\label{subsec:Uniqueness}Uniqueness}

In section \ref{subsec:MAPexpansion} we have made some normalising
assumptions, which restricted the parametrisation of the foliation.
Here we show that those assumptions can be removed and that any sufficiently
smooth submersion $\boldsymbol{U}$ satisfying the invariance equation
(\ref{eq:adjInvariance}) can be transformed so that it satisfies
the normalising assumptions and yet represents the same invariant
foliation, which then implies uniqueness.

Let us write variable $\boldsymbol{x}$ as a tuple $\boldsymbol{x}=\left(\boldsymbol{x}_{1},\boldsymbol{x}_{2}\right)$
such that $\boldsymbol{x}_{1}\in\mathbb{C}^{\nu}$ and $\boldsymbol{x}_{2}\in\mathbb{C}^{n-\nu}$
with the restrictions due to complexification. When finding a solution
of (\ref{eq:adjInvariance}) as per the argument in sections \ref{subsec:MAPexpansion}
and \ref{subsec:ContractionMapping}, we can use normalising conditions
such that the solution of (\ref{eq:homologyInternal}) satisfies $U_{j}^{\boldsymbol{e}_{k}}=\delta_{jk}$
and $U_{j}^{\boldsymbol{m}}=0$ when $\left|\boldsymbol{m}\right|\ge2$
and $m_{\nu+1}=\cdots=m_{n}=0$. This then implies that $\boldsymbol{U}\left(\boldsymbol{x}_{1},\boldsymbol{0}\right)=\boldsymbol{x}_{1}+\mathcal{O}\left(\left|\boldsymbol{x}_{1}\right|^{\sigma}\right)$.
Let us now denote the unique solution of (\ref{eq:adjInvariance})
under these normalising conditions by $\boldsymbol{U}$ and $\boldsymbol{S}$
and another $\sigma$-times differentiable solution of (\ref{eq:adjInvariance})
and (\ref{eq:adjTangency}) by $\hat{\boldsymbol{U}}$ and $\hat{\boldsymbol{S}}$.
We now look for a transformation $\boldsymbol{\Phi}$, such that $\boldsymbol{U}\left(\boldsymbol{x}_{1},\boldsymbol{0}\right)=\boldsymbol{\Phi}\left(\hat{\boldsymbol{U}}\left(\boldsymbol{x}_{1},\boldsymbol{0}\right)\right)$.
To help the notation, we define the function $\boldsymbol{\Psi}:\mathbb{C}^{\nu}\to\mathbb{C}^{\nu}$,
$\boldsymbol{\Psi}\left(\boldsymbol{z}\right)=\hat{\boldsymbol{U}}\left(\boldsymbol{z},\boldsymbol{0}\right)$,
which is an invertible and $\sigma$-times differentiable function
in a sufficiently small neighbourhood of the origin, because its Jacobian
at the origin is invertible due to the tangency condition (\ref{eq:adjTangency}).
Using the inverse $\boldsymbol{\Psi}^{-1}$, we find that $\boldsymbol{\Phi}\left(\boldsymbol{z}\right)=\boldsymbol{U}\left(\boldsymbol{\Psi}^{-1}\left(\boldsymbol{z}\right),\boldsymbol{0}\right)$,
which is again an invertible $\sigma$-times differentiable transformation.
Let us now define 
\[
\tilde{\boldsymbol{U}}\left(\boldsymbol{x}_{1},\boldsymbol{x}_{2}\right)=\boldsymbol{\Phi}\left(\hat{\boldsymbol{U}}\left(\boldsymbol{x}_{1},\boldsymbol{x}_{2}\right)\right),\;\tilde{\boldsymbol{S}}=\boldsymbol{\Phi}\circ\hat{\boldsymbol{S}}\circ\boldsymbol{\Phi}^{-1},
\]
which satisfy the invariance equation (\ref{eq:adjInvariance}), the
tangency condition (\ref{eq:adjTangency}) and our normalising conditions.
However we have shown that there is a unique solution to (\ref{eq:adjInvariance})
and (\ref{eq:adjTangency}) under the normalising conditions, hence
$\tilde{\boldsymbol{U}}=\boldsymbol{U}$ and $\tilde{\boldsymbol{S}}=\boldsymbol{S}$.
This means that solutions of (\ref{eq:adjInvariance}) and (\ref{eq:adjTangency})
can be re-parametrised into each other, therefore the invariant foliation
$\mathcal{F}$ represented by the submersions $\boldsymbol{U}$ or
$\hat{\boldsymbol{U}}$ is unique in a sufficiently small neighbourhood
of the origin.

This conclude the proof of theorem \ref{thm:MapFoliation}.

\section{\label{sec:VFexpansion}Series expansion of ISFs for vector fields}

In section \ref{subsec:MAPexpansion} we have provided an algorithm
to find a power-series expansion of an ISF for a discrete map. Here
we modify the algorithm for vector fields. We assume a first order
differential equation $\dot{\boldsymbol{x}}=\boldsymbol{G}\left(\boldsymbol{x}\right)$,
whose vector field is transformed into the frame of the eigenvectors
of its Jacobian about the origin, such that 
\begin{equation}
\boldsymbol{G}\left(\boldsymbol{x}\right)=\boldsymbol{B}\boldsymbol{x}+\begin{pmatrix}\begin{array}{l}
\boldsymbol{N}_{1}\left(\boldsymbol{x}\right)\\
\boldsymbol{N}_{2}\left(\boldsymbol{x}\right)
\end{array}\end{pmatrix},\label{eq:SplitMap-1}
\end{equation}
where,
\[
\boldsymbol{B}=\begin{pmatrix}\boldsymbol{B}_{1} & \boldsymbol{0}\\
\boldsymbol{0} & \boldsymbol{B}_{2}
\end{pmatrix},\;\text{\ensuremath{\boldsymbol{B}_{1}}}=\begin{pmatrix}\lambda_{1}\\
 & \ddots\\
 &  & \lambda_{\nu}
\end{pmatrix},\;\text{\ensuremath{\boldsymbol{B}_{2}}}=\begin{pmatrix}\lambda_{\nu+1}\\
 & \ddots\\
 &  & \lambda_{n}
\end{pmatrix}.
\]

To arrive at an approximate solution of the invariance equation (\ref{eq:adjInvariance}),
we represent the unknowns as power series in the form of
\begin{align*}
\boldsymbol{S}\left(\boldsymbol{z}\right) & =\sum_{0<\left|\boldsymbol{m}\right|\le\alpha}\boldsymbol{S}^{\boldsymbol{m}}\boldsymbol{z}^{\boldsymbol{m}}+\mathcal{O}\left(\left|\boldsymbol{z}\right|^{\alpha+1}\right),\\
\boldsymbol{U}\left(\boldsymbol{x}\right) & =\sum_{0<\left|\boldsymbol{m}\right|\le\alpha}\boldsymbol{U}^{\boldsymbol{m}}\boldsymbol{x}^{\boldsymbol{m}}+\mathcal{O}\left(\left|\boldsymbol{x}\right|^{\alpha+1}\right),
\end{align*}
where the powers are interpreted as
\[
\boldsymbol{x}^{\boldsymbol{m}}=x_{1}^{m_{1}}\cdots x_{n}^{m_{n}},\qquad\boldsymbol{z}^{\boldsymbol{m}}=z_{1}^{m_{1}}\cdots z_{\nu}^{m_{\nu}},
\]
with $\boldsymbol{m}\in\mathbb{N}^{n}$ or $\boldsymbol{m}\in\mathbb{N}^{\nu}$,
respectively. Using the notation in section \ref{subsec:DATAexpansionOpt},
the equation for the linear terms in the invariance equation becomes
\[
\sum_{k=1}^{n}U_{j}^{\boldsymbol{e}_{k}}B_{k}^{\boldsymbol{e}_{l}}=\sum_{k=1}^{\nu}S_{j}^{\boldsymbol{e}_{k}}U_{k}^{\boldsymbol{e}_{l}},
\]
which does not have a unique solution, however we choose $U_{j}^{\boldsymbol{e}_{k}}=\delta_{jk}$,
$S_{j}^{\boldsymbol{e}_{k}}=\lambda_{k}\delta_{jk}$, where $\delta_{jk}$
is the Kronecker delta. The equations for the order $\left|\boldsymbol{m}\right|$
terms in the invariance equation are written as
\begin{align}
U_{j}^{\boldsymbol{m}}\left(\sum_{k=1}^{\nu}m_{k}\lambda_{k}\right)\boldsymbol{x}^{\boldsymbol{m}} & =\lambda_{j}U_{j}^{\boldsymbol{m}}\boldsymbol{x}^{\boldsymbol{m}}+S_{j}^{\boldsymbol{m}}\boldsymbol{x}^{\boldsymbol{m}}+H_{j}^{\boldsymbol{m}}\boldsymbol{x}^{\boldsymbol{m}}, & \negthickspace\negthickspace\negthickspace\negthickspace\negthickspace\negthickspace\negthickspace\negthickspace\negthickspace\negthickspace m_{\nu+1}=\cdots=m_{n}=0,\label{eq:homologyInternal-1}\\
U_{j}^{\boldsymbol{m}}\left(\sum_{k=1}^{n}m_{k}\lambda_{k}\right)\boldsymbol{x}^{\boldsymbol{m}} & =\lambda_{j}U_{j}^{\boldsymbol{m}}\boldsymbol{x}^{\boldsymbol{m}}+H_{j}^{\boldsymbol{m}}\boldsymbol{x}^{\boldsymbol{m}}, & \negthickspace\negthickspace\negthickspace\negthickspace\negthickspace\negthickspace\negthickspace\negthickspace\negthickspace\negthickspace\exists l\in\left\{ \nu+1,\ldots,n\right\} :m_{l}\neq0,\label{eq:homologyExternal-1}
\end{align}
where $H_{j}^{\boldsymbol{m}}$ are the terms which are composed of
lower order terms of $\boldsymbol{U}$ and $\boldsymbol{S}$ and known
$\left|\boldsymbol{m}\right|$-th order terms of $\boldsymbol{F}$.
The equations are written for two different kinds of exponents. Equation
(\ref{eq:homologyInternal-1}) is for exponents that exist for both
$\boldsymbol{U}$ and $\boldsymbol{S}$ and therefore part of the
conjugate dynamics; equation (\ref{eq:homologyExternal-1}) is for
exponents that only identify terms in $\boldsymbol{U}$ and they correspond
to dynamics that occurs insides the leaves. Equations are solved recursively
starting with $\left|\boldsymbol{m}\right|=2$ and then in increasing
order for $\left|\boldsymbol{m}\right|>2$.

Equation (\ref{eq:homologyInternal-1}) can be solved under any circumstances,
but it is clear that there are multiple solutions. The terms $U_{j}^{\boldsymbol{m}}$
and $S_{j}^{\boldsymbol{m}}$ can be chosen relative to each other.
If there is a resonance or near resonance, that is $\sum_{k=1}^{\nu}m_{k}\lambda_{k}\approx\lambda_{j}$
we can choose the solution 
\[
U_{j}^{\boldsymbol{m}}=0,\quad S_{j}^{\boldsymbol{m}}=H_{j}^{\boldsymbol{m}},
\]
otherwise we can also choose
\[
U_{j}^{\boldsymbol{m}}=\frac{1}{\sum_{k=1}^{\nu}m_{k}\lambda_{k}-\lambda_{j}}H_{j}^{\boldsymbol{m}},\quad S_{j}^{\boldsymbol{m}}=0
\]
or some other combination of $U_{j}^{\boldsymbol{m}}$, $S_{j}^{\boldsymbol{m}}$.
Equation (\ref{eq:homologyExternal-1}) has a unique solution if $\sum_{k=1}^{n}m_{k}\lambda_{k}\neq\lambda_{j}$,
which is
\[
U_{j}^{\boldsymbol{m}}=\frac{1}{\sum_{k=1}^{n}m_{k}\lambda_{k}-\lambda_{j}}H_{j}^{\boldsymbol{m}},
\]
otherwise no solution exists, unless $H_{j}^{\boldsymbol{m}}$ vanishes.
\begin{acknowledgements}
The author would like to thank George Haller, Thomas Breunung and
the anonymous reviewers for questions and comments that prompted the
author to make improvements to the manuscript. He would also like
to thank for the hospitality of the ETH Z\"urich, in particular the
Institute for Mechanical Systems that hosted the author from February-July
2020 during which time the manuscript was revised.

\medskip{}
\end{acknowledgements}

\noindent \textbf{Computer code:} The computer code that reproduces
the figures and calculations alongside the experimental data can be
found at \url{https://github.com/rs1909/ISFpaper}.

\medskip{}

\noindent \textbf{Funding:} The author is supported by a University
Research Fellowship awarded by the University of Bristol.\medskip{}

\noindent \textbf{Conflict of Interest:} The author declares that
he has no conflict of interest.

\noindent \medskip{}

\noindent \textbf{Preprint:} The present manuscript is available on
the arXiv preprint server \cite{szalai2019invariant}.

\bibliographystyle{plainurl}
\bibliography{AllRef}

\end{document}